\documentclass[]{interact}

\usepackage{epstopdf}% To incorporate .eps illustrations using PDFLaTeX, etc.
\usepackage{subfigure}% Support for small, `sub' figures and tables

\usepackage[numbers,sort&compress,merge]{natbib}% Citation support using natbib.sty
\usepackage{amsmath,amssymb}
\usepackage{hyperref}
\usepackage{mathrsfs}
\usepackage{amsfonts}
\usepackage{mathtools}
\usepackage{tikz}
\usepackage{lmodern}
\usepackage{bigints}
\usepackage{relsize}
\usepackage{bbm}
\usepackage{amsthm}
\bibpunct[, ]{(}{)}{,}{n}{,}{,}% Citation support using natbib.sty
\theoremstyle{plain}% Theorem-like structures
\newtheorem{theorem}{Theorem}[section]
\newtheorem{lemma}[theorem]{Lemma}
\newtheorem{corollary}[theorem]{Corollary}
\newtheorem{proposition}[theorem]{Proposition}

\theoremstyle{definition}
\newtheorem{definition}[theorem]{Definition}

\theoremstyle{remark}
\newtheorem{remark}{Remark}
\newtheorem{notation}{Notation}
\newcommand{\norm}[1]{\left\lVert#1\right\rVert}
\DeclareMathOperator{\sech}{sech}
\DeclareMathOperator{\val}{val}

\DeclareMathOperator{\I}{Im}
\begin{document}

%\articletype{ARTICLE TEMPLATE}

\title{Approximate kink-kink solutions for the $\phi^{6}$ model
in the low-speed limit}

\author{
\name{Abdon Moutinho \textsuperscript{a}\thanks{CONTACT Abdon Moutinho Email: moutinho@math.univ-paris13.fr}}
\affil{\textsuperscript{a} LAGA, Université Sorbonne Paris Nord, Villetaneuse, France}
}

\maketitle

\begin{abstract}
This paper is the first part of a series of two papers that study the problem of elasticity and stability of the collision of two kinks with low speed $v$ for the  nonlinear wave equation known as the $\phi^{6}$ model in dimension $1+1$.  In this paper, we construct a sequence of approximate solutions $(\phi_{k}(v,t,x))_{k\in\mathbb{N}_{\geq 2}}$ for this nonlinear wave equation such that each function $\phi_{k}(v,t,x)$ converges in the energy norm to the traveling kink-kink with speed $v$ when $t$ goes to $+\infty.$ The methods used in this paper are not restricted only to the $\phi^{6}$ model. 
\end{abstract}


\section{Introduction}
\par We consider for the potential function $U(\phi)=\phi^{2}(1-\phi^{2})^{2}$ the partial differential equation 
\begin{equation}\label{nlww}
    \partial_{t}^{2}\phi(t,x)-\partial_{x}^{2}\phi(t,x)+ U^{'}(\phi(t,x))=0, \quad (t,x) \in \mathbb{R}\times \mathbb{R},
\end{equation}
 which is known as the $\phi^{6}$ model in the theory of scalar fields.
 The energy associated with \eqref{nlww} is given by 
 \begin{equation}\label{energy}\tag{Energy}
    E(\phi)(t)=\int_{\mathbb{R}}\frac{\partial_{t}\phi(t,x)^{2}}{2}+\frac{\partial_{x}\phi(t,x)^{2}}{2}+U(\phi(t,x))\,dx,
\end{equation}
and the momentum associated to \eqref{nlww} is 
\begin{equation}\label{momentum}\tag{Momentum}
    P(\phi)(t)={-}\int_{\mathbb{R}}\partial_{t}\phi(t,x)\partial_{x}\phi(t,x)\,dx,
\end{equation}
we consider the potential energy 
\begin{equation*}
    E_{pot}(\phi)(t)=\int_{\mathbb{R}}\frac{\partial_{x}\phi(t,x)^{2}}{2}+U(\phi(t,x))\,dx.
\end{equation*}
If the solution of the integral equation associated to the partial differential equation \eqref{nlww} has finite energy, the quantities \eqref{energy} and \eqref{momentum} are preserved for all $t\in\mathbb{R}.$ 
\par The space of solutions of \eqref{nlww} with finite energy is invariant under space translation, time translation and space reflection. Also, if $H$ is a stationary solution of \eqref{nlww}, then, for any $-1<v<1,$ the Lorentz transformation of $H$ given by 
\begin{equation}\label{lorentztransform}
   \phi(t,x)= H\left(\frac{x-vt}{\sqrt{1-v^{2}}}\right)
\end{equation}
is also a solution of \eqref{nlww}. 
\par The only non-constant stationary solutions  of \eqref{nlww} with finite energy are the topological solitons denominated kinks and anti-kinks.
The kinks of \eqref{nlww} are the space translation of 
\begin{equation}\label{kinkformula}
    H_{0,1}(x)=\frac{e^{\sqrt{2}x}}{\sqrt{1+e^{2\sqrt{2}x}}},\quad H_{-1,0}(x)=-H_{0,1}(-x)=-\frac{e^{-\sqrt{2}x}}{\sqrt{1+e^{-2\sqrt{2}x}}},
\end{equation}
and the anti-kinks are the space reflection around $0$ of the kinks. The kinks and anti-kinks are the critical points of the potential energy function. Clearly,
$ H_{0,1}(x)< e^{\sqrt{2}\min(x,0)}$ for all $x\in\mathbb{R}.$ Moreover, we also have the Bogomolny identity, $ H^{'}_{0,1}(x)=\sqrt{2U(H_{0,1}(x))},$
from which with the formula \eqref{kinkformula} of $H_{0,1}$ we can verify by induction for any $k\geq 1$ the existence of a constant $C(k)>0$ such that
\begin{equation}\label{le2}
    \left\vert \frac{d^{k}}{dx^{k}}H_{0,1}(x) \right\vert\leq  C(k)\min\left(e^{\sqrt{2}x},e^{-2\sqrt{2}x}\right) \text{for all $x\in\mathbb{R}.$}
\end{equation}
Finally, since $ H^{'}_{0,1}(x)=\sqrt{2}\frac{e^{\sqrt{2}x}}{\left(1+e^{2\sqrt{2}x}\right)^{\frac{3}{2}}},$ we have that
$
    \norm{ H^{'}_{0,1}}_{L^{2}_{x}(\mathbb{R})}^{2}=\frac{1}{2\sqrt{2}}.
$
\par In \cite{multison}, it was obtained for any $-1<v<1$ the existence of a solution $\phi(t,x) $ of \eqref{nlww} satisfying 
\begin{align}\label{rere1}
    \lim_{t\to+\infty}\norm{\phi(t,x)-H_{0,1}\left(\frac{x-vt}{\sqrt{1-v^{2}}}\right)-H_{-1,0}\left(\frac{x+vt}{\sqrt{1-v^{2}}}\right)}_{H^{1}_{x}(\mathbb{R})}=0,\\ \label{rere2}
    \lim_{t\to+\infty}\norm{\partial_{t}\phi(t,x)+\frac{v}{\sqrt{1-v^{2}}} H^{'}_{0,1}\left(\frac{x-vt}{\sqrt{1-v^{2}}}\right)-\frac{v}{\sqrt{1-v^{2}}} H^{'}_{-1,0}\left(\frac{x+vt}{\sqrt{1-v^{2}}}\right)}_{L^{2}_{x}(\mathbb{R})}=0.
\end{align}
However, the uniqueness of a solution $\phi(t,x)$ satisfying \eqref{rere1} and \eqref{rere2} is still an open problem.
In \cite{first}, the dynamics and orbital stability of two kinks of \eqref{nlww} with energy slightly bigger than two times the energy of a kink were studied. The asymptotic stability of a kink for the $\phi^{6}$ model was obtained in \cite{asympt}. See also the references \cite{klg3}, \cite{munoz}, \cite{klg} and \cite{sinegordon1} for more information on the stability and asymptotic stability of a kink for other one-dimension nonlinear wave equation models. For more information about kinks and other topological solitons, see the book \cite{solitonss}.
\par The objective of this manuscript is to construct a sequence of approximate solutions $\phi_{k}(v,t,x)$ satisfying for any $0<v\ll 1$ and $s>0$
\begin{equation*}
   \norm{ \partial^{2}_{t}\phi_{k}(v,t,x)-\partial^{2}_{x}\phi_{k}(v,t,x)+\dot U(\phi_{k}(v,t,x))}_{L^{\infty}_{t}H^{s}_{x}(\mathbb{R})}\ll v^{2k-\frac{1}{2}}, 
\end{equation*}
 and
\begin{equation*}
    \lim_{t\to+\infty}\norm{\overrightarrow{\phi_{k}}(v,t,x)-\overrightarrow{H_{0,1}}\left(\frac{x-vt}{\sqrt{1-v^{2}}}\right)-\overrightarrow{H_{-1,0}}\left(\frac{x+vt}{\sqrt{1-v^{2}}}\right)}_{H^{s}_{x}(\mathbb{R})}=0,
\end{equation*}
with $\overrightarrow{f}(t,x)=(f(t,x),\partial_{t}f(t,x))$ for any function $f\in C^{1}\left(\mathbb{R}^{2}\right).$ This result is the first part of our work about the study of the collision of two kinks with low speed $v.$  
\par The study of dynamics of multi-kink solutions for the $\phi^{6}$ is motivated from condensed matter, see  \cite{condensed}, and cosmology \cite{cosmic}. Also, there is a large literature about the numerical study of collision of multi-kinks for the $\phi^{6},$ for example in high energy physics see \cite{collision} and \cite{kinkcollision}. More precisely, in the article \cite{kinkcollision} it was numerically proved that there is a critical velocity $v_{c},$ so that if two kinks collide with a velocity smaller than $v_{c},$ the collision is very close to an elastic collision.
\par Motivated by \cite{kinkcollision}, we theoretically study the high elasticity of the collision of two kinks with low speed for the $\phi^{6}$ model. The sequence of approximate solutions $\phi_{k}(v,t,x)$ will be useful later in a second manuscript to study the elasticity of collision of two kinks with low speed. Since the $\phi^{6}$ model is a non-integrable system, there are many issues and difficulties in the studying of the collision problem for two kinks of this model. 
\par There exist few mathematical results about the inelasticity of the collision of two solitons for other dispersive models. In \cite{collision1}, Martel and Merle proved the inelasticity of the collision of two solitons with low speed for the quartic $gKdV$. There are results on the elasticity and inelasticity of the collision of solitons for $gKdV$ for a certain class of nonlinearities, see \cite{munozkdv1} and \cite{munozkdv2} by Muñoz, see also \cite{stabcol} by Martel and Merle. For nonlinear Schrödinger equation, in \cite{perelman}, Perelman studied the collision of two solitons of different sizes and obtained that after the collision the solution doesn't preserve the two solitons' structure.
\subsection{Main Results}
\begin{definition}
We define $\Lambda:C^{2}(\mathbb{R}^{2},\mathbb{R})\to C(\mathbb{R}^{2},\mathbb{R})$ as the nonlinear operator satisfying
\begin{equation*}
    \Lambda (\phi_{1})(t,x)=\partial_{t}^{2}\phi_{1}(t,x)-\partial_{x}^{2}\phi_{1}(t,x)+\dot U(\phi_{1}(t,x)),
\end{equation*}
for any function $\phi_{1} \in C^{2}(\mathbb{R}^{2},\mathbb{R}).$
And, for any smooth functions $w:(0,1)\times \mathbb{R}\to\mathbb{R},\,\phi:\mathbb{R}^{2}\to\mathbb{R},$ let $\phi_{2}(t,x)\coloneqq\phi\left(t,w(t,x)\right),$ then we define
\begin{equation*}
    \Lambda\left(\phi\left(t,w(t,x)\right)\right)=\Lambda\left(\phi_{2}\right)(t,x) \text{, for all $(t,x)\in\mathbb{R}^{2}.$}
\end{equation*}
\end{definition}
The main Theorem obtained in this manuscript is the following result:
\begin{theorem}\label{approximated theorem}
There exist a sequence of functions $\left(\phi_{k}(v,t,x)\right)_{k\geq 2},$ a sequence of real values $\delta(k)>0$ and a sequence of numbers $n_{k}\in\mathbb{N}$ such that for any $0<v<\delta(k),\,\phi_{k}(v,t,x)$ satisfies 
\begin{align*}%\label{+inftymatch}
    \lim_{t\to +\infty}\norm{\phi_{k}(v,t,x)-H_{0,1}\left(\frac{x-vt}{\sqrt{1-v^{2}}}\right)-H_{-1,0}\left(\frac{x+vt}{\sqrt{1-v^{2}}}\right)}_{H^{1}_{x}(\mathbb{R})}=0,\\
    \lim_{t\to +\infty}\norm{\partial_{t}\phi_{k}(v,t,x)+\frac{v}{\sqrt{1-v^{2}}}H_{0,1}\left(\frac{x-vt}{\sqrt{1-v^{2}}}\right)-\frac{v}{\sqrt{1-v^{2}}}H_{-1,0}\left(\frac{x+vt}{\sqrt{1-v^{2}}}\right)}_{L^{2}_{x}(\mathbb{R})}=0,\\
    %\label{-inftymatch} 
    \lim_{t\to -\infty}\norm{\phi_{k}(v,t,x)-H_{0,1}\left(\frac{x+vt-e_{v,k}}{\sqrt{1-v^{2}}}\right)-H_{-1,0}\left(\frac{x-vt+e_{v,k}}{\sqrt{1-v^{2}}}\right)}_{H^{1}_{x}(\mathbb{R})}=0,\\
     \lim_{t\to -\infty}\norm{\partial_{t}\phi_{k}(v,t,x)-\frac{v}{\sqrt{1-v^{2}}}H_{0,1}\left(\frac{x+vt-e_{v,k}}{\sqrt{1-v^{2}}}\right)+\frac{v}{\sqrt{1-v^{2}}}H_{-1,0}\left(\frac{x-vt+e_{v,k}}{\sqrt{1-v^{2}}}\right)}_{ L^{2}_{x}(\mathbb{R})}=0,
    \end{align*}
    with $e_{v,k}\in\mathbb{R}$ satisfying
\begin{equation*}
    \lim_{v\to 0}\frac{\left\vert e_{v,k} -\frac{\ln{\left(\frac{8}{v^{2}}\right)}}{\sqrt{2}}\right\vert}{v\vert\ln{(v)}\vert^{3}}=0.
\end{equation*} 
Moreover, if $0<v<\delta(k),$ then for any $s\geq 0$ and $l\in\mathbb{N}\cup\{0\},$ there exists $C(k,s,l)>0$ such that
     \begin{equation*}
     \norm{\frac{\partial^{l}}{\partial t^{l}}\Lambda(\phi_{k})(v,t,x)}_{H^{s}_{x}(\mathbb{R})}\leq C(k,s,l)v^{2k+l}\left(\vert t\vert v+\ln{\left(\frac{1}{v^{2}}\right)}\right)^{n_{k}}e^{-2\sqrt{2}\vert t\vert v}.
    \end{equation*}
\end{theorem}
\subsection{Organization of the manuscript}
\par In this manuscript, we denote by
$\mathcal{G}\in\mathscr{S}(\mathbb{R})$ the following function
\begin{equation}\label{G(x)}
    \mathcal{G}(x)=e^{-\sqrt{2}x}-\frac{e^{-\sqrt{2}x}}{(1+e^{2\sqrt{2}x})^{\frac{3}{2}}}+2\sqrt{2}x\frac{e^{\sqrt{2}x}}{(1+e^{2\sqrt{2}x})^{\frac{3}{2}}}+k_{1}\frac{e^{\sqrt{2}x}}{(1+e^{2\sqrt{2}x})^{\frac{3}{2}}},
\end{equation}
where $k_{1}\in\mathbb{R}$ is the unique real number such that $
    \left\langle \mathcal{G}(x),\,H^{'}_{0,1}(x) \right\rangle_{L^{2}_{x}(\mathbb{R})}=0.
$
The function $\mathcal{G}$ satisfies
\begin{equation}\label{idddG}
    -\frac{\partial^{2}}{\partial x^{2}}\mathcal{G}(x)+ U^{''}\left(H_{0,1}(x)\right)\mathcal{G}(x)=\left( U^{''}(H_{0,1}(x))-2\right)e^{{-}\sqrt{2}x}+8\sqrt{2}H^{'}_{0,1}(x),
\end{equation}
see Remark \ref{formulaG} in the Appendix section for the proof.
Next, we consider for $0<v<1$ the function
\begin{equation}\label{defd(t)}
d_{v}(t)=\frac{1}{\sqrt{2}}\ln{\left(\frac{8}{v^{2}}\cosh{\left(\sqrt{2}vt\right)}^{2}\right)},
\end{equation}
which is a solution to the ordinary differential equation
\begin{equation}\label{odedvv}
    \ddot d_{v}(t)=16\sqrt{2}e^{-\sqrt{2}d_{v}(t)}.
\end{equation}
\par In Section \ref{functionalan}, we are going to develop the main techniques necessary to construct each approximate solution $\phi_{k}$ of Theorem \ref{approximated theorem}. More precisely, we are going to construct function spaces in Subsection \ref{separationsection} and study the applications of Fredholm alternative of the linear operator $-\partial^{2}_{x}+ U^{''}(H_{0,1}(x))$ restricted to these function spaces in Subsection \ref{fred}. 
\par In Section \ref{prpp}, we will prove auxiliary estimates with the objective of simplify, in the next sections, the computation and evaluation of $\Lambda(\phi_{k})(v,t,x)$  for each $k\in\mathbb{N}_{\geq 2}$ and $0<v\ll 1.$
\par In Section \ref{sub31}, we are going to prove Theorem \ref{approximated theorem} for the case $k=2.$ More precisely, we will first choose the function
\begin{align*}
\varphi_{2,0}\left(t,x\right)=&H_{0,1}\left(\frac{x-\frac{d_{v}(t)}{2}}{\sqrt{1-\frac{\dot d_{v}(t)^{2}}{4}}}\right)-H_{0,1}\left(\frac{{-}x-\frac{d_{v}(t)}{2}}{\sqrt{1-\frac{\dot d_{v}(t)^{2}}{4}}}\right)\\
&{+}e^{-\sqrt{2}d_{v}(t)}\mathcal{G}\left(\frac{x-\frac{d_{v}(t)}{2}}{\sqrt{1-\frac{\dot d_{v}(t)^{2}}{4}}}\right)
-e^{-\sqrt{2}d_{v}(t)}\mathcal{G}\left(\frac{{-}x-\frac{d_{v}(t)}{2}}{\sqrt{1-\frac{\dot d_{v}(t)^{2}}{4}}}\right)
\end{align*}
as a candidate for the case $k=2.$ The next argument is to use the main results of Subsection \ref{separationsection} and Section \ref{prpp} to estimate $\Lambda(\varphi_{2,0})(t,x),$ see also Remark \ref{sepapplic} to understand better the ideas behind this argument. More precisely, we are going to verify the existence of two finite sets of Schwartz functions with exponential decay in both directions
$
    \left(h_{i}(x)\right)_{i\in I}$ and $\left(p_{i}(t)\right)_{i\in I}
$
such that
\begin{gather*}
    \Lambda(\varphi_{2,0})(t,x)=\sum_{i\in I}p_{i}(\sqrt{2}vt)\left[h_{i}\left(\frac{x-\frac{d_{v}(t)}{2}}{\sqrt{1-\frac{\dot d_{v}(t)^{2}}{4}}}\right)-h_{i}\left(\frac{{-}x-\frac{d_{v}(t)}{2}}{\sqrt{1-\frac{\dot d_{v}(t)^{2}}{4}}}\right)\right]+u_{v}(t,x),
\end{gather*}
where the function $u_{v}:\mathbb{R}^{2}\to\mathbb{R}^{2}$ is smooth and satisfies, for a real constant $q>0,$ any $l\in\mathbb{N}\cup\{0\}$ and any $s>0,$ the estimate 
\begin{equation*}
    \norm{\frac{\partial^{l}}{\partial t^{l}}u_{v}(t,x)}_{H^{s}_{x}(\mathbb{R})}\leq C(s,l) v^{6+l}\left[\ln{\left(\frac{1}{v}\right)}+\vert t\vert v\right]^{q}e^{-2\sqrt{2}\vert t\vert v} \text{, for all $t\in\mathbb{R},$ if $0<v\ll 1,$} 
\end{equation*}
 where $C(s,l)$ is a positive number depending only on $l$ and $s.$
Next, using the estimate above of $\Lambda(\varphi_{2,0})(t,x),$ we are going to construct a linear ordinary differential equation with a solution being a smooth function $r_{v}(t)$ with $L^{\infty}(\mathbb{R})$ norm of order $v^{2}\ln{\left(\frac{1}{v}\right)}.$ Using the function $r_{v}(t)$, we are going to verify, for 
\begin{align*}
    \varphi_{2,1}(t,x)=& H_{0,1}\left(\frac{x+r_{v}(t)-\frac{d_{v}(t)}{2}}{\sqrt{1-\frac{\dot d_{v}(t)^{2}}{4}}}\right)-H_{0,1}\left(\frac{{-}x+r_{v}(t)-\frac{d_{v}(t)}{2}}{\sqrt{1-\frac{\dot d_{v}(t)^{2}}{4}}}\right)\\
&{+}e^{-\sqrt{2}d_{v}(t)}\mathcal{G}\left(\frac{x+r_{v}(t)-\frac{d_{v}(t)}{2}}{\sqrt{1-\frac{\dot d_{v}(t)^{2}}{4}}}\right)
-e^{-\sqrt{2}d_{v}(t)}\mathcal{G}\left(\frac{{-}x+r_{v}(t)-\frac{d_{v}(t)}{2}}{\sqrt{1-\frac{\dot d_{v}(t)^{2}}{4}}}\right),
\end{align*}
and an explicit real value $e_{2,v},$ that the function $\phi_{2}(v,t,x)\coloneqq\varphi_{2,1}(t+e_{2,v},x)$ satisfies Theorem \ref{approximated theorem} for the case $k=2.$
\par In Section \ref{sub32}, we are going to prove Theorem \ref{approximated theorem} by an argument of induction on $k\in\mathbb{N}_{\geq 2}.$ The proof of complementary information is done in the Appendix section \ref{auxi}. 
\subsection{Notation}
  In this subsection, we will present the notations that are going to be used in the next sections of the manuscript.
 \begin{notation}
\par For any pair of functions $w:\mathbb{R}^{2}\to \mathbb{R},\,h\in L^{\infty}_{x}(\mathbb{R})$ we denote $h^{w}(t,x)$ by the following function \begin{equation*}
    h^{w}(t,x)=h\left(w(t,x)\right)-h\left(w(t,{-}x)\right) \text{ for any $(t,x)\in\mathbb{R}^{2}.$}
\end{equation*} 
\par For any non-negative function $k:D\subset \mathbb{R}\to \mathbb{R}_{\geq 0},$ we say that $f(x)=O(k(x)),$ if $f$ has the same domain $D$ as $k$ and there is a constant $C>0$ such that
$\vert f(x) \vert \leq C k(x) \text{ for any $x \in D.$}$
For any two non-negative real functions $f_{1}(x)$ and $f_{2}(x),$ we have that the condition $f_{1}\lesssim f_{2}$ is true if there is a constant $C>0$ such that $f_{1}(x)\leq C f_{2}(x) \text{ for any $x\in\mathbb{R}.$} $
Furthermore, for a finite number of real variables $\alpha_{1},\,...,\,\alpha_{n},$ we say that any two non-negative functions $f(\alpha_{1},...,\alpha_{n},x)$ and $g(\alpha_{1},...,\alpha_{n},x)$ satisfy $f\lesssim_{\alpha_{1},\,..., \alpha_{n}}g,$ if and only if $f$ and $g$ have the same domain $\mathcal{D}_{n}\times \mathbb{R}$ for some set $\mathcal{D}_{n}\subset\mathbb{R}^{n}$ and there exists a positive function $C_{n}:\mathcal{D}_{n}\to \mathbb{R}_{+}$ such that
\begin{equation}
    f(\alpha_{1},...,\alpha_{n},x)\leq C_{n}(\alpha_{1},\, ... , \alpha_{n})g(\alpha_{1},...,\alpha_{n},x) \text{ for all $x$ in $\mathbb{R}$ and all $(\alpha_{1},...,\alpha_{n})\in \mathcal{D}_{n}.$}
\end{equation}
 \par We consider for any $f\in H^{1}_{x}(\mathbb{R})$ and any $g\in L^{2}_{x}(\mathbb{R})$ the following notations
\begin{equation*}
    \norm{f}_{H^{1}_{x}}=\norm{f}_{H^{1}_{x}(\mathbb{R})}=\left(\norm{f}_{L^{2}_{x}(\mathbb{R})}^{2}+\norm{\frac{df}{dx}}_{L^{2}_{x}(\mathbb{R})}^{2}\right)^{\frac{1}{2}},\quad
    \norm{g}_{L^{2}_{x}}=\norm{g}_{L^{2}_{x}(\mathbb{R})},
\end{equation*}
and the norm $\norm{\cdot}_{H^{1}_{x}\times L^{2}_{x}}$ given by
\begin{equation*}
    \norm{(f_{1}(x),f_{2}(x))}_{H^{1}_{x}\times L^{2}_{x}}=\left(\norm{f_{1}}_{H^{1}_{x}(\mathbb{R})}^{2}+\norm{f_{2}(x)}_{L^{2}_{x}(\mathbb{R})}^{2}\right)^{\frac{1}{2}},
\end{equation*}
for any $(f_{1},f_{2})\in H^{1}_{x}(\mathbb{R})\times L^{2}_{x}(\mathbb{R}).$ For any $(f_{1},f_{2}) \in L^{2}_{x}(\mathbb{R})\times L^{2}_{x}(\mathbb{R})$ and any $(g_{1},g_{2})\in L^{2}_{x}(\mathbb{R})\times L^{2}_{x}(\mathbb{R}),$ we denote
\begin{equation*}
    \left\langle (f_{1},f_{2}),\,(g_{1},g_{2}) \right\rangle=\int_{\mathbb{R}} f_{1}(x)g_{1}(x)+f_{2}(x)g_{2}(x)\,dx.
\end{equation*}
For any functions $f_{1}(x),\,g_{1}(x)\in L^{2}_{x}(\mathbb{R}),$ we denote
\begin{equation*}
    \left\langle f_{1},g_{1} \right\rangle=\int_{\mathbb{R}}f_{1}(x)g_{1}(x)\,dx.
\end{equation*}
Moreover, for any $s\geq 0,$ we consider the norm $\norm{\cdot}_{H^{s}_{x}}$ given by
\begin{equation*}
    \norm{f}_{H^{s}_{x}}=\norm{f}_{H^{s}_{x}(\mathbb{R})}=\left(\int_{\mathbb{R}}(1+\vert x\vert)^{2s}\vert \hat{f}(x)\vert^{2}\,dx\right)^{\frac{1}{2}} \text{, for any $f\in H^{s}_{x}(\mathbb{R}),$}
\end{equation*}
where $\hat{f}$ is the fourier transform of the function $f.$
We denote $\mathbb{D}$ as the set given by $\{z\in\mathbb{C}\big\vert \, \vert z\vert<1\}.$ In this manuscript, we consider the set $\mathbb{N}$ as the set of all positive integers.
\end{notation}
\section{Functional analysis methods}\label{functionalan}
\subsection{Asymptotic analysis methods}\label{separationsection}
We will use the following Lemma on several occasions.
\begin{lemma}\label{interactt}
For any real numbers $x_{2},x_{1}$, such that $\zeta=x_{2}-x_{1}>0$ and $\alpha,\,\beta,\,m>0$ with $\alpha\neq \beta$ the following bound holds:
\begin{equation*}
    \int_{\mathbb{R}}\vert x-x_{1}\vert ^{m} e^{-\alpha(x-x_{1})_{+}}e^{-\beta(x_{2}-x)_{+}}\lesssim_{\alpha,\beta,m} \max\left(\left(1+\zeta^{m}\right)e^{-\alpha \zeta},e^{-\beta \zeta}\right),
\end{equation*}
For any $\alpha>0$, the following bound holds
\begin{equation*}
    \int_{\mathbb{R}}\vert x-x_{1}\vert^{m} e^{-\alpha(x-x_{1})_{+}}e^{-\alpha(x_{2}-x)_{+}}\lesssim_{\alpha}\left[1+\zeta^{m+1}\right] e^{-\alpha \zeta}.
\end{equation*}
\end{lemma}
\begin{proof}
Elementary computations.
\end{proof}
Next, we define the function spaces $S^{+}$ and $S^{-}.$ They will be used to construct the approximate solutions $\phi_{k}(v,t,x)$ of Theorem \ref{approximated theorem} for each $k\in \mathbb{N}_{\geq 2}.$ 
\begin{definition}\label{s+}
 $S^{+}$ is the linear subspace of $L^{\infty}(\mathbb{R})$ such that $f\in S^{+},$ if and only if all the following conditions are true 
\begin{itemize}
    \item $ f^{'} \in \mathscr{S}(\mathbb{R})$ and there is a holomorphic function $F:\{z\in \mathbb{C}|\,{-}1<\I{(z)}<1\}\to\mathbb{C}$ such that $F(e^{\sqrt{2}x})=f(x)$ for all $x\in\mathbb{R}.$
    \item $F$ satisfies $F(z)=\sum_{k=0}^{+\infty}a_{k}z^{2k+1},$
   for some sequence of real numbers $(a_{k})$ and all $z\in\mathbb{D}.$
\end{itemize}
\end{definition}
\begin{definition}\label{s-}
$S^{-}$ is the linear subspace of $L^{\infty}(\mathbb{R})$ such that $g\in S^{-},$ if and only if all the following conditions are true 
\begin{itemize}
    \item $ g^{'} \in \mathscr{S}(\mathbb{R})$ and there is a holomorphic function $G:\{z\in \mathbb{C}|\,{-}1<\I{(z)}<1\}\to\mathbb{C}$ such that $G(e^{{-}\sqrt{2}x})=g(x)$ for all $x\in\mathbb{R}.$
    \item $G$ satisfies
    $G(z)=\sum_{k=1}^{+\infty}b_{k}z^{2k},$
   for some sequence of real numbers $(b_{k})$ and all $z\in\mathbb{D}.$
\end{itemize}
\end{definition}
\begin{remark}
In Definitions \ref{s+} and \ref{s-}, from standard complex analysis theory, the holomorphic functions $F$ and $G$ are unique.
\end{remark}
\begin{remark}\label{alg}
     From Definition \ref{s-}, if $f_{1},\,f_{2}\in S^{-},$ then $f_{1}f_{2}\in S^{-}.$ Therefore, $S^{-}$ is an algebra.
\end{remark}
\begin{remark}\label{dsremark}
 From Definitions \eqref{s+} and \ref{s-}, if $f\in S^{+}$ and $g\in S^{-},$ then, for any $l\in\mathbb{N},$ $f^{(l)}\in S^{+}$ and $g^{(l)}\in S^{-}.$ 
\end{remark}The following Lemma is a direct consequence of Definitions \ref{s+} and \ref{s-}. 
\begin{lemma}[Multiplicative Lemma]\label{multiplicative}
    If $f_{1},\,f_{2},\,f_{3}\in S^{+},$  then the function $g_{1}(x)\coloneqq f_{1}({-}x)f_{2}({-}x)$ is in $S^{-}$ and the function $g_{2}(x)\coloneqq f_{1}(x)f_{2}(x)f_{3}(x)$ is in $S^{+}.$
\end{lemma}
\begin{definition}
We define, for any $n\in \mathbb{N}\cup\{0\},$ the linear spaces
$S^{+,n}=\{x^{n}f(x)|\,f(x)\in S^{+}\cap \mathscr{S}(\mathbb{R})\}$
and
$S^{-,n}=\{x^{n}f(x)|\,f(x)\in S^{-}\cap \mathscr{S}(\mathbb{R})\},$
and for any $m\in\mathbb{N}\cup\{0\},$ we define 
\begin{equation*}
    S^{+}_{m}=\bigoplus_{n=0}^{m} S^{+,n},\,S^{-}_{m}=\bigoplus_{n=0}^{m} S^{-,n},\,
    S^{+}_{\infty}=\bigoplus_{n=0}^{+\infty} S^{+,n},\,S^{-}_{\infty}=\bigoplus_{n=0}^{+\infty} S^{-,n}.
\end{equation*}
\end{definition}
\begin{remark}\label{genemult}
From Remark \ref{alg}, $S^{-}_{\infty}$ is an algebra. Furthermore, the result of Lemma \ref{multiplicative} is also true if we replace the function spaces $S^{+}$ and $S^{-},$ respectively, with $S^{+}_{\infty}$ and $S^{-}_{\infty}$ in the statement of this lemma.
\end{remark}
\begin{remark}
We will prove later in Lemma \ref{independent} that the linear space generated by the union of all subspaces $S^{+}_{n}\subset\mathscr{S}(\mathbb{R})$ is a direct  sum. By analogy, the same result is true for the union of all subspaces $S^{-}_{n}.$ \end{remark}
\begin{remark}\label{infinite}
In the definition \ref{s+}, we can verify that if $F(z)$ is a polynomial function, then $F\equiv 0.$ Otherwise, the identity $f(x)=F(e^{\sqrt{2}x})$ would imply that $\lim_{x\to +\infty}\vert f(x)\vert=\lim_{x\to+\infty}\vert F(e^{\sqrt{2}x})\vert =+\infty,$ if  $F(z)$ is a non-trivial polynomial, which contradicts first condition in definition \ref{s+}. Similarly, we can verify that $G(z)$ in definition \ref{s-} cannot be a non-zero polynomial.
\end{remark}
\begin{remark}\label{basisS}
 For any odd number $m$ and even number $n,$ we have that $H_{0,1}(x)^{m}\in S^{+}$ and $H_{-1,0}(x)^{n} \in S^{-}.$
\end{remark}
\begin{definition}
In notation of definition \ref{s+}, If $f\in S^{+},$ we define
\begin{equation*}
    \val_{+}(f)=\min\{2k+1\vert\, k\in\mathbb{N}\cup\{0\},\,  a_{k}\neq 0 \}.
\end{equation*}
And in notation of definition \ref{s-}, if $g \in S^{-},$ we define
\begin{equation*}
    \val_{-}(g)=\min\{2k\vert\,k \in\mathbb{N},\, b_{k}\neq 0 \}.
\end{equation*}
\end{definition}
\begin{remark}
The exponential decay of the functions in $S^{+}\cap \mathscr{S}(\mathbb{R}),\,S^{-}\cap \mathscr{S}(\mathbb{R})$ and $S^{+}_{m}$ are  going to be very important to obtain high precision in the approximate solutions of the main theorem.
\end{remark}
\par Now, we can prove the main proposition of this subsection. 
\begin{proposition}[Separation Lemma]\label{separation}
If $f \in S^{+},\,g \in S^{-},$ then there exist a sequence of pairs $(h_{n},d_{n})_{n\geq 1}$ and a set $\Delta\subset \mathbb{N}$ such that $h_{n}(x) \in S^{+}\cap\mathscr{S}(\mathbb{R})$ for all $n\in \Delta,\,h_{n}(-x)$ is in $S^{+}\cap \mathscr{S}(\mathbb{R})$ for all $n\in \Omega=\mathbb{N}\setminus \Delta$ and $(d_{n})_{n\geq 1}\subset \mathbb{N}$ is a strictly increasing sequence satisfying, for any $\mathcal{M} \in \mathbb{N}$ and any $\zeta\geq 1,$ the following equation
\begin{equation}\label{idsepl}
    f(x-\zeta)g(x)=\sum_{\substack{1 \leq n\leq \mathcal{M},\\ n\in \Delta}}h_{n}(x-\zeta)e^{-\sqrt{2}d_{n}\zeta}+\sum_{\substack{1 \leq n\leq \mathcal{M},\\ n\in \Omega}}h_{n}(x)e^{-\sqrt{2}d_{n}\zeta}+e^{-\sqrt{2}d_{\mathcal{M}} \zeta}f_{\mathcal{M}}(x-\zeta)g_{\mathcal{M}}(x),
\end{equation}
where $f_{\mathcal{M}}\in S^{+},$ $g_{\mathcal{M}} \in S^{-}$ with $f_{\mathcal{M}}$ or $g_{\mathcal{M}}$ in $\mathscr{S}(\mathbb{R}).$ 
Also,
$\norm{f_{\mathcal{M}}(x-\zeta)g_{\mathcal{M}}(x)}_{H^{k}_{x}(\mathbb{R})}\lesssim_{k,\mathcal{M}} 1$ for any $\zeta\geq 1.$
\end{proposition}

\begin{lemma}\label{prelema}
Let $f\in S^{+},\,g\in S^{-},$ then: 
\begin{itemize}
\item If $\val_{+}(f)>\val_{-}(g),$ then there exist $h_{1}\in S^{+}\cap\mathscr{S}(\mathbb{R})$ and functions $f_{1}\in S^{+},\,g_{1}\in S^{-}$ satisfying, for any $\zeta\geq 1,$ the following identity
\begin{equation*}
    f(x-\zeta)g(x)=h_{1}(x-\zeta)e^{{-}\val_{-}(g)\zeta}+e^{{-}\sqrt{2}\val_{-}(g)\zeta}f_{1}(x-\zeta)g_{1}(x),
\end{equation*}
and  at least one of the functions $f_{1},\,g_{1}$ is in $\mathscr{S}(\mathbb{R}).$
\item If $\val_{-}(g)>\val_{+}(f),$ then there exist $\hat{h}_{1}\in\mathscr{S}(\mathbb{R})\cap S^{+}$  and functions $f_{1}\in S^{+},\,g_{1}\in S^{-}$ satisfying, for any $\zeta\geq 1,$ the following identity
\begin{equation*}
    f(x-\zeta)g(x)=\hat{h}_{1}({-}x)e^{{-}\val_{+}(f)\zeta}+e^{{-}\sqrt{2}\val_{+}(f)\zeta}f_{1}({-}x+\zeta)g_{1}({-}x),
\end{equation*}
and at least one of the functions $f_{1},\,g_{1}$ is in $\mathscr{S}(\mathbb{R}).$
\end{itemize}
\end{lemma}
\begin{proof}[Proof of Lemma \ref{prelema}]
   We consider the notation of Definition \ref{s+} and Definition \ref{s-}. 
   For $2w_{1}+1=\val_{+}(f)$ and $2w_{2}=\val_{-}(g)$ there are only two cases to consider, which are $2w_{1}+1>2w_{2}$ and $2w_{1}+1<2w_{2}.$ 
   \par First, we consider the case where $2w_{1}+1>2w_{2}.$ 
\begin{align*}
    f(x-\zeta)g(x)=&f(x-\zeta)b_{w_{2}}e^{-2w_{2}\sqrt{2}x}+f(x-\zeta)\left[g(x)-b_{w_{2}}e^{-2w_{2}\sqrt{2}x}\right]\\
    =&f(x-\zeta)b_{w_{2}}e^{-2\sqrt{2}w_{2}(x-\zeta)}e^{-2\sqrt{2}w_{2}\zeta}\\&{+}e^{-2\sqrt{2}w_{2}\zeta}\left[f(x-\zeta)e^{-2\sqrt{2}w_{2}(x-\zeta)}\left(g(x)e^{+2\sqrt{2}w_{2}x}-b_{w_{2}}\right)\right].
\end{align*}
Because $2w_{1}+1>2w_{2}$ and $f\in S^{+},$ we have that $f(x)e^{-2w_{2}\sqrt{2}x}\in S^{+}\cap \mathscr{S}(\mathbb{R}).$
Clearly, if $g(x)e^{+2\sqrt{2}w_{2}x}-b_{w_{2}}\in S^{-},$ then, from the identity above, Lemma \ref{prelema} would be true for the case where $\val_{+}(f)>\val_{-}(g).$ Moreover, for any $x>0,$ we have that
\begin{equation}\label{+++1}
    g(x)e^{+2\sqrt{2}w_{2}x}-b_{w_{2}}=\sum_{n=w_{2}+1}^{+\infty}b_{n}e^{-2(n-w_{2})\sqrt{2}x}.
\end{equation}
Since $G(z)$ is analytic in the region $\mathbb{D},$ we clearly have that the following function
\begin{equation}\label{ddd2}
    Q(z)=\frac{G(z)}{z^{2w_{2}}}-b_{w_{2}}=\sum_{n=w_{2}+1}^{+\infty}b_{n}z^{2(n-w_{2})}
\end{equation}
is analytic in $\mathbb{D},$ from which, using the product rule of the derivative, for any $x>1$ and $l,\,m\in\mathbb{N},$ we deduce that
\begin{equation}\label{schwartz+}
    \left|\left(1+|x|^{m}\right)\frac{d^{l}}{dx^{l}}\left[g(x)e^{+2\sqrt{2}w_{2}x}-b_{w_{2}}\right]\right|\lesssim_{l,m} 1.
\end{equation}
From equation \eqref{ddd2} and from the fact that $G(z)$ has a holomorphic extension in the region $\mathcal{B}=\{z\vert\, -1<\I{z}<1\}$
since $g \in S^{-},$ we conclude that $Q(z)$ has a holomorphic extension in the region $\mathcal{B}.$ Moreover, since $g\in S^{-},$ then $g\in L^{\infty}_{x}(\mathbb{R})\cap C^{\infty}(\mathbb{R})$ and $g^{'}\in \mathscr{S}(\mathbb{R})$, from which we deduce the following estimate
\begin{equation*}
    \left|(1+|x|^{m})\frac{d^{l}}{dx^{l}}\left[g(x)e^{+2\sqrt{2}w_{2}x}\right]\right|\lesssim_{l,m} 1 \text{ for any $x<-1$ and $l\in\mathbb{N}_{\geq 1},$}
\end{equation*}
and so, we conclude that $ \frac{d}{dx}\left[g(x)e^{+2\sqrt{2}w_{2}x}-b_{w_{2}}\right]\in \mathscr{S}(\mathbb{R}).$ 
\par Analogously, if $2w_{2}=val_{-}(g)>val_{+}(f)=2w_{1}+1,$ then we can deduce from the Definition \ref{s+} and Definition \ref{s-} that 
\begin{equation*}
h_{1}(x)=g({-}x)e^{{-}(2w_{1}+1)\sqrt{2}x}\in S^{+}\cap\mathscr{S}(\mathbb{R}),\, g_{1}(x)=f({-}x)e^{(2w_{1}+1)\sqrt{2}x}-a_{w_{1}}\in S^{-},
\end{equation*}
and
\begin{multline}\label{general2}
f(x-\zeta)g(x)=g(x)a_{w_{1}}e^{\sqrt{2}(2w_{1}+1)x} e^{{-}\sqrt{2}(2w_{1}+1)\zeta}\\+e^{{-}\sqrt{2}(2w_{1}+1)\zeta}\left[g(x)e^{\sqrt{2}(2w_{1}+1)x}\left(f(x-\zeta)e^{{-}\sqrt{2}(2w_{1}+1)(x-\zeta)}-a_{w_{1}}\right)\right].
\end{multline}
\end{proof}
\begin{proof}[Proof of Proposition \ref{separation}]
If $f\equiv0$ or $g\equiv0$, we can take $h_{n}=0$ and $d_{n}=n$ for all $n\in\mathbb{N}.$ 
From now on, we consider the case where both $f$ and $g$ are not identically zero. Clearly, Lemma \ref{prelema} implies Proposition \ref{separation} for the case where $\mathcal{M}=1.$ 
\par Moreover, if Proposition \ref{separation} is true when $\mathcal{M}=M_{0}\in\mathbb{N},$ we can repeat the argument above of the proof of Lemma \ref{prelema} using $f_{M_{0}},\,g_{M_{0}}$ in the place of $f,\,g$ and conclude that Proposition \ref{separation} is also true when $\mathcal{M}=M_{0}+1,$ so by induction on $\mathcal{M},$ Proposition \ref{separation} is true for all $\mathcal{M}\in\mathbb{N}.$ 
\end{proof}
\begin{corollary}\label{unique1}
If $f\in S^{+}\,,g\in S^{-}$ and $f\not\equiv 0,\,g\not\equiv 0,$ then the sequence $(h_{n},d_{n})_{n\in\mathbb{N}}$ satisfying Proposition \ref{separation} is unique. Furthermore, $d_{1}=\min\left(\val_{+}(f),\val_{-}(g)\right).$ 
\end{corollary}
\begin{proof}
\par  From an argument of analogy, it is enough to consider the case where $2w_{1}+1=val_{+}(f)>val_{-}(g)=2w_{2}.$ In this case, from the proof of Proposition \ref{separation}, we have that the real function $\hat{h}_{1}(x)=b_{w_{2}}f(x)e^{{-}2\sqrt{2}w_{2}x}$ satisfies $\hat{h}_{1}\in S^{+}\cap \mathscr{S}(\mathbb{R}),\,\hat{h}_{1}\not\equiv 0$ and the following identity
\begin{equation}\label{eqh00}
    f(x-\zeta)g(x)=\hat{h}_{1}(x-\zeta)e^{-2\sqrt{2}w_{2}\zeta}+e^{-2\sqrt{2}w_{2}\zeta}f_{1}(x-\zeta)g_{1}(x),
\end{equation}
where $f_{1}\in S^{+},\,g_{1}\in S^{-}$ and either $f_{1}$ or $g_{1}$ is in $\mathscr{S}(\mathbb{R}).$ In conclusion, Lemma \ref{interactt} and equation \eqref{eqh00} imply for any $s\geq 0$ that
\begin{equation}\label{h00}
   \lim_{\zeta\to+\infty} \norm{f(x-\zeta)g(x)e^{2\sqrt{2}w_{2}\zeta}-\hat{h}_{1}(x-\zeta)}_{H^{s}_{x}}=0,
\end{equation}
and so, 
\begin{equation}\label{h01}
    0<\lim_{\zeta\to+\infty} \norm{f(x-\zeta)g(x)e^{2\sqrt{2}w_{2}\zeta}}_{H^{s}_{x}}<\infty.
\end{equation}
\par Since $\hat{h}_{1}\in \mathscr{S}(\mathbb{R})$ and $\hat{h}_{1}\not \equiv 0,$ $\norm{\hat{h}_{1}}_{H^{s}_{x}}\neq 0$ for all $s\geq 0.$ Therefore, using equations \eqref{h00} and \eqref{h01}, we can verify that the unique possible choice for $d_{1}$ is $2w_{2}.$ And so, the function $h_{1}$ satisfying Proposition \ref{separation} for $f$ and $g$ is unique and equal to $\hat{h}_{1},$ otherwise \eqref{h00} would be false. Similarly, we can repeat the argument above for the case $\val_{+}(f)<\val_{-}(g)$ and obtain in this situation that $d_{1}=\val_{+}(f)$ and $h_{1}(x)=a_{w_{1}}g(x)e^{(2w_{1}+1)\sqrt{2}x}.$
\par Next, assuming that $(h_{n},d_{n})$ is unique for all $1\leq n\leq \mathcal{M}_{0}\in\mathbb{N},$ we can repeat the argument above in $f_{\mathcal{M}_{0}}(x-\zeta)g_{\mathcal{M}_{0}}(x)$ and conclude that $(h_{\mathcal{M}_{0}+1},d_{\mathcal{M}_{0}+1})$ is unique too. In conclusion, from the principle of finite induction applied on $n\in\mathbb{N},$ we obtain the uniqueness of $(h_{n},d_{n})_{n\in\mathbb{N}}$ satisfying Proposition \ref{separation} when both functions $f$ and $g$ are not identically zero. 
\end{proof}
\begin{remark}\label{algounique}
   When $f\not\equiv 0$ and $g\not \equiv 0,$ we can find explicitly the sequence $(h_{n},d_{n})$  satisfying \eqref{idsepl} from the proof of Proposition \ref{separation}.
\end{remark}
\begin{remark}\label{sepc}
If $
f(x)=x^{m}f_{0}(x),\,g(x)=x^{l}g_{0}(x)$
such that $m,\, l \in \mathbb{N},\,f_{0} \in S^{+}\cap \mathscr{S}(\mathbb{R})$ and $g_{0} \in S^{-}\cap \mathscr{S}(\mathbb{R}),$ then there exist a sequence of pairs $(h_{n},d_{n})_{n\geq 1}$ and a set $\Delta\subset \mathbb{N}$ satisfying $h_{n}(x) \in S^{+}\cap\mathscr{S}(\mathbb{R})$ for all $n\in \Delta,\,h_{n}(-x)$ is in $S^{+}\cap \mathscr{S}(\mathbb{R})$ for all $n\in \Omega=\mathbb{N}\setminus \Delta,$ $d_{n}\in \mathbb{N}$ is strictly increasing such that for any $\zeta\geq 1,\,x\neq 0,\,x\neq \zeta$ and $\mathcal{M} \in \mathbb{N}$ we have the following equation
\begin{equation*}
   \frac{f(x-\zeta)g(x)}{(x-\zeta)^{m}x^{l}}=\sum_{\substack{1\leq n \leq \mathcal{M},\\ n\in \Delta}}h_{n}(x-\zeta)e^{-\sqrt{2}d_{n}\zeta}+\sum_{\substack{1\leq n \leq \mathcal{M},\\ n\in \Omega}}h_{n}(x)e^{-\sqrt{2}d_{n}\zeta}+e^{-\sqrt{2}d_{\mathcal{M}}\zeta}f_{\mathcal{\mathcal{M}}}(x-\zeta)g_{\mathcal{M}}(x),
\end{equation*}
where
 $f_{\mathcal{M}} \in S^{+},\,g_{\mathcal{M}}\in S^{-}$ and $f_{\mathcal{M}}$ or $g_{\mathcal{M}}$ is in $\mathscr{S}(\mathbb{R}).$ Furthermore, the sequence $(h_{n},d_{n})_{n\in\mathbb{N}}$ is unique.
\end{remark}
\begin{remark}\label{reflection}
From Proposition \ref{separation}, we can deduce if $f(-x)\in S^{+},\,g(-x)\in S^{-},\,f\not\equiv 0$ and $g\not\equiv 0,$ then
there exists a sequence of 
pairs $(h_{n},d_{n})_{n\geq 1}$ and a set $\Delta\subset \mathbb{N}$ such that $h_{n}(x)$ is in $S^{+}\cap \mathscr{S}(\mathbb{R})$ for all $n\in \Delta,\,h_{n}(-x)$ is in $S^{+}\cap \mathscr{S}(\mathbb{R})$ for all $n\in\Omega=\mathbb{N}\setminus \Delta$ and $(d_{n})_{n\geq 1}\subset \mathbb{N}$ is a strictly increasing sequence satisfying, for any $\mathcal{M} \in \mathbb{N}$ and any $\zeta\geq 1,$ the following equation
\begin{equation}\label{idsepl2}
    f({-}x+\zeta)g({-}x)=\sum_{\substack{1 \leq n\leq \mathcal{M},\\ n\in \Delta}}h_{n}(x-\zeta)e^{-\sqrt{2}d_{n}\zeta}+\sum_{\substack{1 \leq n\leq \mathcal{M},\\ n\in \Omega}}h_{n}(x)e^{-\sqrt{2}d_{n}\zeta}+e^{-\sqrt{2}d_{\mathcal{M}} \zeta}f_{\mathcal{M}}(x-\zeta)g_{\mathcal{M}}(x),
\end{equation}
where
 $f_{\mathcal{M}} \in S^{+},\,g_{\mathcal{M}}\in S^{-}$ and $f_{\mathcal{M}}$ or $g_{\mathcal{M}}$ is in $\mathscr{S}(\mathbb{R}).$ Furthermore, the sequence $(h_{n},d_{n})_{n\in\mathbb{N}}$ is unique.
\end{remark}
\par We also demonstrate the following lemma, which will be essential to obtain the results in the next subsection.
\begin{lemma}\label{independent}
Let $m\in\mathbb{N}$ and $f_{j}\in S^{+}\cap \mathscr{S}(\mathbb{R})$ for $0\leq j\leq m,\,\sum_{j=0}^{m}x^{j}f_{j}(x)=0,$
if and only if $f_{j}\equiv 0$ for all $0\leq j\leq m.$
\end{lemma}
\begin{proof}
For each $0\leq j\leq m,$ since $f_{j}\in S^{+},$ we have that either $f_{j}\equiv 0$ or there exists a natural $d_{j}\in\mathbb{N}\cup\{0\}$ and $a_{j}\in\mathbb{R}$ with $a_{j}\neq 0$ such that $f_{j}(x)=a_{j}e^{(2d_{j}+1)\sqrt{2}x}+O\left(e^{(2d_{j}+3)\sqrt{2}x}\right)$ for all $x\leq -1.$ So, there are only two possible cases to consider. \\
\textbf{Case 1.}($\exists f_{j}$ such that $f_{j}(x)\neq 0$ for some $x\leq -1.$) In this situation, we have that there is a natural $d_{min}\geq 0$ and a non-trivial real polynomial $p(x)$ of degree at most $m$ such that 
\begin{equation}\label{identityper}
    0=\sum_{j=0}^{m}x^{j}f_{j}(x)=e^{(2d_{\min}+1)\sqrt{2}x}p(x)+O\left(e^{(2d_{\min}+3)\sqrt{2}x}|x|^{m+1}\right) \text{ for all $x\leq -1,$}
\end{equation}
which is not possible since if $p(x)$ is a non-identically zero polynomial, then $p(x)=c$ for $c\neq 0$ or $\lim_{|x|\to+\infty}|p(x)|=+\infty,$ but both cases contradict identity \eqref{identityper}.\\
\textbf{Case 2.}($f_{j}\equiv 0$ for all $0\leq j\leq m$)
Clearly, the second case is the only possible.
\end{proof}
\subsection{Applications of Fredholm alternative}\label{fred}
We consider the self-adjoint unbounded linear operator $L:H^{2}_{x}(\mathbb{R})\subset L^{2}_{x}(\mathbb{R})\to L^{2}_{x}(\mathbb{R})$ defined by
\begin{equation}\label{operatorL}
    L(f)(x)=-f^{''}(x)+ U^{''}(H_{0,1}(x))f(x) \text{ for all $x\in\mathbb{R}.$}
\end{equation}
From Lemma $2.6$ of \cite{first} we know for a constant $\lambda>0$ that
$\sigma(L)\subset \{0\}\cup [\lambda,+\infty),\,\ker(L)=\{c H^{'}_{0,1}(x)|\, c\in\mathbb{C}\}.
$
From this, in the proof of Lemma $2.5$ of \cite{first}, we have deduced the existence of a constant $k>0$ such that if $g \in H^{1}_{x}(\mathbb{R})$ satisfies
$
    \langle g,\, H^{'}_{0,1}\rangle=0,
$
we have that
\begin{equation}\label{coerc}
    \left\langle L(g),\,g \right\rangle\geq  k \norm {g}_{H^{1}_{x}}^{2}.
\end{equation}
\par Next, we consider the linear space
\begin{equation*}
  Ort( H^{'}_{0,1})=\left\{g\in L^{2}_{x}(\mathbb{R})\vert\,\langle g, H^{'}_{0,1}\rangle=0\right\}.
\end{equation*}
Since $0<H_{0,1}<1$ and $U$ is a smooth function, Cauchy-Schwarz inequality implies for any $u,\,v\in Ort(H^{'}_{0,1})\cap H^{1}_{x}\left(\mathbb{R}\right)$ that
\begin{equation}\label{bound}
    \left|\left\langle L(u),v\right\rangle \right|\leq \norm{u^{'}}_{L^{2}_{x}(\mathbb{R})}\norm{ v^{'}}_{L^{2}_{x}(\mathbb{R})}+\norm{U^{''}}_{L^{\infty}_{x}[-1,1]}\norm{u}_{L^{2}_{x}(\mathbb{R})}\norm{v}_{L^{2}_{x}(\mathbb{R})}. 
\end{equation}
\par In conclusion, from Lax-Milgram Theorem and inequalities \eqref{coerc}, \eqref{bound}, we obtain for any bounded linear map
$\mathcal{A}:\left(Ort(H^{'}_{0,1})\cap H^{1}_{x}\left(\mathbb{R}\right),\norm{\cdot}_{H^{1}_{x}(\mathbb{R})}\right)\to \mathbb{R}$
the existence of a unique $h_{\mathcal{A}} \in Ort( H^{'}_{0,1})\cap H^{1}_{x}\left(\mathbb{R}\right)$ such that,
for any $u \in Ort( H^{'}_{0,1})\cap H^{1}_{x}(\mathbb{R}),$ we have
\begin{equation}\label{inversefunction}
    \left\langle L(h_{\mathcal{A}}),\,u \right\rangle=\mathcal{A}(u).
\end{equation}
As a consequence, we can obtain, for any $v \in L^{2}_{x}(\mathbb{R}),$ the existence of a unique $h(v) \in Ort\left( H^{'}_{0,1}\right)\cap H^{1}_{x}\left(\mathbb{R}\right)$ satisfying for any $u\in  H^{1}_{x}\left(\mathbb{R}\right)$ the following identity
\begin{equation*}
     \left\langle L(h(v)),\,u \right\rangle=\left\langle v,u\right\rangle.
\end{equation*}
Then, inequalities \eqref{coerc}, \eqref{bound} imply the existence of $\beta>0$ such that for any $v \in Ort\left(H^{'}_{0,1}\right)\cap H^{1}_{x}\left(\mathbb{R}\right),$ 
$\norm{h(v)}_{H^{1}_{x}(\mathbb{R})}\leq\beta\norm{v}_{L^{2}_{x}(\mathbb{R})}.
$
In conclusion, from the density of $H^{1}_{x}(\mathbb{R})$ in $L^{2}_{x}(\mathbb{R})$ and the fact that $h(v)\in Ort\left(H^{'}_{0,1}\right)\cap H^{1}_{x}\left(\mathbb{R}\right),$ we deduce the following lemma:
\begin{lemma}\label{firstinvert}
There is a unique injective and bounded map
\begin{equation*}
    L_{1}: \left(Ort(H^{'}_{0,1}),\norm{\cdot}_{L^{2}_{x}(\mathbb{R})}\right)\to \left( Ort\left(H^{'}_{0,1}\right)\cap H^{1}_{x}\left(\mathbb{R}\right),\norm{\cdot}_{H^{1}_{x}(\mathbb{R})}\right),
\end{equation*}
such that for any $v\in Ort(H^{'}_{0,1}),$ $L(L_{1}(v))=v.$  
\end{lemma}
Now, for all $m\in \mathbb{N}\cup\{0\},$ we are going to consider the linear spaces $S^{+}_{m}\cap Ort\left(H^{'}_{0,1}\right)$ and study the applications of the operator $L_{1}$ in these subspaces. More precisely, we are going to prove the following lemma: 
\begin{lemma}\label{secondinvert}
The map $L_{1}$ defined in Lemma \ref{firstinvert} satisfies
$L_{1}\left(S^{+}_{m}\cap Ort(H^{'}_{0,1})\right)\subset S^{+}_{m+1}\cap Ort(H^{'}_{0,1})$
for all $m\in\mathbb{N}\cup{\{0\}}.$
\end{lemma}
\begin{proof}
From Lemma \ref{schwartzinver} in Appendix section, we have that if $f\in \mathscr{S}(\mathbb{R})\cap Ort\left(H^{'}_{0,1}\right),$ then $L_{1}(f)\in \mathscr{S}(\mathbb{R}).$ Since $L_{1}$ is a linear map, it is enough to prove for any $g(x) \in S^{+}\cap \mathscr{S}(\mathbb{R})$ and any $m\in\mathbb{N}\cup\{0\}$ that 
\begin{equation}\label{essentialcond}
    L_{1}\left(x^{m}g(x)-\kappa H^{'}_{0,1}(x)\right)\in S^{+}_{m+1},
\end{equation}
with $\kappa$ satisfying 
$
    \left\langle x^{m}g(x),H^{'}_{0,1}(x)\right\rangle=\kappa\norm{H^{'}_{0,1}}_{L^{2}_{x}}^{2}.
$
To simplify our notation, we denote $h(x)=x^{m}g(x)-kH^{'}_{0,1}(x).$ From Lemma \ref{firstinvert}, $L_{1}\left(x^{m}g(x)-kH^{'}_{0,1}(x)\right)$ is well defined, so it is only necessary to prove \eqref{essentialcond} by induction on $m\in \mathbb{N}\cup\{0\}.$ 
We also observe that we can apply a change of variable $z(x)=e^{\sqrt{2}x}$ to rewrite the ordinary differential equation
\begin{equation}\label{odeprincipal}
{-} f^{''}(x)+ U^{''}(H_{0,1}(x))f(x)=h(x)
\end{equation}
as
\begin{equation}\label{analyticode}
    -2z^{2}\frac{d^{2}F_{0}(z)}{dz^{2}}-2z\frac{dF_{0}(z)}{dz}+\left(2+E(z)\right)F_{0}(z)=H(z), 
\end{equation}
where $F_{0}(e^{\sqrt{2}x})=f(x),\,H(e^{\sqrt{2}x})=h(x)$  and 
\begin{equation*}
E:\{z\in \mathbb{C}|\,-1<\I{(z)}<1\}\to \mathbb{C}
\end{equation*}
is the analytic function 
\begin{equation*}
E(z)=-24\frac{z^{2}}{1+z^{2}}+30\frac{z^{4}}{\left(1+z^{2}\right)^{2}},
\end{equation*}
because of the following identity
$ U^{''}(H_{0,1}(x))=2-24\frac{e^{2\sqrt{2}x}}{(1+e^{2\sqrt{2}x})}+30\frac{e^{4\sqrt{2}x}}{(1+e^{2\sqrt{2}x})^{2}}.$
\par We also recall that the operator $L$ defined in \eqref{operatorL} satisfies $L\left(H^{'}_{0,1}\right)=0$ and $H^{'}_{0,1}(x)=\frac{\sqrt{2}e^{\sqrt{2}x}}{(1+e^{2\sqrt{2}x})^{\frac{3}{2}}}.$ Also, using the method of variation of parameters, we have that the real function
\begin{equation}\label{c(x)}
    c(x)=\frac{1-e^{-2\sqrt{2}x}}{4\sqrt{2}}+\frac{3x}{2}+\frac{3(e^{2\sqrt{2}x}-1)}{4\sqrt{2}}+\frac{e^{4\sqrt{2}x}-1}{8\sqrt{2}},
\end{equation}
satisfies $L\left(c(x)H^{'}_{0,1}(x)\right)=0.$ In conclusion, from the Picard–Lindelöf Theorem, we deduce that
\begin{equation}\label{kernel}
  L^{-1}\{0\}=\left\{\left[c_{1}\left(\frac{-e^{-2\sqrt{2}x}}{4\sqrt{2}}+\frac{3x}{2}+\frac{3e^{2\sqrt{2}x}}{4\sqrt{2}}+\frac{e^{4\sqrt{2}x}}{8\sqrt{2}}\right)+c_{2}\right]\frac{e^{\sqrt{2}x}}{(1+e^{2\sqrt{2}x})^{\frac{3}{2}}}\Bigg|\,c_{1},\,c_{2}\in\mathbb{R}\right\}.
\end{equation}
Moreover, we can verify that $c(x)H^{'}_{0,1}(x)$ satisfies
\begin{equation*}
\int_{0}^{+\infty}c(x)^{2}H^{'}_{0,1}(x)^{2}\,dx=+\infty,\,\int_{-\infty}^{0}c(x)^{2}H^{'}_{0,1}(x)^{2}\,dx=+\infty,
\end{equation*}
from which we deduce with identity \eqref{kernel} that
\begin{equation}\label{verygood}
     L^{-1}\{0\}\cap L^{2}_{x}(\mathbb{R}_{\leq -1})=L^{-1}\{0\}\cap L^{2}_{x}(\mathbb{R})=\left\{c_{1}H^{'}_{0,1}(x)|\,c_{1}\in\mathbb{R}\right\}.
\end{equation}
\par In conclusion, from Theorem \ref{firstinvert} and identity \eqref{verygood}, we deduce that if $h\in Ort\left(H^{'}_{0,1}\right),$ $f\in L^{2}_{x}(\mathbb{R}_{\leq -1})$ and $- f(x)+U^{''}(H_{0,1}(x))f(x)=h(x)$ for all $x\in\mathbb{R}$, then there exists a constant $\kappa_{1}\in\mathbb{R}$ such that $L_{1}(h)(x)-f(x)=\kappa_{1}H^{'}_{0,1}(x)$ for all $x\in\mathbb{R}.$ So, to prove Lemma \ref{secondinvert} it is enough to find one $f\in S^{+}_{m+1}$ such that $L(f)(x)=h(x).$\\
\textbf{Case ($m=0.$)}
If $h\in S^{+}_{0},$ there exist an analytic function
\begin{equation*}
    H:\{z\in \mathbb{C}|\,-1<\I{(z)}<1\}\to\mathbb{C},
\end{equation*}
and a sequence $(h_{k})_{k\in\mathbb{N}}$ such that $H(z)=\sum_{k=0}^{+\infty}h_{k}z^{2k+1}$ for any $z\in\mathbb{D}$ and $h(x)=H\left(e^{\sqrt{2}x}\right)$ for all $x\in\mathbb{R}.$
We are going to construct a sequence $(c_{k})_{k\in\mathbb{N}\cup\{0\}}$ such that there exists a solution $f\in S^{+}_{1}\cap \mathscr{S}(\mathbb{R})$ of $L(f)(x)=h(x)$ satisfying for all $x<0$
\begin{equation}\label{ff}
    f(x)=c_{0}xH^{'}_{0,1}(x)+\sum_{k=0}^{+\infty}c_{k}e^{(2k+1)\sqrt{2}x}.
\end{equation}
 \par First, since $L(H^{'}_{0,1})(x)=0,$ we have for any smooth function $g(x)$ that
\begin{equation}\label{initialL}
    L(g)(x)=-2c_{0} H^{''}_{0,1}(x)-\frac{d^{2}}{dx^{2}}\left[g(x)-c_{0}xH^{'}_{0,1}(x)\right]+ U^{''}(H_{0,1}(x))\left[g(x)-c_{0}x H^{''}_{0,1}(x)\right].
\end{equation}
Next, if $(c_{k})_{k\in\mathbb{N}}$ is a real sequence such that the function $F_{1}(z)=\sum_{k=0}^{+\infty}c_{k}z^{2k+1}$ is analytic in the open unitary disk $\mathbb{D},$ then the chain rule of derivative implies for any $x<0$ that 
\begin{equation}\label{aa}
    \frac{d F_{1}(e^{\sqrt{2}x})}{dx}=\sqrt{2}\sum_{k=0}^{+\infty}c_{k}(2k+1)e^{(2k+1)\sqrt{2}x},\,
    \frac{d^{2} F_{1}(e^{\sqrt{2}x})}{dx^{2}}=2\sum_{k=0}^{+\infty}c_{k}(2k+1)^{2}e^{(2k+1)\sqrt{2}x}.
\end{equation}
We also denote the analytic expansion of $E(z)$ in the open complex unitary disk as
\begin{equation}\label{anal1}
    E(z)=\sum_{k=1}^{+\infty} p_{k}z^{2k},
\end{equation}
and since $H^{''}_{0,1}=2H_{0,1}-8H_{0,1}^{3}+6H_{0,1}^{5}\in S^{+},$ we have for $x<0$ that
$ H^{''}_{0,1}(x)=2e^{\sqrt{2}x}+\sum_{k=1}^{+\infty}u_{k}e^{(2k+1)\sqrt{2}x},$
with $\mathcal{U}(z)=\sum_{k=1}^{+\infty}u_{k}z^{k}$ analytic in $\mathbb{D}.$
\par Moreover, using identity \eqref{initialL}, we would obtain that if $L(g)=h,$ 
\begin{equation*}
    g(x)=c_{0}xH^{'}_{0,1}(x)+\sum_{k=0}^{{+}\infty} c_{k} e^{(2k+1)\sqrt{2}x} \text{, for any $x<0,$}
\end{equation*}
and $\limsup_{k\to+\infty} \vert c_{k} \vert^{\frac{1}{k}}\leq 1,$ then $\left(c_{k}\right)_{k\in\mathbb{N}\cup\{0\}}$ should satisfy the following equations: 
\begin{equation}\label{recurrence}
   \begin{cases}
  {-} 4c_{0}=h_{0},\\
\left(2-2(2k+1)^{2}\right)c_{k}=\left[h_{k}+2c_{0}u_{k}-\sum_{j+m=k,\,j\geq 1}c_{m}p_{j}\right]\text{, for any $k\geq 1.$}
    \end{cases}
\end{equation}
From now on, we consider the sequence $(c_{k})_{k\in\mathbb{N}\cup\{0\}}$ to be the unique solution of the linear recurrence \eqref{recurrence}.
 Clearly, for any $0<\epsilon<1,$ we have that $\lim_{k\to+\infty}|c_{k}|\epsilon^{k}=0,$ which implies 
 \begin{equation*}
 \limsup_{k\to+\infty}\vert c_{k} \vert^{\frac{1}{k}}\leq 1. 
 \end{equation*}
 Otherwise, $(c_{k}\epsilon^{\frac{k}{2}})_{k\in\mathbb{N}}$ would be an unbounded sequence and there would be a subsequence $(c_{k_{j}})_{j\in\mathbb{N}},$ so that $\vert c_{l}\vert \epsilon^{\frac{l}{2}} <\vert c_{k_{j}}\vert \epsilon^{\frac{k_{j}}{2}}$ for all $0\leq l<k_{j},$ from which we would obtain with the identities $\lim_{n\to+\infty}p_{n}\epsilon^{\frac{n}{2}}=\lim_{n\to+\infty} h_{n} \epsilon^{\frac{n}{2}}=\lim_{n\to+\infty} u_{n}\epsilon^{\frac{n}{2}}=0$ that
\begin{equation*}
    \epsilon^{\frac{k_{j}}{2}}|c_{k_{j}}|(2(2k_{j}+1)^{2}-2)\gg 2|c_{0}u_{k_{j}}|\epsilon^{\frac{k_{j}}{2}}+|h_{k_{j}}|\epsilon^{\frac{k_{j}}{2}}+2(k_{j}+1)|c_{k_{j}}|\epsilon^{\frac{k_{j}}{2}}\norm{(\epsilon^{\frac{j}{2}}p_{j})}_{L^{\infty}(\mathbb{N})},
\end{equation*}
but this estimate would contradict \eqref{recurrence}. So, we deduced that
\begin{equation*}
F_{1}(z)=\sum_{k=0}^{+\infty}c_{k}z^{2k+1}
\end{equation*}
is analytic in $\mathbb{D}.$ In conclusion, the recurrence \eqref{recurrence} implies that the function $f(x)$ denoted in \eqref{ff} satisfies $L(f)(x)=h(x)$ for all $x<0.$   
\par  Moreover, because $E(z),\,\frac{1}{z}$ are analytic in the simply connected regions 
\begin{align*}
    \mathcal{B}_{\delta,+}=&\left\{z\in \mathbb{C}\vert\, -1<\I(z)<1,\,\vert z\vert>\delta,\,\operatorname{Re}(z)>-\frac{4}{5}\delta \right\},\\ \mathcal{B}_{\delta,-}=&\left\{z\in \mathbb{C}\vert\, -1<\I(z)<1,\,\vert z\vert>\delta,\,\operatorname{Re}(z)<\frac{4}{5}\delta\right\}
\end{align*}
for any $0<\delta<1,$ we obtain, from $h\in S^{+}$ and the ordinary differential equation \eqref{analyticode}, the existence of a unique holomorphic function $F_{+}$ in the region  $\mathcal{B}_{\delta,+}$ which is a solution of \eqref{analyticode} and satisfies   $F_{1}(e^{\sqrt{2}x})+c_{0}xH^{'}_{0,1}(x)=F_{+}(e^{\sqrt{2}x})$ for all $e^{\sqrt{2}x}\in \mathcal{B}_{\delta,+}\cap\mathbb{D},$ see Chapter $3.7$ of \cite{ode}. By analogy, there exists a unique holomorphic function $F_{-}$ with domain $\mathcal{B}_{\delta,-}$ which is a solution of \eqref{analyticode} and satisfies $F_{1}(e^{\sqrt{2}x})+c_{0}xH^{'}_{0,1}(x)=F_{-}(e^{\sqrt{2}x})$ for all $e^{\sqrt{2}x}\in \mathcal{B}_{\delta,-}\cap \mathbb{D}.$ In conclusion, there exists a unique analytic function $F_{2}$ in the region $\mathcal{B}=\{z\in \mathbb{C}\vert\, -1<\I(z)<1\}$ such that $F_{2}(z)=F_{1}(z)$ for all $z\in \mathbb{D}$ and the real function
\begin{equation*}
    c_{0}xH^{'}_{0,1}(x)+F_{2}\left(e^{\sqrt{2}x}\right) \in L^{-1}\{h\}.
\end{equation*}
\par Indeed, from the recurrence relation \eqref{recurrence} and identities \eqref{initialL}, \eqref{aa}, \eqref{anal1}, we conclude that if   \begin{equation}\label{idtodo}
    f(x)=c_{0}xH^{'}_{0,1}(x)+F_{2}(e^{\sqrt{2}x}),
\end{equation} then $f(x) \in L^{2}_{x}(\mathbb{R}_{\leq -1}),$ and $L(f)(x)=h(x)$ for all $x\in\mathbb{R}.$ In conclusion, there exists $\tau\in\mathbb{R}$ such that
$L_{1}(h)(x)=f(x)-\tau H^{'}_{0,1}(x),$ and since $L_{1}(h)(x)\in \mathscr{S}(\mathbb{R}),$ identity \eqref{idtodo} implies that
$L_{1}(h)(x)$ is in $S^{+}_{1}.$\\ 
\\
\textbf{General case}($m\geq1.$)
 Based on the observation made in \eqref{essentialcond}, it suffices to check for any $g\in S^{+}\cap\mathscr{S}(\mathbb{R})$ that
\begin{equation}\label{tootooeasy}
   T_{m}(g)\coloneqq L_{1}\left(x^{m}g-\left\langle x^{m}g,\,H^{'}_{0,1}\right\rangle\frac{H^{'}_{0,1}}{\norm{H^{'}_{0,1}}_{L^{2}_{x}}^{2}}\right)\in S^{+}_{m+1},
\end{equation}
for all $m\in\mathbb{N}\cup\{0\}.$ 
Clearly, we checked \eqref{tootooeasy} when $m=0$ in the first case. Now, we assume that \eqref{tootooeasy} is true for all $m\in\mathbb{N}\cup\{0\}$ satisfying $0\leq m\leq M,$ for some number $M\in\mathbb{N}\cup\{0\}.$ From the inductive hypothesis, if $g\in S^{+}\cap\mathscr{S}(\mathbb{R}),$ then $T_{M}(g)\in S^{+}_{M+1},$ which implies the existence of a finite set of functions $\left(f_{m}\right)_{0\leq m\leq M+1}\subset S^{+}\cap \mathscr{S}(\mathbb{R})$ such that
\begin{equation}\label{trivialidentity}
    T_{M}(g)=\sum_{m=0}^{M+1}x^{m}f_{m}.
\end{equation}
\par Moreover, since $L\left(T_{M}(g)\right)\in S^{+}_{M},$ we derive from Lemma \ref{independent} and identity \eqref{trivialidentity} the identity
$x^{M+1}L(f_{M+1})(x)=0,$
which is possible only if $f_{M+1}=\sigma H^{'}_{0,1}$ for a real number $\sigma.$ Therefore, we have 
\begin{equation}\label{eqtM}
    T_{M}(g)(x)= \sigma x^{M+1}H^{'}_{0,1}(x)+\sum_{m=0}^{M}x^{m}f_{m}(x) \text{ for any $x\in\mathbb{R}.$}
\end{equation}
Consequently,
\begin{equation*}
    \frac{d}{dx}T_{M}(g)(x)-\sigma x^{M+1} H^{''}_{0,1}(x) \text{ is in $S^{+}_{M},$}
\end{equation*}
from which, using \eqref{tootooeasy} and identity $L(H^{'}_{0,1})(x)=0,$
we obtain that
\begin{equation*}
    -\frac{d^{2}}{dx^{2}}\Big[xT_{M}(g)(x)\Big]+ U^{''}(H_{0,1}(x))xT_{M}(g)(x)-\left[x^{M+1}g(x)-\tau_{M} xH^{'}_{0,1}(x)-2\sigma x^{M+1} H^{''}_{0,1}(x)\right] 
\end{equation*}
is in $S^{+}_{M},$ where
\begin{equation*}
    \tau_{M}=
   \frac{ \left\langle x^{M}g,\,H^{'}_{0,1}\right\rangle}{\norm{H^{'}_{0,1}}_{L^{2}_{x}}^{2}}.
\end{equation*}
\par  With identity $L(H^{'}_{0,1})(x)=0$ we also obtain that
\begin{multline*}
    -\frac{d^{2}}{dx^{2}}\left[-\frac{\sigma x^{M+2}H^{'}_{0,1}(x)}{M+2}\right]+U^{''}(H_{0,1}(x))\left[-\frac{\sigma x^{M+2}H^{'}_{0,1}(x)}{M+2}\right]\\=2\sigma x^{M+1} H^{''}_{0,1}(x)+\sigma(M+1)x^{M}H^{'}_{0,1}(x).
\end{multline*}
Therefore, using that $xH^{'}_{0,1}$ and $x^{M}H^{'}_{0,1}$ are in $S^{+}_{M},$ we deduce
\begin{equation*}
    L\left(xT_{M}(g)(x)+\frac{\sigma x^{M+2}H^{'}_{0,1}(x)}{M+2}\right)-x^{M+1}g(x)  \text{ is in $S^{+}_{M},$}
\end{equation*}
from which, for $\tau_{M+1}\norm{H^{'}_{0,1}}_{L^{2}_{x}}^{2}=\langle H^{'}_{0,1},x^{M+1}g(x)\rangle,$ we obtain that
\begin{equation}
    L_{1}\left(x^{M+1}g-\tau_{M+1}H^{'}_{0,1}\right)-\left[x T_{M}(g)+\frac{\sigma x^{M+2}H^{'}_{0,1}}{M+2}\right] \text{ is in $S^{+}_{M+1}.$}
\end{equation}
\par In conclusion, we obtain that \eqref{tootooeasy} is true for $m=M+1,$ so by induction, it is true for all $m\in\mathbb{N}\cup\{0\},$ so Lemma \ref{secondinvert} is true for all $m\in\mathbb{N}\cup\{0\}.$
\end{proof}
\section{Auxiliary estimates}\label{prpp}
In this section, we will prove useful lemmas, which will be used later to estimate $\frac{d^{l}}{dt^{l}}\Lambda(\phi_{k})(v,t,x)$ for all $k\in\mathbb{N}_{\geq 2}$ and $l\in\mathbb{N}\cup\{0\}.$ 
\par First, we can verify by induction that 
$
    \vert d^{(l)}(t)\vert \lesssim_{l} v^{l},
$
for any $l\in\mathbb{N},$ more precisely:
\begin{lemma}\label{dlemma}
The function $d_{v}(t)=\frac{1}{\sqrt{2}}\ln{\left(\frac{8}{v^{2}}\cosh{\left(\sqrt{2}vt\right)}^{2}\right)}$ satisfies $\norm{\dot d_{v}(t)}_{L^{\infty}(\mathbb{R})}=2v$ and
\begin{equation*}
    \vert d_{v}^{(l)}(t)\vert\lesssim_{l} v^{l}e^{-2\sqrt{2}\vert t\vert v} \text{ for all natural number $l\geq 2.$} 
\end{equation*}
\end{lemma}
\begin{proof}
From $\dot d_{v}(t)=2v\tanh{\left(\sqrt{2}vt\right)},$ we obtain that $\norm{\dot d_{v}(t)}_{L^{\infty}(\mathbb{R})}=2v.$ Moreover, because of
\begin{equation*}
    \ddot d_{v}(t)=16\sqrt{2}e^{-\sqrt{2}d_{v}(t)}=2\sqrt{2}v^{2}\sech{(\sqrt{2}vt)}^{2},
\end{equation*}
Lemma \ref{dlemma} is also true for $l=2.$  Since for any $l\in\mathbb{N}\cup\{0\},$ 
$
  \left\vert  \frac{d^{l}}{dx^{l}}\sech{(x)}\right\vert\lesssim_{l} e^{-\vert x\vert},
$
we can deduce the result of Lemma \ref{dlemma} from the chain rule of derivative. 
\end{proof}
From now on, we denote the function $w_{0}:\mathbb{R}^{2}\to\mathbb{R}$ by
\begin{equation}\label{w0}
    w_{0}(t,x)=\frac{x-\frac{d_{v}(t)}{2}}{\sqrt{1-\frac{\dot d_{v}(t)^{2}}{4}}}
\end{equation}
We will use several times the function $w_{0}(t,x)$ in the next sections too.
Clearly, from \eqref{odedvv}, for any $h\in C^{\infty}(\mathbb{R}),$ we have the following identity
\begin{equation}\label{dfw0}
\frac{\partial }{\partial t}\left[h\left(w_{0}(t,x)\right)\right]={-}\frac{\dot d_{v}(t)}{\sqrt{4-\dot d_{v}(t)^{2}}} h^{'}(w_{0}(t,x))+\frac{16\sqrt{2}\dot d_{v}(t)}{4-\dot d_{v}(t)^{2}}e^{-\sqrt{2}d_{v}(t)}w_{0}(t,x) h^{'}(w_{0}(t,x)).
\end{equation}
 Moreover, we have: 
\begin{lemma}\label{geraldt}
If $f\in S^{+}_{m}$ for some $m\in \mathbb{N}\cup\{0\},$ then for any numbers $l,\,k_{1}\in\mathbb{N}\cup\{0\}$ the function
$f(w_{0}(t,x))$
satisfies the following estimate
\begin{equation}\label{popo}
   \norm{\frac{\partial^{l}f(w_{0}(t,x))}{\partial t^{l}}}_{H^{k_{1}}_{x}(\mathbb{R})}\lesssim_{l,k_{1}} v^{l}\norm{(1+\vert x\vert)^{l}\max_{0\leq j\leq k_{1}+l}\left \vert f^{(j)}(x) \right \vert}_{L^{2}_{x}(\mathbb{R})}\lesssim_{f,l,k_{1}} v^{l}.
\end{equation}
%$\frac{d^{k}f(w_{0}(t,x))}{dt^{k}}=\sum_{i}f^{(l)}\left(w_{0}(t,x)\right)\prod_{j=1}^{l}\partial^{n_{j}+1}_{t}w_{0}(t,x)$ such that all $n_{j}\in\mathbb{N}$ and $l+\sum_{j=1}^{l}n_{j}=k$
% \frac{d^{k}}{dt^{k}}\left[f(t)g(t)\right]=\sum_{l=0}^{k}\begin{pmatrix} k
%l
%\end{pmatrix}f^{(l)}(t)g^{(k-l)}(t)
More precisely, there exist a natural number $N_{l}$ and a finite set $\{(h_{i,l},p_{i,l,v})\in S^{+}_{m+l}\times C^{\infty}\vert\, 1\leq i\leq N_{l}\}$ such that
\begin{equation}\label{induct0}
    \frac{\partial^{l}f(w_{0}(t,x))}{\partial t^{l}}=\sum_{i=1}^{N_{l}}h_{i,l}(w_{0}(t,x))p_{i,l,v}(t),
\end{equation}
and, for all $1\leq i\leq N_{l}$ and all $k_{1}\in\mathbb{N}\cup\{0\}$
\begin{equation}\label{induct1}
    \left\vert \frac{\partial^{k_{1}}h_{i,l}(x)}{\partial x^{k_{1}}} \right\vert\lesssim_{k_{1},l} (1+\vert x \vert)^{l}\max_{0\leq j\leq k_{1}+l}\left\vert f^{(j)} (x)\right\vert,\quad
    \norm{\frac{\partial^{k_{1}}p_{i,l,v}(t)}{\partial t^{k_{1}}}}_{L^{\infty}(\mathbb{R})}\lesssim_{l,k_{1}} v^{k_{1}+l} \text{, if $0<v\ll 1.$}
\end{equation}
Furthermore, if $l$ is odd, then $p_{i,l,v}(t)$ is an odd function for all $1\leq i\leq N_{l},$ otherwise they are all even functions.
\end{lemma}
\begin{proof}
  We will prove by induction for all $l\in\mathbb{N}\cup\{0\}$ the existence of $N_{l}\in\mathbb{N}$ such that \eqref{induct0} holds, 
and for all $1\leq i\leq N_{l}$ $h_{i,l}\in S^{+}_{m+l},\,p_{i,l,v}(t)=({-}1)^{l}p_{i,l,v}({-}t)$
and they also satisfy
\eqref{induct1}
for all $1\leq i\leq N_{l}$ and all $k_{1}\in\mathbb{N}\cup\{0\}.$   
\par  The case $l=0$ is trivial, we can just take the unitary set $\{(f,1)\}\subset S^{+}_{m}\times C^{\infty}.$ So, there exists $l_{0}\in\mathbb{N}\cup\{0\}$ such that Lemma \ref{geraldt}
is true for all $l\in\mathbb{N}\cup\{0\}$ satisfying $0\leq l\leq l_{0}.$  In conclusion, using the identity \eqref{induct0} for $l=l_{0}$ and identity \eqref{dfw0}, we obtain that
\begin{align}\nonumber
    \frac{\partial^{l_{0}+1}f(w_{0}(t,x))}{\partial t^{l_{0}+1}}\\ \label{k01}=&\sum_{i=1}^{N_{l_{0}}}\frac{\partial h_{i,l_{0}}(w_{0}(t,x))}{\partial t}p_{i,l_{0}}(t)+h_{i,l_{0}}(w_{0}(t,x))\dot p_{i,l_{0},v}(t)\\\nonumber
    =&\sum_{i=1}^{N_{l_{0}}}{-} h^{'}_{i,l_{0}}(w_{0}(t,x))\frac{\dot d_{v}(t) p_{i,l_{0},v}(t)}{\sqrt{4-\dot d_{v}(t)^{2}}}\\ \nonumber &{+}\sum_{i=1}^{N_{l_{0}}}w_{0}(t,x) h^{'}_{i,l_{0}}(w_{0}(t,x))\frac{16\sqrt{2}\dot d_{v}(t) p_{i,l_{0},v}(t)}{4-\dot d_{v}(t)^{2}}e^{-\sqrt{2}d_{v}(t)}\\\label{k012}&{+}\sum_{i=1}^{N_{l_{0}}}h_{i,l_{0}}(w_{0}(t,x))\dot p_{i,l_{0},v}(t).
\end{align}
Since $h_{i,l_{0}}\in S^{+}_{m+l_{0}},$ we deduce that $ h^{'}_{i,l_{0}}\in S^{+}_{m+l_{0}}\subset S^{+}_{m+l_{0}+1}$ and $x h^{'}_{i,l_{0}}\in S^{+}_{m+l_{0}+1}.$ 
Also, we recall that the function $d_{v}(t)=\frac{1}{\sqrt{2}}\ln{\left(\frac{8}{v^{2}}\cosh{\left(\sqrt{2}vt\right)}^{2}\right)}$ satisfies for all $l\in\mathbb{N}$
\begin{equation}\label{ostost}
    \norm{d_{v}^{(l)}(t)}_{L^{\infty}(\mathbb{R})}\lesssim_{l} v^{l} \text{, if $0<v\ll 1.$}
\end{equation}
Moreover, for any $m\in\mathbb{N}\cup\{0\}$ and any $0<\delta<1,$
\begin{equation}\label{impooo}
    \norm{\frac{d^{m}}{d\theta^{m}}\left[\frac{1}{\sqrt{1-\theta^{2}}}\right]}_{L^{\infty}_{\theta}(\vert\theta\vert<\delta)}<+\infty,
\end{equation}
because the function $q(\theta)=\left(1-\theta^{2}\right)^{-\frac{1}{2}}$ is smooth in the set $\{\theta\vert\,\vert\theta\vert \leq \delta\}.$
Therefore, since the functions $h_{i,l_{0}}$ and $p_{i,l_{0},v}$ satisfy \eqref{induct1}, using the chain rule of derivative, estimate \eqref{ostost} and \eqref{k01}, \eqref{k012}, we deduce the existence of a natural number $N_{l_{0}+1}$ such that
\begin{equation*}
    F_{l_{0}+1,t}(x)=\sum_{i=1}^{N_{l_{0}+1}}h_{i,l_{0}+1}(x)p_{i,l_{0}+1,v}(t),
\end{equation*}
 and, for all $1\leq i\leq N_{l_{0}+1},$ the functions $h_{i,l_{0}+1},\,p_{i,l_{0}+1,v}$ satisfy \eqref{induct1}, $h_{i,l_{0}+1}\in S^{+}_{m+l_{0}+1}.$ More precisely, from \eqref{k01} and \eqref{k012}, we choose $N_{l_{0}+1}=3N_{l_{0}}$ and 
\begin{equation*}
    \begin{cases}
    \left(h_{i,l_{0}+1}(x),p_{i,l_{0+1},v}(t)\right)=\left(- h^{'}_{i,l_{0}}(x),\frac{\dot d_{v}(t) p_{i,l_{0},v}(t)}{\sqrt{4-\dot d(t)^{2}}}\right) \text{, if $\leq i\leq N_{l_{0}},$}\\
     \left(h_{i,l_{0}+1}(x),p_{i,l_{0+1},v}(t)\right)=\left (x h^{'}_{i-N_{l_{0}},l_{0}}(x),\frac{16\sqrt{2}\dot d_{v}(t) p_{i,l_{0},v}(t)}{1-\dot d_{v}(t)^{2}}e^{{-}\sqrt{2}d_{v}(t)} \right) \text{, if $N_{l_{0}}+1\leq i\leq 2N_{l_{0}},$}\\
     \left(h_{i,l_{0}+1}(x),p_{i,l_{0+1},v}(t)\right)=\left(h_{i-2N_{l_{0}}}(x),\dot p_{i,l_{0},v}(t)\right) \text{, if $2N_{l_{0}}+1\leq i\leq 3N_{l_{0}},$}
    \end{cases}
\end{equation*}
for all $(t,x)\in\mathbb{R}^{2}.$
In conclusion, \eqref{induct0}, \eqref{induct1} are true for $l=l_{0}+1$ and $h_{i,l_{0}+1}\in S^{+}_{m+l_{0}+1}$ for all $1\leq i\leq N_{l_{0}+1}.$ Finally, since $d_{v}(t)$ is an even smooth function and, for any $1\leq i\leq N_{l_{0}},$ $p_{i,l_{0},v}(t)=({-}1)^{l_{0}}p_{i,l_{0},v}({-}t),$ then, from \eqref{k01} and \eqref{k012}, we deduce that $p_{i,l_{0}+1,v}(t)=({-}1)^{l_{0}+1}p_{i,l_{0}+1,v}({-}t)$ for all $1\leq i\leq N_{l_{0}+1}.$ In conclusion, the statement of Lemma \ref{geraldt} is true for $l=l_{0}+1,$ and so, by induction, it is true for all $l\in\mathbb{N}\cup\{0\}.$   \end{proof}
\begin{remark}\label{perturbt}
If $\gamma:(0,1)\times \mathbb{R}\to\mathbb{R}$  is a continuous function such that $\gamma(v,\cdot):\mathbb{R}\to\mathbb{R}$ is smooth for all $0<v<1$ and    
\begin{equation*}
    \left\vert\frac{\partial^{l}\gamma(v,t)}{\partial t^{l}}\right\vert\lesssim_{l} v^{l} \text{ for any $l\in\mathbb{N}\cup\{0\}$ and all $t\in\mathbb{R},$ if $0<v\ll 1,$}
\end{equation*}
then for any Schwartz function $f$ and
\begin{equation*}
    \omega(t,x)=\frac{x-\frac{d_{v}(t)}{2}+\gamma(v,t)}{\sqrt{1-\frac{\dot d_{v}(t)^{2}}{4}}},
\end{equation*}
we obtain similarly to the proof of Lemma \ref{geraldt} that if $v\ll 1,$ then, for all $l\in\mathbb{N}\cup\{0\}$ and $k_{1}\in\mathbb{N},$ 
\begin{equation}\label{grt}
     \norm{\frac{\partial^{l}f(\omega(t,x))}{\partial t^{l}}}_{H^{k_{1}}_{x}(\mathbb{R})}\lesssim_{l,k_{1}} v^{l}\norm{(1+\vert x\vert)^{l}\max_{0\leq j\leq k_{1}+l}\left \vert f^{(j)}(x) \right \vert}_{L^{2}_{x}(\mathbb{R})}\lesssim_{f,l,k_{1}} v^{l}.
\end{equation}
Furthermore, if $f\in C^{\infty}(\mathbb{R})$ and $\dot f \in\mathscr{S}(\mathbb{R}),$ for example $f=H_{0,1},$ then
from identity 
\begin{align*}
    \frac{\partial}{\partial t}f\left(\omega(t,x)\right)=&\left[\partial_{t}\gamma(v,t)-\frac{\dot d_{v}(t)}{2}\right]\frac{1}{\sqrt{1-\frac{\dot d_{v}(t)^{2}}{4}}} f^{'}\left(\omega(t,x)\right)\\ &{+}\sqrt{1-\frac{\dot d_{v}(t)^{2}}{4}}\frac{d}{dt}\left[\frac{1}{\sqrt{1-\frac{\dot d_{v}(t)^{2}}{4}}}\right]\omega (t,x)\ f^{'}(\omega(t,x)),
\end{align*}
we obtain from the same argument above any $l,\,k_{1}\in\mathbb{N}$ that estimate \eqref{grt} holds.
We are going to use this remark later in Section \ref{sub32}.
\end{remark}
\begin{lemma}\label{porrataylor}
For any $n_{1}\in\mathbb{N}$ and $n_{2}\in \mathbb{N}\cup\{0\},$ let $r:(0,1)\times \mathbb{R}\to\mathbb{R}$ be a function such that $r_{v}\coloneqq r(v,\cdot):\mathbb{R}\to\mathbb{R}$ is smooth for all $0<v<1$ and satisfies for $n_{1}\in\mathbb{N},\,n_{2}\in\mathbb{N}\cup\{0\}$ 
\begin{equation*}
   \left\vert\frac{d^{l}r_{v}(t)}{dt^{l}}\right\vert\lesssim_{l}v^{n_{1}+l}\ln{\left(\frac{1}{v}\right)}^{n_{2}},
\end{equation*}
for all $l\in\mathbb{N}\cup\{0\},$ if $0<v\ll 1.$
 Then, for any $s\geq 1$ and any smooth function $h:\mathbb{R}\to\mathbb{R}$ such that $\dot h\in \mathscr{S}(\mathbb{R}),$ we have
\begin{align*}
    \norm{\frac{\partial^{l}}{\partial t^{l}}\left[h\left(w_{0}(t,x+r_{v}(t))\right)-h\left(w_{0}(t,x)\right)\right]}_{H^{s}_{x}(\mathbb{R})}&\lesssim_{h,s,l} v^{n_{1}+l}\ln{\left(\frac{1}{v}\right)}^{n_{2}},\\
    \norm{\frac{\partial^{l}}{\partial t^{l}}\left[h\left(w_{0}(t,x+r_{v}(t))\right)-h\left(w_{0}(t,x)\right)-\frac{r_{v}(t)}{\sqrt{1-\frac{\dot d(t)^{2}}{4}}} h^{'}\left(w_{0}(t,x)\right)\right]}_{H^{s}_{x}(\mathbb{R})}&\lesssim_{h,s,l} v^{2n_{1}+l}\ln{\left(\frac{1}{v}\right)}^{2n_{2}},
\end{align*}
if $0<v\ll 1.$
\end{lemma}
\begin{proof}[Proof of Lemma \ref{porrataylor}]
From the Fundamental Theorem of Calculus and the definition of $w_{0}(t,x)$, we have
\begin{equation}\label{taylor1}
    h\left(w_{0}(t,x+r_{v}(t))\right)-h\left(w_{0}(t,x)\right)=\frac{r_{v}(t)}{\sqrt{1-\frac{\dot d(t)^{2}}{4}}}\bigintsss_{0}^{1} h^{'}\left(\frac{x-\frac{d_{v}(t)}{2}+\theta r_{v}(t)}{\sqrt{1-\frac{\dot d_{v}(t)^{2}}{4}}}\right)\,d\theta,
\end{equation}
and
\begin{multline}\label{taylor2}
   h\left(w_{0}(t,x+r_{v}(t))\right)-h\left(w_{0}(t,x)\right)-\frac{r_{v}(t)}{\sqrt{1-\frac{\dot d(t)^{2}}{4}}} h^{'}\left(w_{0}(t,x)\right)\\
  =\frac{r_{v}(t)^{2}}{1-\frac{\dot d(t)^{2}}{4}} \bigintsss_{0}^{1} h^{''}\left(\frac{x-\frac{d_{v}(t)}{2}+\theta r_{v}(t)}{\sqrt{1-\frac{\dot d(t)^{2}}{4}}}\right)(1-\theta)\,d\theta.
\end{multline}
From Remark \ref{perturbt}, we obtain for all $0\leq \theta\leq 1$ and $0<v\ll 1$ that
\begin{equation*}
   \norm{\frac{\partial^{l}}{\partial t^{l}}\left[ h^{''}\left(w_{0}(t,x+\theta r_{v}(t))\right)\right]}_{H^{s}_{x}(\mathbb{R})} +\norm{\frac{\partial^{l}}{\partial t^{l}}\left[h^{'}\left(w_{0}(t,x+\theta r_{v}(t))\right)\right]}_{H^{s}_{x}(\mathbb{R})}\lesssim_{l} v^{l} \text{ for all $l\in\mathbb{N}\cup\{0\}.$}
\end{equation*}
In conclusion, from identities \eqref{taylor1} and \eqref{taylor2}, we conclude Lemma \ref{porrataylor} using the product rule of derivative and Lemma \ref{dlemma}.
\end{proof}
\begin{lemma}\label{explemma}
For any $n_{1}\in\mathbb{N}$ and $n_{2}\in \mathbb{N}\cup\{0\}$ and for $0<v<1,$ let $r_{v}:\mathbb{R}\to\mathbb{R}$ being a smooth function satisfying
\begin{equation*}
   \left\vert\frac{d^{l}r_{v}(t)}{dt^{l}}\right\vert\lesssim_{l}v^{n_{1}+l}\ln{\left(\frac{1}{v}\right)}^{n_{2}} \text{, if $0<v\ll 1$}
\end{equation*}
for all $l\in\mathbb{N}\cup\{0\}.$ For any $m_{1}\in\mathbb{N},\, m_{2}\in\mathbb{N}\cup\{0\}$ and $m_{3}\in \mathbb{Z},$ let $p:(0,1)\times\mathbb{R}\to\mathbb{R}$ be the function
\begin{equation*}
    p(v,t)=\left(1-\frac{\dot d_{v}(t)^{2}}{4}\right)^{\frac{m_{3}}{2}}\exp\left(\frac{{-}m_{1}\sqrt{2}(d_{v}(t)+r_{v}(t))}{\left(1-\frac{\dot d_{v}(t)^{2}}{4}\right)^{\frac{m_{2}}{2}}}\right)-e^{{-}m_{1}\sqrt{2}d_{v}(t)}. 
\end{equation*}
If $m_{2}=m_{3}=0$ and $0<v\ll 1,$ then for all $l\in\mathbb{N}\cup\{0\}$
\begin{equation}\label{T1}
    \left\vert \frac{\partial^{l}}{\partial t^{l}}p(v,t) \right\vert\lesssim_{m_{1},l} v^{2m_{1}+n_{1}+l}\left(\ln{\left(\frac{1}{v}\right)}+\vert t\vert v\right)^{n_{2}}e^{{-}2\sqrt{2}\vert t\vert v}.
\end{equation}
If $m_{3}\neq 0,\,m_{2}=0$ and $0<v\ll 1,$ then for all $l\in\mathbb{N}\cup\{0\}$
\begin{equation}\label{T2}
     \left\vert \frac{\partial^{l}}{\partial t^{l}}p(v,t) \right\vert\lesssim_{l,m_{1}}\max\left( v^{2m_{1}+2+l},v^{2m_{1}+n_{1}+l}\left(\ln{\left(\frac{1}{v}\right)}+\vert t\vert v\right)^{n_{2}}\right)e^{{-}2\sqrt{2}\vert t\vert v}.
\end{equation}
If $m_{2}\neq 0$ and $0<v\ll 1,$ then for all $l\in\mathbb{N}\cup\{0\}$
\begin{equation}\label{T3}
     \left\vert \frac{\partial^{l}}{\partial t^{l}}p(v,t) \right\vert\lesssim_{l,m_{1}} \max{\left(v^{2m_{1}+2+l}\left(\vert t\vert v+\ln{\left(\frac{1}{v}\right)}\right),v^{2m_{1}+n_{1}+l}\left(\vert t\vert v+\ln{\left(\frac{1}{v}\right)}\right)^{n_{2}}\right)}e^{-2\sqrt{2}\vert t\vert v}.
\end{equation}
\end{lemma}
\begin{proof}
If $m_{2}=m_{3}=0,$ then, from the Fundamental Theorem of Calculus, we have
\begin{equation*}
    p(v,t)={-}\sqrt{2}m_{1}\int_{0}^{1}e^{{-}\sqrt{2}m_{1}(d_{v}(t)+\theta r_{v}(t))}r_{v}(t)\,d\theta.
\end{equation*}
So, for all $l\in\mathbb{N}\cup\{0\},$ we deduce that
\begin{align*}
    \frac{\partial^{l}}{\partial t^{l}}p(v,t)=&{-}\sqrt{2}\int_{0}^{1}\frac{d^{l}}{d t^{l}}\left[e^{{-}\sqrt{2}m_{1}(d_{v}(t)+\theta r_{v}(t))}r_{v}(t)\right]\,d\theta=\\&{-}\sqrt{2}\int_{0}^{1}\sum_{j=0}^{l}
    \begin{pmatrix}
        l\\
        j
    \end{pmatrix}\frac{d^{j}}{d t^{j}}\left[e^{{-}\sqrt{2}m_{1}(d_{v}(t)+\theta r_{v}(t))}\right]\frac{d^{l-j}}{dt^{l-j}}r_{v}(t)\,d\theta.
\end{align*}
From the hypothesis of $r_{v}(t),$ $\left\vert e^{{-}\theta\sqrt{2}r_{v}(t)}\right\vert\lesssim 1$ for any $0\leq\theta\leq 1$ if $0<v\ll 1,$ so, using the chain and product rules, we obtain that
\begin{equation}\label{rr(v)}
    \left\vert \frac{d^{l}}{d t^{l}} e^{{-} \sqrt{2}\theta r_{v}(t)}\right\vert\lesssim_{l} v^{l} \text{, for any $l\in\mathbb{N}$ and any $0\leq \theta\leq 1.$ }
\end{equation}
Moreover, since $8^{m_{1}}e^{-\sqrt{2}m_{1}d_{v}(t)}=v^{2m_{1}}\sech{\left(\sqrt{2}v t\right)}^{2m_{1}}=\ddot d_{v}(t)^{m_{1}}2^{{-}\frac{3m_{1}}{2}},$ we have from Lemma \ref{dlemma} and the product rule of derivative that
\begin{equation}\label{dexp}
   \left\vert \frac{d^{l}}{dt^{l}}e^{{-}\sqrt{2}m_{1}d_{v}(t)}\right\vert\lesssim_{l,m_{1}} v^{2m_{1}+l}e^{{-}2\sqrt{2}m_{1}\vert t\vert v}\lesssim v^{2m_{1}+l}e^{{-}2\sqrt{2}\vert t\vert v} \text{, for all $l\in\mathbb{N}\cup\{0\},$}
\end{equation}
if $0<v\ll 1.$ 
In conclusion, using the hypotheses satisfied by the function $r_{v}$ and the estimates above, we obtain inequality \eqref{T1}.
\par If $m_{3}\neq 0$ and $m_{2}=0,$ we have
\begin{align*}
    p(v,t)=&\left(1-\frac{\dot d_{v}(t)^{2}}{4}\right)^{\frac{m_{3}}{2}}e^{{-}m_{1}\sqrt{2}(d_{v}(t)+r_{v}(t))}-e^{{-}m_{1}\sqrt{2}d_{v}(t)}\\=&e^{{-}m_{1}\sqrt{2}(d_{v}(t)+r_{v}(t))}-e^{{-}m_{1}\sqrt{2}d_{v}(t)}
    +e^{{-}m_{1}\sqrt{2}(d_{v}(t)+r_{v}(t))}\left[\left(1-\frac{\dot d_{v}(t)^{2}}{4}\right)^{\frac{m_{3}}{2}}-1\right].
\end{align*}
From the argument above, we have for any $l\in\mathbb{N}\cup\{0\}$ that
\begin{equation*}
    \left\vert\frac{d^{l}}{dt^{l}}\left[e^{{-}m_{1}\sqrt{2}\left(d_{v}(t)+r_{v}(t)\right)}-e^{{-}\sqrt{2}m_{1}d_{v}(t)}\right]\right\vert\lesssim_{l,m_{1}} v^{2m_{1}+n_{1}+l}\left(\ln{\left(\frac{1}{v}\right)}+\vert t\vert v\right)^{n_{2}}e^{{-}2\sqrt{2}\vert t\vert v},
\end{equation*}
if $0<v\ll 1.$
Moreover, since the function $q:({-}1,1)\to\mathbb{R}$ denoted by 
\begin{equation*}
    q(x)=(1-x^{2})^{\frac{m_{3}}{2}}-1
\end{equation*}
is smooth when restricted to the compact set $[{-1}+\delta,1-\delta]$ for any $0<\delta<1,$ we conclude from Lemma \ref{dlemma}, the chain rule and product rule of derivative that if $0<v\ll 1,$ then
\begin{equation}\label{fracder}
    \left\vert\frac{d^{l}}{dt^{l}}\left[\left(1-\frac{\dot d_{v}(t)^{2}}{4}\right)^{\frac{m_{3}}{2}}-1\right]\right\vert\lesssim_{l,m_{3}}  v^{2+l} \text{, for all $l\in\mathbb{N}\cup\{0\}.$}
\end{equation}
In conclusion, using the product rule of derivative, we obtain \eqref{T2} from \eqref{rr(v)}, \eqref{dexp} and \eqref{fracder}
\par Finally, we will prove now \eqref{T3}. Clearly, using estimates \eqref{T2} and \eqref{fracder}, if the function
\begin{equation*}
    p_{1}(v,t)=\exp\left(\frac{{-}m_{1}\sqrt{2}(d_{v}(t)+r_{v}(t))}{\left(1-\frac{\dot d_{v}(t)^{2}}{4}\right)^{\frac{m_{2}}{2}}}\right)-e^{{-}m_{1}\sqrt{2}(d_{v}(t)+r_{v}(t))}
\end{equation*}
satisfies, for any $m_{1},\,m_{2}\in\mathbb{N}$ and $0<v\ll 1,$ the following inequality
\begin{equation}\label{lorentzexp}
     \left\vert \frac{\partial^{l}}{\partial t^{l}}p_{1}(v,t) \right\vert\lesssim_{l,m_{1},m_{2}} v^{2m_{1}+2+l}\left(\vert t\vert v+\ln{\left(\frac{1}{v}\right)}\right)e^{-2\sqrt{2}\vert t\vert v} \text{, for all $l\in\mathbb{N}\cup\{0\},$}
\end{equation}
then \eqref{T3} is true.
From the Fundamental Theorem of Calculus, we obtain
\begin{align*}
   p_{1}(v,t)=&{-}m_{1}\sqrt{2}(r_{v}(t)+d_{v}(t))\int_{0}^{1}\exp\left({-}m_{1}\sqrt{2}(d_{v}(t)+r_{v}(t))\left[1-\theta+\frac{\theta}{\left(1-\frac{\dot d_{v}(t)^{2}}{4}\right)^{\frac{m_{2}}{2}}}\right]\right)\,d \theta
    \\&{+}\frac{m_{1}\sqrt{2}(r_{v}(t)+d_{v}(t))}{\left(1-\frac{\dot d_{v}(t)^{2}}{4}\right)^{\frac{m_{2}}{2}}}\int_{0}^{1}\exp\left({-}m_{1}\sqrt{2}(d_{v}(t)+r_{v}(t))\left[1-\theta+\frac{\theta}{\left(1-\frac{\dot d_{v}(t)^{2}}{4}\right)^{\frac{m_{2}}{2}}}\right]\right)\,d \theta.
\end{align*}
Similarly to the proof of \eqref{fracder}, we deduce if $0<v\ll 1,$ then
\begin{equation}\label{fracm}
    \left\vert\frac{d^{l}}{dt^{l}}\left(1-\frac{\dot d_{v}(t)^{2}}{4}\right)^{{-}\frac{m_{2}}{2}}\right\vert\lesssim_{l,m_{2}} v^{2+l}e^{{-}2\sqrt{2}\vert t\vert v} \text{ for all $l,\,m_{2}\in\mathbb{N}.$}
\end{equation}
Moreover, from the hypotheses satisfied by $r_{v},$ we obtain using Lemma \ref{dlemma}, estimate \eqref{fracm} and the product rule of derivative that if $0<v\ll 1,$ then
\begin{equation*}
    \left\vert \frac{d^{l}}{dt^{l}}\exp\left({-}m_{1}\sqrt{2}r_{v}(t)\left[1-\theta+\frac{\theta}{\left(1-\frac{\dot d(t)^{2}}{4}\right)^{m_{2}}}\right]\right)\right\vert\lesssim_{l,m_{2},m_{1}}v^{l} \text{, for all $0\leq \theta\leq 1$}
\end{equation*}
and $l\in\mathbb{N}\cup\{0\}.$ Similarly, since $e^{-\sqrt{2}d_{v}(t)}\lesssim v^{2}\ll 1$ and $d_{v}(t)\lesssim v\vert t\vert +\ln{\left(\frac{1}{v}\right)} $ we obtain from Lemma \ref{dlemma}, estimate \eqref{fracm} and the product rule of derivative that
\begin{equation*}
    \left\vert \frac{d^{l}}{dt^{l}}\exp\left({-}m_{1}\sqrt{2}d_{v}(t)\theta\left[\frac{1}{\left(1-\frac{\dot d(t)^{2}}{4}\right)^{m_{2}}}-1\right]\right)\right\vert\lesssim_{l,m_{2},m_{1}}v^{l} \text{, for all $0\leq \theta\leq 1$}
\end{equation*}
and $l\in\mathbb{N}\cup\{0\}.$ In conclusion, using \eqref{dexp}, Lemma \ref{dlemma}, and the product rule of derivative, we obtain \eqref{lorentzexp}, and so \eqref{T3} is true.
\end{proof}
\begin{lemma}\label{interactionsize}
Let $m,\,n\in\mathbb{N}\cup\{0\},\,f\in S^{+},\,g\in S^{-}.$ Let $\gamma:(0,1)\times\mathbb{R}\to\mathbb{R}$ be a  continuous function satisfying for any $l\in\mathbb{N}\cup\{0\}$
\begin{equation}\label{decaygamma}
    \left\vert\frac{d^{l}}{dt^{l}}\gamma(v,t)\right\vert\lesssim_{l} v^{l} \text{ if $0<v\ll 1.$}
\end{equation}
Then, for
\begin{equation}\label{formuw}
    \omega (t,x)=w_{0}(t,x+\gamma(v,t))=\frac{x-\frac{d_{v}(t)}{2}+\gamma(v,t)}{\sqrt{1-\frac{\dot d_{v}(t)^{2}}{4}}},
\end{equation} if $0<v\ll 1,$ then, for any $s\geq 0$ and all $l\in\mathbb{N}\cup\{0\},$ we have
\begin{multline}\label{gelint}
    \norm{\frac{\partial^{l}}{\partial t^{l}}\left[\omega(t,x)^{m}f\left(\omega(t,x)\right)\omega(t,{-}x)^{n}g\left({-}\omega(t,{-}x)\right)\right]}_{H^{s}_{x}(\mathbb{R})}\\ \lesssim_{s,l,m,n} v^{2\min\left(\val_{+}(f),\val_{-}(g)\right)+l}\left(\ln{\left(\frac{1}{v}\right)}+\vert t\vert v\right)^{m+n} e^{{-}2\sqrt{2}\vert t\vert v }.
\end{multline}
 Furthermore, if $0<v\ll 1,\ \,\val_{+}(f)+1\neq val_{-}(g)$ and $\val_{-}(g)+1\neq \val_{+}(f),$ then for all $l\in\mathbb{N}\cup\{0\}$
\begin{multline}\label{<>+1}
\left\vert \frac{d^{l}}{dt^{l}}\left\langle \omega(t,x)^{m}f\left(\omega(t,x)\right)\omega(t,{-}x)^{n} g\left({-}\omega(t,{-}x)\right),H^{'}_{0,1}\left(\omega(t,x)\right) \right\rangle \right\vert\\ \lesssim_{l,m,n} v^{l+2\min\left(\val_{+}(f)+1,\val_{-}(g)\right)}\left(\vert t\vert v+\ln{\left(\frac{1}{v}\right)}\right)^{m+n} e^{{-}2\sqrt{2}\vert t\vert v}, 
\end{multline}
and
\begin{multline}\label{<>-1}
    \left\vert \frac{d^{l}}{dt^{l}}\left\langle \omega(t,x)^{m}f\left(\omega(t,x)\right)\omega(t,{-}x)^{n} g\left({-}\omega(t,{-}x)\right),H^{'}_{0,1}\left(\omega(t,{-}x)\right) \right\rangle \right\vert\\\lesssim_{l,m,n} v^{l+2\min\left(\val_{+}(f),\val_{-}(g)+1\right)}\left(\vert t\vert v+\ln{\left(\frac{1}{v}\right)}\right)^{m+n} e^{{-}2\sqrt{2}\vert t\vert v}. 
\end{multline}
 Otherwise, if $0<v\ll 1$ and $\val_{+}(f)+1=\val_{-}(g),$ then for any $l\in\mathbb{N}\cup\{0\}$
 \begin{multline}\label{<>-}
     \left\vert \frac{d^{l}}{dt^{l}}\left\langle \omega(t,x)^{m}f\left(\omega(t,x)\right)\omega(t,{-}x)^{n} g\left({-}\omega(t,{-}x)\right),H^{'}_{0,1}\left(\omega(t,x)\right) \right\rangle \right\vert\\ \lesssim_{l,m,n}v^{l+2\val_{-}(g)}\left(\vert t\vert v+\ln{\left(\frac{1}{v}\right)}\right)^{m+n+1} e^{{-}2\sqrt{2}\vert t\vert v}.
 \end{multline}
If $0<v\ll 1$ and $\val_{+}(f)=\val_{-}(g)+1,$ then
\begin{multline}\label{<>+}
     \left\vert \frac{d^{l}}{dt^{l}}\left\langle \omega(t,x)^{m}f\left(\omega(t,x)\right)\omega(t,{-}x)^{n} g\left({-}\omega(t,{-}x)\right),H^{'}_{0,1}\left(\omega(t,{-}x)\right) \right\rangle \right\vert\\ \lesssim_{l,m,n}v^{l+2\val_{+}(f)}\left(\vert t\vert v+\ln{\left(\frac{1}{v}\right)}\right)^{m+n+1} e^{{-}2\sqrt{2}\vert t\vert v}.
 \end{multline}
\end{lemma}
\begin{proof}[Proof of the Lemma \ref{interactionsize}]
\par Fist, by an argument of analogy, it is enough to prove that estimate \eqref{gelint} is true for the case $\val_{+}(f)=2w_{1}+1>2w_{2}=\val_{-}(g),$ such that $w_{1},\,w_{2}\in\mathbb{N}.$ From the Separation Lemma and Corollary \ref{unique1}, we have that there exists functions $h_{1}\in S^{+}\cap\mathscr{S}(\mathbb{R}),\,f_{1}\in S^{+},\,g_{1}\in S^{-}$ with either $f_{1}$ or $g_{1}\in \mathscr{S}(\mathbb{R})$ such that
\begin{equation*}
    f(x-\zeta)g(x)=h_{1}(x-\zeta)e^{{-}2\sqrt{2}w_{2}\zeta}+e^{{-}2\sqrt{2}w_{2}\zeta}f_{1}(x-\zeta)g_{1}(x),
\end{equation*}
for all $x\in\mathbb{R}$ and $\zeta\geq 1.$ Moreover, after a change of variables, we obtain that
\begin{multline*}
    \omega(t,x)^{m}\omega(t,{-x})^{n}f\left(\omega(t,x)\right)g\left({-}\omega(t,{-}x)\right)\\
    \begin{aligned}
    =&\omega(t,x)^{m}\omega(t,{-x})^{n}h_{1}\left(\omega(t,x)\right)\exp\left(\frac{{-}2w_{2}\sqrt{2}(d_{v}(t)-2\gamma(v,t))}{\sqrt{1-\frac{\dot d_{v}(t)^{2}}{4}}}\right)\\&{+}
     \omega(t,x)^{m}\omega(t,{-x})^{n}\exp\left(\frac{{-}2\sqrt{2}w_{2}(d_{v}(t)-2\gamma(v,t))}{\sqrt{1-\frac{\dot d_{v}(t)^{2}}{4}}}\right)f_{1}\left(\omega(t,x)\right)g_{1}\left({-}\omega(t,{-}x)\right).
\end{aligned}
\end{multline*}
\par Since $f_{1}$ or $g_{1} \in \mathscr{S}(\mathbb{R})$ and $f_{1}\in S^{+},\,g_{1}\in S^{-},$ then either $x^{k_{1}}f_{1}(x)\in S^{+}_{\infty}\subset \mathscr{S}(\mathbb{R})$ for all $k_{1}\in\mathbb{N}\cup\{0\}$ or $x^{k_{1}}g_{1}(x)\in S^{-}_{\infty}\subset \mathscr{S}(\mathbb{R})$ for all $k_{1}\in\mathbb{N}\cup\{0\}.$ Consequently, from Remark \ref{perturbt}, if $0<v\ll 1, $ then for all $l,\,k_{1}\in\mathbb{N}\cup\{0\}$ and $s\geq 1$ either
\begin{equation*}
    \norm{\frac{\partial^{l}}{\partial t^{l}}\left[\omega(t,x)^{k_{1}}f_{1}\left(\omega(t,x)\right)\right]}_{H^{s}_{x}}\lesssim_{s,l,k_{1}} v^{l}, 
\end{equation*}
or
\begin{equation*}
    \norm{\frac{\partial^{l}}{\partial t^{l}}\left[\omega(t,{-}x)^{k_{1}}g_{1}\left({-}\omega(t,{-}x)\right)\right]}_{H^{s}_{x}}\lesssim_{s,l,k_{1}} v^{l}. 
\end{equation*}
From Lemma \ref{dlemma}, if $0<v\ll 1,$ then we also have the following estimate for all $l\in\mathbb{N}$
\begin{equation}\label{mmma}
   \left\vert \frac{d^{l}}{dt^{l}}\left[\frac{1}{\sqrt{4-\dot d_{v}(t)^{2}}}\right]\right\vert\lesssim_{l} v^{2+l}e^{{-}2\sqrt{2}\vert t\vert v},
\end{equation}
which with the hypotheses satisfied by $\gamma(v,t)$ and the product rule of derivative implies that if $0<v\ll 1,$ then
\begin{equation}\label{fracestt}
   \left\vert \frac{d^{l}}{dt^{l}}\left[\frac{d_{v}(t)-2\gamma(v,t)}{\sqrt{4-\dot d_{v}(t)^{2}}}\right]\right\vert\lesssim_{l} v^{l} \text{, for all $l\in\mathbb{N}.$}
\end{equation}
\par Therefore, since 
\begin{equation}\label{symeq}
   \omega(t,x)+\omega(t,{-}x)=\frac{{-}2d_{v}(t)+4\gamma(v,t)}{\sqrt{4-\dot d_{v}(t)^{2}}}, 
\end{equation}
we deduce, from the product rule of derivative, the hypotheses \eqref{decaygamma} and Cauchy-Schwarz inequality, that if $0<v\ll 1,$ then for all $k_{1},\,l\in\mathbb{N}\cup\{0\}$ 
\begin{multline}\label{intelet}
\norm{\frac{\partial^{l}}{\partial t^{l}}\Big[\omega(t,x)^{m}\omega(t,{-x})^{n}f_{1}\left(\omega(t,x)\right)g_{1}\left({-}\omega(t,{-}x)\right)\Big]}_{H^{k_{1}}_{x}(\mathbb{R})}\\ \lesssim_{m,n,l,f_{1},g_{1}} d_{v}(t)^{\max(m,n)} v^{l}.
\end{multline}
\par Moreover, since $d_{v}(t)=\frac{1}{\sqrt{2}}\ln{\left(\frac{8}{v^{2}}\cosh{\left(\sqrt{2}vt\right)}^{2}\right)}$ and $\sup_{t\in\mathbb{R}}\vert\gamma(v,t)\vert\lesssim 1$ when $0<v\ll 1,$ then
\begin{equation}\label{mu}
   \mu(t)= \exp\left(\frac{{-}2\sqrt{2}w_{2}(d_{v}(t)-2\gamma(v,t))}{\sqrt{1-\frac{\dot d(t)^{2}}{4}}}\right)\lesssim v^{4w_{2}}
    \sech{\left(\sqrt{2}vt\right)}^{2} \text{, if $0<v\ll 1$,}
\end{equation}
from which with estimate \eqref{fracestt} implies for all $l\in\mathbb{N}\cup\{0\}$ that if $0<v\ll 1,$ then $\left\vert \frac{d^{l}\mu(t)}{dt^{l}}\right\vert\lesssim_{l,w_{2}}v^{l+4w_{2}}e^{{-}2\sqrt{2}v\vert t\vert}.$ In conclusion, estimate \eqref{intelet}  implies, if $0<v\ll 1,$ that for all $m,\,n,\,l \in\mathbb{N}\cup\{0\}$ we have
\begin{multline}\label{intelet2}
 \norm{\frac{\partial^{l}}{\partial t^{l}}\Big[\mu(t) \omega(t,x)^{m}\omega(t,{-x})^{n}f_{1}\left(\omega(t,x)\right)g_{1}\left({-}\omega(t,{-}x)\right)\Big]}_{H^{k_{1}}_{x}(\mathbb{R})}\\ \lesssim_{w_{1},w_{2},m,n,l} \left(\vert t\vert v+\ln{\left(\frac{1}{v}\right)}\right)^{\max\left(m,n\right)} v^{4w_{2}+l}e^{{-}2\sqrt{2}\vert t\vert v} . \end{multline}
Finally, since $h_{1}\in S^{+}\cap\mathscr{S}(\mathbb{R}),$ we have $\norm{x^{k_{1}}h_{1}(x)}_{H^{s}_{x}}\lesssim_{s,k_{1}} 1$ for all $s,\,k_{1}\in\mathbb{N}\cup\{0\}.$ Therefore, Remark \ref{perturbt} implies for $0<v\ll 1$ that
\begin{equation*}
    \norm{\frac{\partial^{l}}{\partial t^{l}}\left[\omega(t,x)^{k_{1}}h_{1}\left(\omega(t,x)\right)\right]}_{H^{s}_{x}}\lesssim_{s,k_{1},l} v^{l} \text{ for all $k_{1}\in\mathbb{N}\cup\{0\}.$ }
\end{equation*}
In conclusion, if $0<v\ll 1,$ then, using \eqref{symeq}, \eqref{mu} and Lemma \ref{dlemma}, we obtain from the product rule of derivative for any $k_{1},\,l\in\mathbb{N}\cup\{0\}$ that
\begin{equation*}
    \norm{\frac{\partial^{l}}{\partial t^{l}}\left[\omega(t,x)^{m}\omega(t,{-x})^{n}h_{1}\left(\omega(t,x)\right)\mu(t)\right]}_{H^{k_{1}}_{x}} \lesssim_{l,k_{1}} v^{l+4w_{2}}\left(\vert t\vert v+\ln{\left(\frac{1}{v}\right)}\right)^{n}e^{{-} 2\sqrt{2}\vert t\vert v},
\end{equation*}
from which with inequality \eqref{intelet2} and triangle inequality, we deduce \eqref{gelint}.
\par From now on, we will prove estimates \eqref{<>+1}, \eqref{<>-1}, \eqref{<>-} and \eqref{<>+}. Indeed, it is sufficient to demonstrate estimates \eqref{<>+1} and \eqref{<>+}, because the proof of the other inequalities follows from a similar argument. 
\par Since $\omega$ satisfies \eqref{formuw}, we obtain after a change of variables that
\begin{multline}\label{dot identity}
    \left\langle \omega(t,x)^{m}f\left(\omega(t,x)\right)\omega(t,{-}x)^{n}g\left({-}\omega(t,{-}x)\right),H^{'}_{0,1}\left(\omega(t,x)\right) \right\rangle\\=
   \sqrt{1-\frac{\dot d_{v}(t)^{2}}{4}} \left\langle x^{m}f(x) \left({-}x-\frac{d_{v}(t)-2\gamma(v,t)}{\sqrt{1-\frac{\dot d_{v}(t)^{2}}{4}}}\right)^{n}g\left(x+\frac{d_{v}(t)-2\gamma(v,t)}{\sqrt{1-\frac{\dot d_{v}(t)^{2}}{4}}}\right),H^{'}_{0,1}\left(x\right) \right\rangle.
\end{multline}
Moreover, since $f\in S^{+}$ and $g \in S^{-},$ we deduce from Lemma \ref{interactt} for any $\zeta\geq 1$ and all $l\in\mathbb{N}\cup\{0\}$ that if $\val_{+}(f)+1\neq \val_{-}(g),$ then
\begin{equation*}
     \left\vert\frac{d^{l}}{d\zeta^{l}}\left\langle x^{m}f(x) (x+\zeta)^{n}g(x+\zeta),H^{'}_{0,1}(x)\right\rangle\right\vert\lesssim_{l}  
    \zeta^{m+n}\max\left(e^{{-}\sqrt{2}(1+\val_{+}(f))\zeta},e^{{-}\sqrt{2}\val_{-}(g)\zeta}\right),
\end{equation*}
otherwise
\begin{equation*}
     \left\vert\frac{d^{l}}{d\zeta^{l}}\left\langle x^{m}f(x) (x+\zeta)^{n}g(x+\zeta),H^{'}_{0,1}(x)\right\rangle\right\vert\lesssim_{l}  
    \zeta^{m+n+1}e^{{-}\sqrt{2}\val_{-}(g)\zeta}.
\end{equation*}
\par Finally, from Lemma \ref{dlemma} and the hypotheses satisfied by $\gamma(v,t),$ we obtain if $0<v\ll 1,$ then for all $l\in\mathbb{N}$
\begin{equation*}
    \left\vert\frac{d^{l}}{dt^{l}}\left[\frac{2d_{v}(t)-4\gamma(v,t)}{\sqrt{4-\dot d_{v}(t)^{2}}}\right]\right\vert +\left\vert \frac{d^{l}}{dt^{l}}\sqrt{1-\frac{\dot d_{v}(t)^{2}}{4}} \right\vert\lesssim_{l} v^{l}.  
\end{equation*}
 In conclusion, the product rule of derivative and identity \eqref{dot identity} imply \eqref{<>+1}, if $\val_{+}(f)+1\neq \val_{-}(g),$ otherwise they imply \eqref{<>+}. The inequalities \eqref{<>-1} and \eqref{<>-} can be demonstrated using an analogous argument. 
\end{proof}
 \begin{remark}\label{elint}
For any $m,\,n\in\mathbb{N}\cup\{0\}$ and any $f\in S^{+}_{m},\,g\in S^{+}_{n},$ we have the following identity
 \begin{equation*}
     H(v,t)=\left\langle f\left(w(t,x)\right),g(\left(w(t,{-}x)\right))\right\rangle=\sqrt{1-\frac{\dot d_{v}(t)^{2}}{4}}\left\langle f\left(x-\frac{d_{v}(t)-2\gamma(v,t)}{\sqrt{1-\frac{\dot d_{v}(t)^{2}}{4}}}\right),g({-}x) \right\rangle.
 \end{equation*}
 So, we can use Lemmas \ref{interactt}, \ref{dlemma} and Remark \ref{perturbt} to conclude that if $0<v\ll 1,$ then, for all $l\in\mathbb{N}\cup\{0\},$
 \begin{equation*}
    \left\vert \frac{\partial^{l}}{\partial t^{l}}H(v,t)\right\vert\lesssim_{l} v^{2+l}\left(\vert t\vert v+\ln{\left(\frac{1}{v}\right)}\right)^{m+n+1}e^{{-}2\sqrt{2}v\vert t\vert}.
 \end{equation*}
 \end{remark}

 \section{Approximate solution for $k=2$}\label{sub31}
\par  First, we recall the function $w_{0}:\mathbb{R}^{2}\to\mathbb{R}$ denoted by
\begin{equation*}
    w_{0}(t,x)=\frac{x-\frac{d_{v}(t)}{2}}{\sqrt{1-\frac{\dot d_{v}(t)^{2}}{4}}},
\end{equation*}
and the function $\varphi_{2,0}$ denoted by
\begin{equation}\label{phi2(0)}
    \varphi_{2,0}(t,x)=H_{0,1}(w_{0}(t,x))-H_{0,1}(w_{0}(t,{-}x))+e^{-\sqrt{2}d_{v}(t)}\left[\mathcal{G}(w_{0}(t,x))-\mathcal{G}(w_{0}(t,{-}x))\right].
\end{equation}
 Using the results of the last section, we will estimate with high precision precision the function $\Lambda(\varphi_{2,0})(t,x).$ 
 We recall the identity \eqref{idddG} satisfied by the function $\mathcal{G}$
\begin{equation}\label{Gprop}
    -\frac{d^{2}}{dx^{2}} \mathcal{G}(x)+ U^{''}(H_{0,1}(x))\mathcal{G}(x)=\left[-24 H_{0,1}(x)^{2}+30 H_{0,1}(x)^{4}\right]e^{-\sqrt{2}x}+8\sqrt{2} H^{'}_{0,1}(x).
\end{equation} 
Since $ H^{'}_{0,1}$ is in the kernel of the linear self-adjoint operator $-\frac{d^{2}}{dx^{2}}+ U^{''}(H_{0,1}),$ we can deduce using \eqref{Gprop} that 
\begin{equation}\label{oioi}
    \int_{\mathbb{R}}\left[24 H_{0,1}(x)^{2}-30 H_{0,1}(x)^{4}\right]e^{-\sqrt{2}x} H^{'}_{0,1}(x)\,dx=8\sqrt{2}\norm{ H^{'}_{0,1}}_{L^{2}_{x}}^{2}=4.
\end{equation}
  \par The main objective of this section is to demonstrate the following theorem.
\begin{theorem}\label{k=2}
Let $d_{v}(t)$ be the function defined in \eqref{defd(t)}. If $0<v\ll 1,$ then there is a smooth even function $r_{v}(t)$ and a value $e(v)$ such that for the following approximate solution
\begin{multline}\label{phi2}
    \varphi_{2}(t,x)=H_{0,1}\left(w_{0}(t,x+r_{v}(t))\right)-H_{0,1}\left(w_{0}(t,{-}x+r_{v}(t))\right)\\+e^{-\sqrt{2}d_{v}(t)}\left[\mathcal{G}\left(w_{0}(t,x+r_{v}(t))\right)-\mathcal{G}\left(w_{0}(t,{-}x+r_{v}(t))\right)\right],
\end{multline}
$\phi_{2}(v,t,x)=\varphi_{2}(t+e(v),x)$ satisfies the conclusion of Theorem \ref{approximated theorem} for $k=2$ and there exists $n_{2}\in\mathbb{N}$ such that if $0<v\ll 1,$ then
\begin{equation}\label{orthodec2}
   \left\vert \frac{d^{l}}{dt^{l}}\left\langle \Lambda(\varphi_{2})(t,x), H^{'}_{0,1}(w_{0}(t,\pm x+r_{v}(t))) \right\rangle\right\vert \lesssim_{l} v^{6+l}\left(\vert t\vert v+\ln{\left(\frac{1}{v}\right)}\right)^{n_{2}+1}e^{{-}2\sqrt{2}\vert t\vert v},
\end{equation}
for all $l\in\mathbb{N}\cup\{0\}.$ Furthermore, if $v\ll 1,$ the function $r_{v}$ satisfies
\begin{equation*}
    \norm{r_{v}}_{L^{\infty}(\mathbb{R})}\lesssim v^{2}\ln{\left(\frac{1}{v^{2}}\right)},\,\, \left\vert \frac{d^{l}}{dt^{l}}r_{v}(t)\ \right\vert \lesssim_{l} v^{2+l}\left[\ln{\left(\frac{1}{v}\right)}+\vert t \vert v\right]e^{{-}2\sqrt{2}\vert t\vert v},
\end{equation*}
for all $l\in\mathbb{N}.$
\end{theorem}
 \par From now on, we say that any two smooth functions $f,\,g:\mathbb{R}^{2}\to\mathbb{R}$ satisfy the relation of equivalence $f\cong_{6} g$ if, and only if, for any $s\geq 0$ and $l\in\mathbb{N}\cup\{0\}$ there exists a positive number $C(s,l)$ such that
\begin{equation*}
    \norm{\frac{\partial^{l}}{\partial t^{l}}\left[f(t,x)-g(t,x)\right]}_{H^{s}_{x}}\leq C(s,l)v^{6+l}\left[\vert t\vert v+\ln{\left(\frac{1}{v^{2}}\right)}\right]^{2}e^{{-}2\sqrt{2}\vert t\vert v},
\end{equation*}
for all $t\in\mathbb{R}.$ With the objective of simplifying our reasoning, we also say in this section that two functions $f,\,g$ are equivalent if, and only if,
$f\cong_{6} g$ and that a function $f$ is negligible if $f\cong_{6} 0.$
 \subsection{Estimate of non interacting terms of $\Lambda(\varphi_{2,0})(t,x).$}
In this subsection, we only focus on estimating the main terms of order $O(v^{2})$ of  
\begin{equation*}
    \Lambda\Big(H_{0,1}\left(w_{0}\left(t,x\right)\right)+e^{{-}\sqrt{2}d_{v}(t)}\mathcal{G}\left(w_{0}(t,x)\right)\Big).
\end{equation*}
\begin{lemma}\label{dt2kink}
For any $(t,x)\in\mathbb{R}^{2},$ we have
\begin{equation}\label{lambdah}
   \Lambda\left(H_{0,1}(w_{0}(t,x))\right)= -\frac{8\sqrt{2}e^{-\sqrt{2}d_{v}(t)}}{\sqrt{1-\frac{\dot d_{v}(t)^{2}}{4}}} H^{'}_{0,1}(w_{0}(t,x))+R_{1,v}\left(t,w_{0}(t,x)\right), 
\end{equation}
where the function $R_{1,v}(t,x)$  in \eqref{lambdah} is a finite sum of functions 
$
  h_{i}(x)p_{i,v}(t),
$
with $h_{i}(x)\in S^{+}_{2}$ and $p_{i,v}(t)\in C^{\infty}(\mathbb{R})$ being an even function satisfying $\lvert\frac{d^{l}p_{i,v}(t)}{dt^{l}}\rvert\lesssim_{l}v^{4+l}e^{-2\sqrt{2}v\lvert t \rvert}$ for all $l\in \mathbb{N}\cup \{0\}.$ Furthermore, for any $s\geq 1$ and any $l\in\mathbb{N}\cup\{0\},$
\begin{equation}\label{kinkw}
    \norm{\frac{\partial^{l}}{\partial t^{l}}R_{1,v}(t,w_{0}(t,x))}_{H^{s}_{x}}\lesssim_{s,l} v^{4+l}e^{{-}2\sqrt{2}v\vert t\vert}.
\end{equation}
\end{lemma}
\begin{proof}
First, from identities
$\frac{\partial^{2}}{\partial x^{2}}H_{0,1}(w_{0}(t,x))=\frac{1}{\left(1-\frac{\dot d(t)^{2}}{4}\right)} H^{''}_{0,1}(w_{0}(t,x)),\, H^{''}_{0,1}(x)= U^{'}(H_{0,1}(x)),$
we have the following equation
\begin{equation}\label{0trivial}
    \frac{\dot d_{v}(t)^{2}}{4-\dot d_{v}(t)^{2}} H^{''}_{0,1}(w_{0}(t,x))-\frac{\partial^{2}}{\partial x^{2}}H_{0,1}(w_{0}(t,x))+ U^{'}(H_{0,1}(w_{0}(t,x)))=0.
\end{equation}
 Next, from \eqref{dfw0}, we have
\begin{equation}\label{fdp1}
    \frac{\partial}{\partial t}H_{0,1}(w_{0}(t,x))=-\frac{\dot d_{v}(t)}{\sqrt{4-\dot d_{v}(t)^{2}}}H^{'}_{0,1}(w_{0}(t,x))+\frac{16\sqrt{2}\dot d_{v}(t)}{4-\dot d_{v}(t)^{2}}e^{-\sqrt{2}d_{v}(t)}w_{0}(t,x) H^{'}_{0,1}(w_{0}(t,x)).
\end{equation}
\par Now, since $f(x)=x H^{'}_{0,1}(x)\in S^{+}_{1}$ and $d_{v}(t)$ is an even smooth function, we obtain from Lemma \ref{geraldt} the existence of $N_{1}\in\mathbb{N}$ satisfying
\begin{equation}\label{fdp11}
   \frac{\partial}{\partial t}\left[w_{0}(t,x) H^{'}_{0,1}(w_{0}(t,x))\right]=\sum_{i=1}^{N_{1}}h_{i,1}\left(w_{0}(t,x)\right)p_{i,1,v}(t),
\end{equation}
 such that for all $1\leq i\leq N_{1}$ $h_{i,1}\in S^{+}_{2},\,p_{i,1,v}\in C^{\infty}(\mathbb{R})$ and $p_{i,1,v}$ is an odd function. They also satisfy for any $1\leq i\leq N_{1}$
\begin{equation}\label{fdp2}
    \left\vert \frac{d^{l}h_{i,1}(x)}{d x^{l}} \right\vert\lesssim_{l} (1+\vert x \vert)^{2}\max_{0\leq j\leq 2+l}\left\vert f^{(j)} (x)\right\vert,\quad
    \norm{\frac{d^{l}p_{i,1,v}(t)}{dt^{l}}}_{L^{\infty}(\mathbb{R})}\lesssim_{l} v^{1+l} \text{ for all $l\in\mathbb{N}\cup\{0\}$}.
\end{equation}
In conclusion, we have 
\begin{multline}\label{fdp6}
    \frac{\partial}{\partial t}\left[\frac{16\sqrt{2}\dot d_{v}(t)e^{-\sqrt{2}d_{v}(t)}}{4-\dot d_{v}(t)^{2}}w_{0}(t,x) H^{'}_{0,1}(w_{0}(t,x))\right]\\=w_{0}(t,x) H^{'}_{0,1}(w_{0}(t,x))\frac{d}{d t}\left[\frac{16\sqrt{2}\dot d_{v}(t)e^{-\sqrt{2}d_{v}(t)}}{4-\dot d_{v}(t)^{2}}\right]\\{+}\frac{16\sqrt{2}\dot d_{v}(t)e^{-\sqrt{2}d_{v}(t)}}{4-\dot d_{v}(t)^{2}}\sum_{i=1}^{N_{1}}h_{i,1}(w_{0}(t,x))p_{i,1,v}(t). 
\end{multline}
\par Moreover, from estimate \eqref{impooo} and Lemma \ref{dlemma}, we deduce using the chain and product rule of derivative that
\begin{equation*}
    \left \vert \frac{d^{l}}{d t^{l}}\left[\frac{\dot d_{v}(t)e^{-\sqrt{2}d_{v}(t)}}{4-\dot d_{v}(t)^{2}}\right]\right\vert\lesssim_{l} v^{3+l}e^{-2\sqrt{2}\vert t\vert v} \text{ for all $l\in\mathbb{N}\cup\{0\}.$ }
\end{equation*}
\par Next, since $w_{0}(t,x)=\left(x-\frac{d_{v}(t)}{2}\right)\left(1-\frac{\dot d_{v}(t)^{2}}{4}\right)^{{-}\frac{1}{2}}$ and $\ddot d_{v}(t)=16\sqrt{2}e^{{-}\sqrt{2}d_{v}(t)},$ we can verify that
\begin{align}\label{fdp33}
    \frac{\partial}{\partial t}\left[-\frac{\dot d_{v}(t)}{\sqrt{4-\dot d_{v}(t)^{2}}} H^{'}_{0,1}(w_{0}(t,x))\right] =&{-}\frac{8\sqrt{2}e^{-\sqrt{2}d_{v}(t)}}{\sqrt{1-\frac{\dot d_{v}(t)^{2}}{4}}} H^{'}_{0,1}(w_{0}(t,x))+\frac{\dot d_{v}(t)^{2}}{4-\dot d_{v}(t)^{2}} H^{''}_{0,1}(w_{0}(t,x))\\ \nonumber
    &{-}\dot d_{v}(t)\left[\frac{d}{dt}\left(4-\dot d_{v}(t)^{2}\right)^{{-}\frac{1}{2}}\right]H^{'}_{0,1}(w_{0}(t,x))\\ \nonumber &{-}\dot d_{v}(t)\left[\frac{d}{dt}\left(4-\dot d_{v}(t)^{2}\right)^{{-}\frac{1}{2}}\right]w_{0}(t,x) H^{''}_{0,1}(w_{0}(t,x)).
\end{align}
Using the Remarks \ref{dsremark} and \ref{basisS}, we can verify that $ H^{'}_{0,1}\in S^{+}\cap\mathscr{S}(\mathbb{R}),$ and $x H^{''}_{0,1}\in S^{+}_{1}.$
We also recall the estimate \eqref{mmma} which is given by
\begin{equation*}
    \left\vert\frac{d^{l}}{dt^{l}}\left(1-\frac{\dot d_{v}(t)^{2}}{4}\right)^{-\frac{1}{2}}\right\vert\lesssim_{l} v^{2+l} e^{-2\sqrt{2}\vert t\vert v} \text{ for all $l\in\mathbb{N}.$}
\end{equation*}
In conclusion, from Lemmas \ref{dlemma}, \ref{geraldt}, identities \eqref{0trivial}, \eqref{fdp1}, equations \eqref{fdp6} and \eqref{fdp33}, we obtain that $R_{1,v}(t,x)$ is a finite sum of functions $p_{i,v}(t)h_{i}(x)$ with $h_{i}\in S^{+}_{2}$ and $p_{i,v}$ satisfying
\begin{equation*}
    \left\vert\frac{d^{l}}{dt^{l}}p_{i,v}(t)\right\vert\lesssim_{l} v^{4+l}e^{{-}2\sqrt{2}\vert t\vert v} \text{ for all $l\in\mathbb{N}\cup\{0\}.$}
\end{equation*}
Since $d_{v}(t)$ is an even function, equations \eqref{fdp6} and \eqref{fdp33} imply that all the functions $p_{i,v}(t)$ are also even.
Estimate \eqref{kinkw} is obtained from Lemma \ref{dlemma} and the product rule of derivative on time applied to each function
$
   p_{i,v}(t)h_{i}\left(w_{0}(t,x)\right). 
$
\end{proof}
\begin{lemma}\label{LambdaG}
The function $\mathcal{G}$ defined in \eqref{G(x)} satisfies the following identity
\begin{multline}\label{lambdag}
    \left[\frac{\partial^{2}}{\partial t^{2}}-\frac{\partial^{2}}{\partial x^{2}} + U^{''}\left(H_{0,1}(w_{0}(t,x))\right)\right]\left(e^{-\sqrt{2}d_{v}(t)}\mathcal{G}\left(w_{0}(t,x)\right)\right)\\ 
    \begin{aligned}
    =&8\sqrt{2} H^{'}_{0,1}(w_{0}(t,x))e^{-\sqrt{2}d_{v}(t)}\\ &{-}\left[24 H_{0,1}(w_{0}(t,x))^{2}-30H_{0,1}(w_{0}(t,x))^{4}\right]e^{-\sqrt{2}w_{0}(t,x)}e^{-\sqrt{2}d_{v}(t)} {+}R_{2,v}\left(t,w_{0}(t,x)\right),
    \end{aligned}
 \end{multline}
%\begin{remark}\label{R1R2}
 where $R_{2,v}(t,x)$ is a finite sum of functions
$h_{i}(x)p_{i,v}(t)$ with $h_{i}(x)\in S^{+}_{3}$ and $p_{i,v}(t)\in C^{\infty}(\mathbb{R})$ being an even function such that, for any $l\in\mathbb{N}\cup\{0\},$ $\left \vert\frac{d^{l}p_{i,v}(t)}{dt^{l}}\right \vert\lesssim_{l}v^{4+l}e^{-2\sqrt{2}v\lvert t \rvert}.$ Furthermore, if $0<v\ll1,$ then for any $s\geq 1,\,l\in\mathbb{N}\cup\{0\},$ we have
\begin{equation}\label{gw}
    \norm{\frac{\partial^{l}}{\partial t^{l}}R_{2,v}\left(t,w_{0}(t,x)\right)}_{H^{s}_{x}}\lesssim_{s,l} v^{4+l}e^{{-}2\sqrt{2}\vert t\vert v}.
\end{equation}
%\end{remark}
\end{lemma}
\begin{proof}
First, using equation \eqref{Gprop}, we deduce that
\begin{multline*}
    \frac{\dot d_{v}(t)^{2}}{4-\dot d_{v}(t)^{2}}\mathcal{G}^{''}(w_{0}(t,x))-\frac{\partial^{2}}{\partial x^{2}}\mathcal{G}(w_{0}(t,x))+U^{''}\left(H_{0,1}(w_{0}(t,x))\right)\mathcal{G}(w_{0}(t,x))\\= {-}\left[24 H_{0,1}(w_{0}(t,x))^{2}-30H_{0,1}(w_{0}(t,x))^{4}\right]e^{-\sqrt{2}w_{0}(t,x)}+8\sqrt{2} H^{'}_{0,1}(w_{0}(t,x)).
\end{multline*}
Consequently, we have 
\begin{align}\nonumber
    R_{2,v}\left(t,w_{0}(t,x)\right)=&\left[\frac{d^{2}}{d t^{2}}e^{{-}\sqrt{2}d_{v}(t)}\right]\mathcal{G}\left(w_{0}(t,x)\right)+2\left[\frac{d}{d t}e^{{-}\sqrt{2}d_{v}(t)}\right]\frac{\partial}{\partial t}\mathcal{G}\left(w_{0}(t,x)\right)\\ \label{r2ident}
    &{+}e^{{-}\sqrt{2}d_{v}(t)}\frac{\partial^{2}}{\partial t^{2}}\mathcal{G}\left(w_{0}(t,x)\right)-\frac{\dot d_{v}(t)^{2}}{4-\dot d_{v}(t)^{2}}e^{{-}\sqrt{2}d_{v}(t)}\mathcal{G}^{''}\left(w_{0}(t,x)\right).
\end{align}
\par Clearly identities \eqref{dfw0} and $\ddot d_{v}(t)=16\sqrt{2}e^{{-}\sqrt{2}d_{v}(t)}$ imply the following equations
\begin{align}\label{fdp7}
    \frac{\partial}{\partial t}\mathcal{G}(w_{0}(t,x))=&\left[{-}\frac{\dot d_{v}(t)}{\sqrt{4-\dot d_{v}(t)^{2}}} +\frac{16\sqrt{2}\dot d_{v}(t)}{4-\dot d_{v}(t)^{2}}e^{-\sqrt{2}d_{v}(t)}w_{0}(t,x)\right] \mathcal{G}^{'}(w_{0}(t,x)),\\
\label{fdp8}
    \frac{\partial}{\partial t}\mathcal{G}^{'}(w_{0}(t,x))=&\left[{-}\frac{\dot d_{v}(t)}{\sqrt{4-\dot d_{v}(t)^{2}}} +\frac{16\sqrt{2}\dot d_{v}(t)}{4-\dot d_{v}(t)^{2}}e^{-\sqrt{2}d_{v}(t)}w_{0}(t,x) \right]\mathcal{G}^{''}(w_{0}(t,x)),\\ \nonumber
   \frac{\partial}{\partial t}\left[w_{0}(t,x) \mathcal{G}^{'}(w_{0}(t,x))\right]=&{-}\frac{\dot d_{V}(t)}{\sqrt{4-\dot d(t)^{2}}} \left[w_{0}(t,x)\mathcal{G}^{''}(w_{0}(t,x))+\mathcal{G}^{'}(w_{0}(t,x))\right]\\ \nonumber
   &{+}\frac{16\sqrt{2}\dot d_{v}(t)}{4-\dot d_{v}(t)^{2}}e^{-\sqrt{2}d_{v}(t)}w_{0}(t,x)\left[w_{0}(t,x)\mathcal{G}^{''}(w_{0}(t,x))\right]
   \\ \label{fdp10}&{+}\frac{16\sqrt{2}\dot d_{v}(t)}{4-\dot d_{v}(t)^{2}}e^{-\sqrt{2}d_{v}(t)}w_{0}(t,x)\left[\mathcal{G}^{'}(w_{0}(t,x))\right].
\end{align}
Moreover, since $\mathcal{G}\in S^{+}_{1},$ then $\mathcal{G}^{''}(x),$ and $x^{2}\mathcal{G}^{''}(x)$ are in $S^{+}_{3}.$ 
Therefore, using estimates \eqref{mmma}, Lemma \ref{dlemma} and identities \eqref{r2ident}, \eqref{fdp7}, \eqref{fdp8}, \eqref{fdp10}, we deduce from the time derivative of \eqref{fdp7} and the product rule that $R_{2,v}(t,x)$ is a finite sum of functions $h_{i}(x)p_{i,v}(t)$ satisfying, for any index $i,$ the conditions $h_{i}\in S^{3}_{+}$ and
\begin{equation*}
    \left\vert \frac{d^{l}}{dt^{l}} p_{i,v}(t)\right\vert\lesssim_{l}v^{4+l}e^{{-}2\sqrt{2}\vert t\vert v} \text{ for all $l\in\mathbb{N}\cup\{0\}.$}
\end{equation*}
Therefore, estimate \eqref{gw} follows from Lemmas \ref{dlemma}, \ref{geraldt} and the product rule of derivative. Finally, Since $d_{v}(t)$ is an even function, we can deduce from Lemma \ref{geraldt} applied on $\mathcal{G}$ and identity \eqref{fdp7} that all the functions $p_{i,v}$ are even.
\end{proof}
\subsection{Applications of Proposition \ref{separation} .}
\par  This subsection contains lemmas that are consequences of Proposition \ref{separation} and Remarks \ref{sepc}, \ref{reflection}. These lemmas are going to be used later to estimate the remaining terms of $\Lambda(\varphi_{2,0})(t,x).$
From now on,
we denote
\begin{equation}\label{def1}
    M(x)=\frac{H_{0,1}(x)}{\sqrt{1+e^{2\sqrt{2}x}}},\, N(x)=\frac{H_{0,1}(x)^{3}}{\sqrt{1+e^{2\sqrt{2}x}}},\, V(x)=\frac{H_{0,1}(x)}{1+\sqrt{1+e^{2\sqrt{2}x}}}.
\end{equation}
\begin{remark}\label{Gequationn}
From \eqref{Gprop}, $-\frac{d^{2}}{dx^{2}} \mathcal{G}(x)+ U^{''}(H_{0,1})\mathcal{G}(x)=-24M(x)+30N(x)+8\sqrt{2}H^{'}_{0,1}(x).$
\end{remark}
\begin{lemma}\label{l1}
For any $\zeta>1,$ we have that
\begin{multline}\label{m2}
     U^{'}\left(H_{0,1}(x-\zeta)+H_{-1,0}(x)\right)- U^{'}\left(H_{0,1}(x-\zeta)\right)- U^{'}\left(H_{-1,0}(x)\right)\\
    \begin{aligned}
    =&24e^{-\sqrt{2}\zeta}\left[M(x-\zeta)-M(-x)\right]{-}30e^{-\sqrt{2}\zeta}\left[N(x-\zeta)-N(-x)\right]\\&{+}24e^{-2\sqrt{2}\zeta}\left[V(x-\zeta)-V(-x)\right]+\frac{60e^{-2\sqrt{2}\zeta}}{\sqrt{2}}\left[ H^{'}_{0,1}(x-\zeta)- H^{'}_{-1,0}(x)\right]\\&{+}R(x,\zeta),
    \end{aligned}
\end{multline}
where $R(x,\zeta)$ is a finite sum of terms $m_{i}(x-\zeta)n_{i}(x)e^{-(2+d_{i})\sqrt{2}\zeta},$ with $m_{i} \in S^{+},\,n_{i}\in S^{-}$ and $d_{i}\in\mathbb{N}\cup\{0\}.$ 
\end{lemma}
\begin{remark}\label{reint} In notation of Lemma \ref{l1}, if we replace $x,\,\zeta,$ respectively, with $-w_{0}(t,{-}x)$ and $\frac{d_{v}(t)}{\sqrt{1-\frac{\dot d_{v}(t)^{2}}{4}}},$ we obtain from Lemmas \ref{geraldt} and \ref{interactionsize} the following estimate
\begin{multline*}
U^{'}\left(H^{w_{0}}_{0,1}(t,x)\right)- U^{'}\left(H_{0,1}\left(w_{0}(t,x)\right)\right)- U^{'}\left({-}H_{0,1}\left(w_{0}(t,{-}x)\right)\right)\\
   \begin{aligned}
   \cong_{6} &
   24\exp\left(\frac{{-}\sqrt{2}d_{v}(t)}{\sqrt{1-\frac{\dot d_{v}(t)^{2}}{4}}}\right)M^{w_{0}}(t,x)-30\exp\left(\frac{{-}\sqrt{2}d_{v}(t)}{\sqrt{1-\frac{\dot d_{v}(t)^{2}}{4}}}\right)N^{w_{0}}(t,x)\\&{+}24\exp\left(\frac{{-}2\sqrt{2}d_{v}(t)}{\sqrt{1-\frac{\dot d(t)^{2}}{4}}}\right)V^{w_{0}}(t,x) +\frac{60}{\sqrt{2}}\exp\left(\frac{{-}2\sqrt{2}d_{v}(t)}{\sqrt{1-\frac{\dot d_{v}(t)^{2}}{4}}}\right)\left(H^{'}_{0,1}\right)^{w_{0}}(t,x). 
\end{aligned}
\end{multline*}
\end{remark}
\begin{proof}[Proof of Lemma \ref{l1}.] 
Using Remarks \ref{alg}, \ref{basisS} and Lemmas \ref{multiplicative}, we can demonstrate Lemma \ref{l1} from Lemma \ref{prelema} and from the identity $U(\phi)=\phi^{2}(1-\phi^{2})^{2}.$ 
\end{proof}
\begin{lemma}\label{interg}
There exist $A,\, B,\, C,\,D\in S^{+}\cap \mathscr{S}(\mathbb{R})$ and there exists a finite set of quadruples $\left (h_{i,+},\,h_{i,-},\,d_{i},\,l_{i}\right) \in S^{+}\times S^{-}\times \mathbb{N}^{2},$ with $h_{i,+}$ or $h_{i,-}$ in $\mathscr{S}(\mathbb{R}),\,l_{i}\in\{0,1\}$ and $d_{i}\geq 0,$ satisfying the following identity
\begin{multline}\label{int23}
   \left[U^{''}\left(H_{0,1}(x-\zeta )+H_{-1,0}(x)\right)-U^{''}\left(H_{0,1}(x-\zeta)\right)\right]e^{-\sqrt{2}\zeta}\mathcal{G}(x-\zeta)\\=
    (x-\zeta)A(x-\zeta)e^{-2\sqrt{2}\zeta}+(x-\zeta) B(-x)e^{-2\sqrt{2}\zeta}+C(x-\zeta)e^{-2\sqrt{2}\zeta}+D(-x)e^{-2\sqrt{2}\zeta}\\+\sum_{i}(x-\zeta)^{l_{i}}h_{i,+}(x-\zeta)h_{i,-}(x)e^{-(2+d_{i})\sqrt{2}\zeta},
\end{multline}
for all $x\in\mathbb{R}$ and any $\zeta>1.$
\end{lemma}
\begin{remark}\label{remarkinterg}
In notation of Lemma \ref{interg}, for all $(x,\zeta)\in\mathbb{R}^{2},$ we denote the real function $\mathcal{Q}:\mathbb{R}^{2}\to\mathbb{R}$ by 
\begin{equation}\label{QQ}
    \mathcal{Q}(x,\zeta)=(x-\zeta)A(x-\zeta)e^{-\sqrt{2}\zeta}\\+(x-\zeta) B(-x)e^{-\sqrt{2}\zeta}+C(x-\zeta)e^{-\sqrt{2}\zeta}+D(-x)e^{-\sqrt{2}\zeta},
\end{equation}
and the function $R_{q}:\mathbb{R}^{2}\to\mathbb{R}$ by 
\begin{equation}\label{Rq}
    R_{q}(x,\zeta)=\sum_{i}(x-\zeta)^{l_{i}}h_{i,+}(x-\zeta)h_{i,-}(x)e^{-(1+d_{i})\sqrt{2}\zeta},
\end{equation}
 for any $(x,\zeta)\in\mathbb{R}^{2}.$
 If we change the variables $x,\,\zeta,$ respectively, with ${-}w_{0}(t,{-}x)$ and
$
    \frac{d_{v}(t)}{\sqrt{1-\frac{\dot d_{v}(t)^{2}}{4}}},
$
we obtain using Lemmas \ref{explemma} and \ref{interg} that
\begin{equation*}
    \left[ U^{''}\left(H^{w_{0}}_{0,1}(t,x)\right)- U^{''}\left(H_{0,1}\left(w_{0}(t,x)\right)\right)\right]e^{{-}\sqrt{2}d_{v}(t)}\mathcal{G}\left(w_{0}(t,x)\right)\cong_{6}
    \mathcal{Q}\left({-}w_{0}(t,{-}x),d_{v}(t)\right)e^{{-}\sqrt{2}d_{v}(t)}.
\end{equation*}
Indeed, using Lemma \ref{interactionsize} and Lemma \ref{dlemma}, we deduce that if $0<v\ll 1,$ then 
\begin{equation*}
R_{q}\left({-}w_{0}(t,{-}x),\frac{d_{v}(t)}{\sqrt{1-\frac{\dot d_{v}(t)^{2}}{4}}}\right)e^{{-}\sqrt{2}d_{v}(t)}\cong_{6} 0.
\end{equation*}
\end{remark}
\begin{proof}[Proof of Lemma \ref{interg}.]
The identity \eqref{G(x)} implies that $\mathcal{G}_{1}(x)=\mathcal{G}(x)-2x H^{'}_{0,1}(x)\in S^{+}\cap \mathscr{S}(\mathbb{R}).$ So,
the proof follows from Remark \ref{basisS}, applications of Proposition \ref{separation}, and Remark \ref{reflection}
in the following expressions
\begin{align*}
     \left(U^{''}\left(H_{0,1}(x-\zeta)+H_{-1,0}(x)\right)- U^{''}\left(H_{0,1}(x-\zeta)\right)\right)\mathcal{G}_{1}(x-\zeta)e^{-\sqrt{2}\zeta},\\
2\left( U^{''}\left(H_{0,1}(x-\zeta)+H_{-1,0}(x)\right)- U^{''}\left(H_{0,1}(x-\zeta)\right)\right)(x-\zeta)H^{'}_{0,1}(x-\zeta)e^{-\sqrt{2}\zeta}.
\end{align*}
\par More precisely, since
\begin{multline*}
    U^{''}\left(H_{0,1}(x-\zeta)+H_{{-}1,0}(x)\right)- U^{''}\left(H_{0,1}(x-\zeta)\right)\\
    \begin{aligned}
    =&{-}24H_{{-}1,0}(x)^{2}+30H_{{-}1,0}(x)^{4}\\ &{-}48H_{0,1}(x-\zeta)H_{-1,0}(x)
    +30\sum_{i=1}^{3}
    \begin{pmatrix}
    4\\
    i
    \end{pmatrix}
    H_{{-}1,0}(x)^{i}H_{0,1}(x-\zeta)^{4-i},
    \end{aligned}
\end{multline*}
we obtain that $ \left( U^{''}\left(H_{0,1}(x-\zeta)+H_{{-}1,0}(x)\right)-U^{''}\left(H_{0,1}(x-\zeta)\right)\right)\mathcal{G}_{1}(x-\zeta)$ is a linear combination of functions
$
    H_{0,1}(x-\zeta)^{m_{i}}H_{-1,0}(x)^{l_{i}}h_{i}(x-\zeta),
$
such that $h_{i}\in S^{+}\cap\mathscr{S}(\mathbb{R}),\,m_{i}\in\mathbb{N}\cup\{0\},\,l_{i}\in\mathbb{N}$ and $0<m_{i}+n_{i}$ is an even number. By similar reasoning, we can verify that
\begin{equation*}
     2\left[ U^{''}\left(H_{0,1}(x-\zeta)+H_{{-}1,0}(x)\right)- U^{''}\left(H_{0,1}(x-\zeta)\right)\right](x-\zeta) H^{'}_{0,1}(x-\zeta)
\end{equation*}
is also a linear combination of functions
$
    (x-\zeta)H_{0,1}(x-\zeta)^{m_{i}}H_{-1,0}(x)^{l_{i}} H^{'}_{0,1}(x-\zeta),
$
such that $m_{i}\in\mathbb{N}\cup\{0\},\,l_{i}\in\mathbb{N}$ and $0<m_{i}+l_{i}$ is an even number. Therefore, using Lemma \ref{multiplicative}, we can verify that 
\begin{equation*}
     \left[ U^{''}\left(H_{0,1}(x-\zeta)+H_{-1,0}(x)\right)-U^{''}\left(H_{0,1}(x-\zeta)\right)\right]\mathcal{G}(x-\zeta)
\end{equation*}
is a linear combination of functions
$
    (x-\zeta)^{\alpha_{i}}h_{i,1}\left(x-\zeta\right)h_{i,2}\left(x\right)
$
such that $\alpha_{i}\in\{0,1\},$ $h_{i,1}$ or $h_{i,2} \in \mathscr{S}(\mathbb{R})$ and either $h_{i,1}(x)\in S^{+}$ and $h_{i,2}(x)\in S^{-}$ or $h_{i,1}({-}x)\in S^{-}$ and $h_{i,2}({-}x)\in S^{+}.$ In conclusion, the statement of Lemma \ref{gelint} is a consequence of Proposition \ref{separation} and Remarks \ref{sepc}, \ref{reflection}.
\end{proof}
\begin{lemma}\label{Remaindertaylor}
For all $\zeta\geq 1,$ 
\begin{equation*}
    \mathcal{D}_{1}(x,\zeta)=\sum_{j=4}^{6}\frac{1}{(j-1)!} U^{(j)}\left(H_{0,1}(x-\zeta)+H_{-1,0}(x)\right)\left(\mathcal{G}(x-\zeta)-\mathcal{G}(-x)\right)^{j-1}e^{-(j-1)\sqrt{2}\zeta}
\end{equation*}
satisfies for any $l_{1},\,l_{2}\in\mathbb{N}\cup\{0\}$ the following estimate
\begin{equation*}
    \norm{\frac{\partial^{l_{1}+l_{2}}}{\partial x^{l_{1}}\partial\zeta^{l_{2}}}\mathcal{D}_{1}(x,\zeta)}_{L^{2}_{x}}\lesssim_{l_{1}+l_{2}}e^{-3\sqrt{2}\zeta}.
\end{equation*}
\end{lemma}
\begin{remark}\label{remarktaylorr}
Indeed, using Lemmas \ref{dlemma}, \ref{geraldt} and the product rule of derivative, we have that 
\begin{equation*}
    \norm{\frac{\partial^{l_{1}+l_{2}}}{\partial x^{l_{1}}\partial t^{l^{2}}}\left[\mathcal{D}_{1}\left({-}w_{0}(t,{-}x),\frac{d_{v}(t)}{\sqrt{1-\frac{\dot d_{v}(t)^{2}}{4}}}\right)\right]}_{L^{2}_{x}}\lesssim_{l_{1},l_{2}} v^{l_{2}+6}e^{{-}2\sqrt{2}\vert t\vert v}.
\end{equation*}
In conclusion, the following function
\begin{equation*}
    \mathcal{D}_{1,1}(t,x)=\sum_{j=4}^{6}\frac{1}{(j-1)!}U^{(j)}\left(H^{w_{0}}_{0,1}(t,x)\right)\mathcal{G}^{w_{0}}(t,x)^{j-1}e^{{-}(j-1)\sqrt{2}d_{v}(t)}
\end{equation*}
satisfies $D_{1,1}\cong_{6} 0.$
\end{remark}
\begin{proof}[Proof of Lemma \ref{Remaindertaylor}.]
First, since $U\in C^{\infty}(\mathbb{R}),\,0< H_{0,1}<1$ and $H^{'}_{0,1} \in \mathscr{S}(\mathbb{R}),$ we obtain for all $\zeta\in\mathbb{R}$ and any $l_{1},\,l_{2},\,l_{3}\in\mathbb{N}\cup\{0\}$ that
\begin{equation*}
    \norm{\frac{\partial^{l_{1}+l_{2}}}{\partial x^{l_{1}}\partial \zeta^{l_{2}}}U^{(l_{3})}\left(H_{0,1}(x-\zeta)+H_{{-}1,0}(x)\right)}_{L^{\infty}_{x}(\mathbb{R})}\lesssim_{l_{1},l_{2},l_{3}} 1.
\end{equation*} 
In conclusion, since $\mathcal{G}\in\mathscr{S}(\mathbb{R})$ and 
\begin{equation*}
\norm{f g}_{H^{s}_{x}}\lesssim_{s}\norm{f}_{H^{s}_{x}}\norm{g}_{L^{\infty}_{x}(\mathbb{R})}+\norm{f}_{H^{s}_{x}}\norm{g}_{L^{\infty}_{x}(\mathbb{R})},
\end{equation*}
$\norm{f g}_{H^{s}_{x}}\lesssim_{s} \norm{f }_{H^{s}_{x}}\norm{g}_{H^{s}_{x}}$ for all $f,\,g\in H^{s}_{x}$ when $s\geq 1$, we deduce for any $l_{1},\,l_{2}\in\mathbb{N}\cup\{0\}$ and all $\zeta\geq 1$
that
\begin{equation*}
   \norm{ \frac{\partial^{l_{1}+l_{2}}}{\partial x^{l_{1}}\partial\zeta^{l_{2}}}\mathcal{D}_{1}(x,\zeta)}_{L^{2}_{x}}\lesssim_{l_{1}+l_{2}}\left[\norm{\mathcal{G}}_{H^{l_{1}+l_{2}+1}_{x}}^{3}+\norm{\mathcal{G}}_{H^{l_{1}+l_{2}+1}_{x}}^{5}\right]e^{-3\sqrt{2}\zeta}\lesssim e^{-3\sqrt{2}\zeta}.
\end{equation*}
\end{proof}
Next, we  consider the following lemma.
\begin{lemma}\label{g3}
There exists a finite set of elements $\left(W_{i},\mathcal{W}_{i},d_{i},j_{i},l_{i}\right)\in S^{+}\times S^{-}\times \left(\mathbb{N}\cup\{0\}\right)^{3}$ such that $W_{i}$ or $\mathcal{W}_{i}$ is in $\mathscr{S}(\mathbb{R}),\, j_{i},\,l_{i}$
satisfy $0\leq j_{i}+l_{i}\leq 2$ for all $i$ and we have the following identity
\begin{align*}
  \mathcal{D}_{2}(x,\zeta)=&\frac{1}{2}  U^{(3)}\left(H_{0,1}(x-\zeta)+H_{-1,0}(x)\right)\left(\mathcal{G}(x-\zeta)-\mathcal{G}(-x)\right)^{2}e^{-2\sqrt{2}\zeta}\\=&
  \frac{1}{2}\left[U^{(3)}\left(H_{0,1}(x-\zeta)\right)\mathcal{G}(x-\zeta)^{2}e^{-2\sqrt{2}\zeta}+U^{(3)}\left(H_{-1,0}(x)\right)\mathcal{G}(-x)^{2}e^{-2\sqrt{2}\zeta}\right]\\&{+}\sum_{i}(x-\zeta)^{j_{i}}({-}x)^{l_{i}}W_{i}(x-\zeta)\mathcal{W}_{i}(x)e^{-(2+d_{i})\sqrt{2}\zeta}
  \\&{-}\sum_{i}({-}x)^{j_{i}}(x-\zeta)^{l_{i}}W_{i}({-}x)\mathcal{W}_{i}({-}x+\zeta)e^{-(2+d_{i})\sqrt{2}\zeta},
\end{align*}
for all $\zeta\geq 1.$
\end{lemma}
\begin{remark}\label{remarkg3}
Using Lemmas \ref{dlemma}, \ref{interactionsize} and \ref{g3}, we can verify that
\begin{equation*}
     \frac{1}{2}U^{(3)}\left(H^{w_{0}}_{0,1}(t,x)\right)\mathcal{G}^{w_{0}}(t,x)^{2}e^{{-}2\sqrt{2}d_{v}(t)}\cong_{6} \frac{1}{2}\left[U^{(3)}\left(H_{0,1}\right)\mathcal{G}^{2}\right]^{w_{0}}(t,x)e^{{-}2\sqrt{2}d_{v}(t)}.
\end{equation*}
 \end{remark}
\begin{proof}[Proof of Lemma \ref{g3}.]
The proof follows from Proposition \ref{separation} and Remarks \ref{sepc}, \ref{reflection}. More precisely, since
\begin{align*} 
    U^{(3)}\left(H_{0,1}(x-\zeta)+H_{{-}1,0}(x)\right)= U^{(3)}\left(H_{0,1}(x-\zeta)\right)+U^{(3)}\left(H_{{-}1,0}(x)\right)\\+360\left[H_{0,1}(x-\zeta)^{2}H_{{-}1,0}(x)+H_{0,1}(x-\zeta)H_{{-}1,0}(x)^{2}\right],
\end{align*}
we deduce that
\begin{multline}\label{mmm}
        \frac{1}{2}  U^{(3)}\left(H_{0,1}(x-\zeta)+H_{{-}1,0}(x)\right)\left(\mathcal{G}(x-\zeta)-\mathcal{G}({-}x)\right)^{2}e^{-2\sqrt{2}\zeta}
  \\{-}\frac{1}{2}U^{(3)}\left(H_{0,1}(x-\zeta)\right)\mathcal{G}(x-\zeta)^{2}e^{-2\sqrt{2}\zeta}-\frac{1}{2}U^{(3)}\left(H_{{-}1,0}(x)\right)\mathcal{G}({-}x)^{2}e^{-2\sqrt{2}\zeta}\\
  \begin{aligned}
  =&\frac{1}{2}U^{(3)}\left(H_{0,1}(x-\zeta)\right)\left(\mathcal{G}({-}x)^{2}-2\mathcal{G}(x-\zeta)\mathcal{G}({-}x)\right)e^{{-}2\sqrt{2}\zeta}
  \\&{+}\frac{1}{2}U^{(3)}\left(H_{{-}1,0}(x)\right)\left(\mathcal{G}(x-\zeta)^{2}-2\mathcal{G}(x-\zeta)\mathcal{G}({-}x)\right)e^{{-}2\sqrt{2}\zeta}\\&{+}\left[360H_{0,1}(x-\zeta)^{2}H_{{-}1,0}(x)+360H_{0,1}(x-\zeta)H_{{-}1,0}(x)^{2}\right]\left(\mathcal{G}(x-\zeta)^{2}+\mathcal{G}({-}x)^{2}\right)e^{{-}2\sqrt{2}\zeta}\\
  &{-}2\left[360H_{0,1}(x-\zeta)^{2}H_{{-}1,0}(x)+360H_{0,1}(x-\zeta)H_{{-}1,0}(x)^{2}\right]\mathcal{G}(x-\zeta)\mathcal{G}({-}x)e^{{-}2\sqrt{2}\zeta}.
\end{aligned}
\end{multline}
Moreover, since $U^{(3)}(\phi)=-48\phi+120\phi^{3}$ is an odd polynomial and $H_{{-}1,0}(x)=-H_{0,1}({-}x),$ the right-hand side of  \eqref{mmm} is a finite sum of functions 
\begin{equation*}
    \mathcal{G}(x-\zeta)^{l_{1}}\mathcal{G}({-}x)^{l_{2}}H^{\zeta}_{0,1}\left(x\right)^{l_{3}}H_{0,1}({-}x)^{l_{4}}-\mathcal{G}(x-\zeta)^{l_{2}}\mathcal{G}({-}x)^{l_{1}}H^{\zeta}_{0,1}\left(x\right)^{l_{4}}H_{0,1}({-}x)^{l_{3}},
\end{equation*}
such that $l_{1},\,l_{2},\,l_{3},\,l_{4}\in\mathbb{N}\cup\{0\},\, l_{1}+l_{2}= 2,\,\sum_{i=1}^{4}l_{i}$ is odd and   $\min\left(l_{1}+l_{3},l_{2}+l_{4}\right)>0.$ Therefore, using Lemma \ref{multiplicative} and Remark \ref{genemult}, we deduce that \eqref{mmm} is a finite sum of functions
\begin{equation*}
    \mathcal{J}_{i}\left(x-\zeta\right)\mathcal{N}_{i}\left(x\right)-\mathcal{J}_{i}({-}x)\mathcal{N}_{i}\left({-}x+\zeta\right),
\end{equation*}
where $\mathcal{J}_{i}\in S^{+}\cup S^{+}_{\infty}$ and $\mathcal{N}_{i}\in S^{-}\cup S^{-}_{\infty}.$ 
In conclusion, we obtain the statement of Lemma \ref{g3} from the Proposition \ref{separation} and Remarks \ref{sepc}, \ref{reflection} applied in the right-hand side of \eqref{mmm}.  
\end{proof}
\par Now, we can start the estimate of $\Lambda(\varphi_{2,0})(t,x).$
First, from the definition of $\varphi_{2,0}(t,x)$ in \eqref{phi2(0)}, we have for any $(t,x)\in\mathbb{R}^{2}$ that $\Lambda\left(\varphi_{2,0}\right)(t,x)$ is equal to
\begin{multline*}
\left[\frac{\partial^{2}}{\partial t^{2}}-\frac{\partial^{2}}{\partial x^{2}}\right]\left(H^{w_{0}}_{0,1}(t,x)+e^{{-}\sqrt{2}d_{v}(t)}\mathcal{G}^{w_{0}}(t,x)\right)
    {+} U^{'}\left(H^{w_{0}}_{0,1}(t,x)+e^{{-}\sqrt{2}d_{v}(t)}\mathcal{G}^{w_{0}}(t,x)\right)\\
    =\left[\frac{\partial^{2}}{\partial t^{2}}-\frac{\partial^{2}}{\partial x^{2}}\right]\left[e^{{-}\sqrt{2}d_{v}(t)}\mathcal{G}^{w_{0}}(t,x)\right]+\Lambda\left(H_{0,1}\left(w_{0}(t,x)\right)\right)-\Lambda\left(H_{0,1}\left(w_{0}(t,{-}x)\right)\right)\\
   {+} U^{'}\left(H^{w_{0}}_{0,1}(t,x)+e^{{-}\sqrt{2}d_{v}(t)}\mathcal{G}^{w_{0}}(t,x)\right)-U^{'}\left(H_{0,1}\left(w_{0}(t,x)\right)\right)\\{-}U^{'}\left({-}H_{0,1}\left(w_{0}(t,{-}x)\right)\right).
\end{multline*}
Therefore, using Taylor's Expansion Theorem, we deduce that
\begin{multline}\label{sumbig0}
   \Lambda(\varphi_{2,0})(t,x)-\Lambda( H_{0,1}\left(w_{0}(t,x)\right))+\Lambda\left( H_{0,1}\left(w_{0}(t,{-}x)\right)\right)\\
   \begin{aligned}
   =& \left[\frac{\partial^{2}}{\partial t^{2}}-\frac{\partial^{2}}{\partial x^{2}}\right]\left[e^{{-}\sqrt{2}d_{v}(t)}\mathcal{G}^{w_{0}}(t,x)\right]
   \\&{+} U^{'}\left(H^{w_{0}}_{0,1}(t,x)\right)
   - U^{'}\left(H_{0,1}\left(w_{0}(t,x)\right)\right)-U^{'}\left({-}H_{0,1}\left(w_{0}(t,{-}x)\right)\right)
  \\ & {+}\sum_{j=2}^{6}\frac{U^{(j)}\left(H^{w_{0}}_{0,1}(t,x)\right)}{(j-1)!}\left[e^{{-}\sqrt{2}d_{v}(t)}\mathcal{G}^{w_{0}}\left(t,x\right)\right]^{j-1}.
\end{aligned}
\end{multline}
Consequently, we deduce using Lemma \ref{dt2kink} that
\begin{align}\nonumber
   \Lambda(\varphi_{2,0})(t,x)= &
   {-}\frac{8\sqrt{2}e^{{-}\sqrt{2}d_{v}(t)}}{\sqrt{1-\frac{\dot d_{v}(t)^{2}}{4}}} \left(H^{'}_{0,1}\right)^{w_{0}}(t,x)+R_{1}\left(t,w_{0}(t,x)\right)-R_{1}\left(t,w_{0}(t,{-}x)\right)
   \\ \label{intelgg1}
&{+}e^{{-\sqrt{2}}d_{v}(t)}\left[ U^{''}\left(H^{w_{0}}_{0,1}(t,x)\right)\mathcal{G}^{w_{0}}(t,x)-\left( U^{''}\left(H_{0,1}\right)\mathcal{G}\right)^{w_{0}}(t,x)\right]
  \\  \label{sumbig}
   &{+}\sum_{j=4}^{6}\frac{U^{(j)}\left(H^{w_{0}}_{0,1}\left(t,x\right)\right)}{(j-1)!}\left[e^{{-}\sqrt{2}d_{v}(t)}\mathcal{G}^{w_{0}}(t,x)\right]^{j-1}
   \\  \label{nolinteract}
   &{+} U^{'}\left(H^{w_{0}}_{0,1}(t,x)\right)
   -U^{'}\left(H_{0,1}\left(w_{0}(t,x)\right)\right)- U^{'}\left({-}H_{0,1}\left(w_{0}(t,{-}x)\right)\right)
  \\ \label{quadraticterm} &{+}\frac{U^{(3)}\left(H^{w_{0}}_{0,1}\left(t,x\right)\right)}{2}\left[e^{{-}\sqrt{2}d_{v}(t)}\mathcal{G}^{w_{0}}(t,x)\right]^{2}
  \\ \label{lineargterm} &{+}\left[\frac{\partial^{2}}{\partial t^{2}}-\frac{\partial^{2}}{\partial x^{2}}\right]\left[e^{{-}\sqrt{2}d_{v}(t)}\mathcal{G}^{w_{0}}(t,x)\right]+e^{{-}\sqrt{2}d_{v}(t)}\left[U^{''}\left(H_{0,1}\right)\mathcal{G}\right]^{w_{0}}(t,x).
\end{align}
\par Next, from Remark \ref{remarkinterg}, we have that the expression \eqref{intelgg1} is equivalent to
\begin{equation*}
    e^{{-}\sqrt{2}d_{v}(t)}\left[\mathcal{Q}\left({-}w_{0}(t,{-}x),d_{v}(t)\right)-\mathcal{Q}\left({-}w_{0}(t,x),d_{v}(t)\right)\right].
\end{equation*}
\par Moreover, Remark \ref{remarktaylorr} implies that the term \eqref{sumbig} is negligible.
 \par Additionally, using Remark \ref{reint}, we obtain that the expression \eqref{nolinteract} is equivalent to
\begin{align*}
    24\exp\left(\frac{{-}\sqrt{2}d_{v}(t)}{\sqrt{1-\frac{\dot d_{v}(t)^{2}}{4}}}\right)M^{w_{0}}(t,x)-30\exp\left(\frac{{-}\sqrt{2}d_{v}(t)}{\sqrt{1-\frac{\dot d_{v}(t)^{2}}{4}}}\right)N^{w_{0}}(t,x)\\{+}24\exp\left(\frac{{-}2\sqrt{2}d_{v}(t)}{\sqrt{1-\frac{\dot d(t)^{2}}{4}}}\right)V^{w_{0}}(t,x)+\frac{60}{\sqrt{2}}\exp\left(\frac{{-}2\sqrt{2}d_{v}(t)}{\sqrt{1-\frac{\dot d_{v}(t)^{2}}{4}}}\right)\left(H^{'}_{0,1}\right)^{w_{0}}(t,x).
\end{align*}
\par Finally, Remark \ref{remarkg3} implies that the term \eqref{quadraticterm} is equivalent to
\begin{equation*}
    \frac{e^{{-}2\sqrt{2}d_{v}(t)}}{2}\left[U^{(3)}\left(H_{0,1}\right)\mathcal{G}^{2}\right]^{w_{0}}(t,x),
\end{equation*}
and Lemma \ref{LambdaG} implies the equivalence between the expression \eqref{lineargterm} with
\begin{multline*}
{-}e^{{-}\sqrt{2}d_{v}(t)}\left[24M^{w_{0}}(t,x)-30N^{w_{0}}(t,x)-8\sqrt{2} \left(H^{'}_{0,1}\right)^{w_{0}}(t,x) \right]+R_{2,v}(t,w_{0}(t,x))\\{-}R_{2,v}(t,w_{0}(t,{-}x)).
\end{multline*}
\par Consequently, we have the following estimate
\begin{align*}
    \Lambda\left(\varphi_{2,0}\right)(t,x)\cong_{6}& {-}\frac{8\sqrt{2}e^{{-}\sqrt{2}d_{v}(t)}}{\sqrt{1-\frac{\dot d_{v}(t)^{2}}{4}}}\left(H^{'}_{0,1}\right)^{w_{0}}(t,x)+R_{1,v}\left(t,w_{0}(t,x)\right)-R_{1}\left(t,w_{0}(t,{-}x)\right)\\ 
&{+}e^{{-}\sqrt{2}d_{v}(t)}\left[\mathcal{Q}\left({-}w_{0}(t,{-}x),d_{v}(t)\right)-\mathcal{Q}\left({-}w_{0}(t,x),d_{v}(t)\right)\right]\\
&{+}24\exp\left(\frac{{-}\sqrt{2}d_{v}(t)}{\sqrt{1-\frac{\dot d_{v}(t)^{2}}{4}}}\right)M^{w_{0}}(t,x)-30\exp\left(\frac{{-}\sqrt{2}d_{v}(t)}{\sqrt{1-\frac{\dot d_{v}(t)^{2}}{4}}}\right)N^{w_{0}}(t,x)\\&{+}24\exp\left(\frac{{-}2\sqrt{2}d_{v}(t)}{\sqrt{1-\frac{\dot d(t)^{2}}{4}}}\right)V^{w_{0}}(t,x) +\frac{60}{\sqrt{2}}\exp\left(\frac{{-}2\sqrt{2}d_{v}(t)}{\sqrt{1-\frac{\dot d_{v}(t)^{2}}{4}}}\right)\left(H^{'}_{0,1}\right)^{w_{0}}(t,x)\\
&{-}e^{{-}\sqrt{2}d_{v}(t)}\left[24M^{w_{0}}(t,x)-30N^{w_{0}}(t,x)-8\sqrt{2}\left( H^{'}_{0,1}\right)^{w_{0}}(t,x) \right]\\&{+}R_{2,v}(t,w_{0}(t,x))-R_{2,v}(t,w_{0}(t,{-}x))
\\&{+}\frac{e^{{-}2\sqrt{2}d_{v}(t)}}{2}\left[U^{(3)}\left(H_{0,1}\right)\mathcal{G}^{2}\right]^{w_{0}}(t,x).
\end{align*}
Furthermore, using Lemma \ref{explemma} the following result, we deduce the following estimate
\begin{equation*}
   \left\vert \frac{d^{l}}{dt^{l}}\left[ \exp\left(\frac{{-}2\sqrt{2}d_{v}(t)}{\sqrt{1-\frac{\dot d_{v}(t)^{2}}{4}}}\right)-e^{{-}2\sqrt{2}d_{v}(t)}\right]\right\vert\lesssim_{l} v^{6+l}\left[\vert t\vert v+\ln{\left(\frac{1}{v}\right)}\right]e^{{-}2\sqrt{2}\vert t\vert v},
\end{equation*}
for any $l\in\mathbb{N}\cup\{0\}$ and $t\in\mathbb{R},$ if $0<v\ll 1.$ In conclusion, from Lemma \ref{geraldt}, Remark \ref{remarkinterg} and the estimate above of $\Lambda(\varphi_{2,0}),$ we deduce the following result:
\begin{lemma}\label{phi222}
The function $\varphi_{2,0}(t,x)$ satisfies if $1<v\ll 1,$ for all $l\in\mathbb{N}\cup\{0\}$ and $s\geq 0,$    \begin{equation*}
 \norm{\frac{\partial^{l}}{\partial t^{l}}\Lambda(\varphi_{2,0})(t,x)}_{H^{s}_{x}}\lesssim_{l,s}v^{4+l}\left[\vert t \vert v+\ln{\left(\frac{1}{v^{2}}\right)}\right]e^{-2\sqrt{2}\vert t\vert v}. 
\end{equation*}
Furthermore, we have that

\begin{equation*}
    \Lambda\left(\varphi_{2,0}\right)(t,x)\cong_{6} Sym\left(t,w_{0}(t,x)\right)-
    Sym\left(t,w_{0}(t,{-}x)\right),
\end{equation*}
where, for $0<v\ll 1,$ the function $Sym:\mathbb{R}^{2}\to\mathbb{R}$ satisfies, for all $(t,x)\in\mathbb{R},$ the following identity
\begin{align*}
 Sym(t,x)
    =& 8\sqrt{2} H^{'}_{0,1}(x)\left[e^{-\sqrt{2}d_{v}(t)}-\frac{e^{-\sqrt{2}d_{v}(t)}}{\sqrt{1-\frac{\dot d_{v}(t)^{2}}{4}}}\right]+\frac{1}{2}U^{(3)}\left(H_{0,1}(x)\right)\mathcal{G}(x)^{2}e^{-2\sqrt{2}d_{v}(t)}\\&{+}\left[-24 H_{0,1}(x)^{2}+30H_{0,1}(x)^{4}\right]e^{-\sqrt{2}x}\left[e^{-\sqrt{2}d_{v}(t)}-\exp\left(\frac{-\sqrt{2}d_{v}(t)}{\sqrt{1-\frac{\dot d(t)^{2}}{4}}}\right)\right]
    \\&{+}R_{1,v}(t,x)+R_{2,v}(t,x)
    \\&{+}e^{-2\sqrt{2}d_{v}(t)}\left[xA(x)+xB(x)-d(t)B(x)+C(x)-D(x)+24 V(x)+\frac{60}{\sqrt{2}} H^{'}_{0,1}(x)\right].
\end{align*}
\end{lemma}
Now, we can start the demonstration of Theorem \ref{k=2}. 
\subsection{Proof of Theorem \ref{k=2}.}
\begin{proof}[Proof of Theorem \ref{k=2}.]
\textbf{Step 1.}(Construction of $r_{v}(t)$ for $k=2.$)
\par  First, we recall $R_{1,v}(t,x),\,R_{2,v}(t,x)$ defined, respectively, in equation \eqref{lambdah} of Lemma \ref{dt2kink} and in equation \eqref{lambdag} of Lemma \ref{LambdaG}. To lighten more our notation, we denote $R_{1,v},\, R_{2,v},\,d_{v}(t)$ by $R_{1},\,R_{2},\,d(t)$ from now on. Also, we recall the functions $M(x),\,N(x)$ and $V(x)$ from \eqref{def1} and the functions $A,\,B,\,C,\,D$ from Lemma \ref{interg}.
Next, based on Lemma \ref{phi222}, we consider the following ordinary differential equation
\begin{multline}\label{ode1}
    \begin{cases}
    \norm{ H^{'}_{0,1}}_{L^{2}_{x}}^{2}\ddot r(t)={-}32e^{-\sqrt{2}d(t)}\norm{ H^{'}_{0,1}}_{L^{2}_{x}}^{2}r(t)-\left\langle H^{'}_{0,1}(x),\,Sym(t,x)\right\rangle, \\
    r(t)=r(-t).
    \end{cases}
\end{multline}
Indeed, from the definition of $Sym$ in the statement of Lemma \ref{phi222}, identities \eqref{oioi} and $\norm{ H^{'}_{0,1}}_{L^{2}_{x}}^{2}=\frac{1}{2\sqrt{2}},$ the ordinary differential equation \eqref{ode1} can be rewritten for fixed constants $c_{1},\,c_{2} \in \mathbb{R}$ as
\begin{multline}\label{odeasy}
    \begin{cases}
\begin{aligned}
    \norm{ H^{'}_{0,1}}_{L^{2}_{x}}^{2}\ddot r(t)=&{-}32e^{-\sqrt{2}d(t)}\norm{H^{'}_{0,1}}_{L^{2}_{x}}^{2}r(t)-\left\langle H^{'}_{0,1}(x),\,R_{1}(t,x)+R_{2}(t,x)\right\rangle\\&{+}c_{1}d(t)e^{-2\sqrt{2}d(t)}{+}c_{2}e^{-2\sqrt{2}d(t)}
     +4\left[\frac{e^{{-}\sqrt{2}d(t)}}{\sqrt{1-\frac{\dot d(t)^{2}}{4}}}-\exp\left(\frac{{-}\sqrt{2}d(t)}{\sqrt{1-\frac{\dot d(t)^{2}}{4}}}\right)\right],
    \end{aligned}
     \\
     r(t)= r(-t).
    \end{cases}
\end{multline}
\par Since $d(t)=\frac{1}{\sqrt{2}}\ln{\left(\frac{8}{v^{2}}\cosh{(\sqrt{2}vt)}^{2}\right)},$ we have that all the solutions of the linear ordinary differential equation
$\ddot r_{0}(t)=-32e^{-\sqrt{2}d(t)}r_{0}(t)
$
are a linear combination of
\begin{equation*}
sol_{1}(t)=\tanh{(\sqrt{2}vt)} \text{ and } sol_{2}(t)=\sqrt{2}v t\tanh{(\sqrt{2}vt)}-1.
\end{equation*}
From Lemma \ref{dlemma}, we obtain if $0<v\ll 1,$ and $l\in\mathbb{N}\cup\{0\},$ 
\begin{equation}\label{dd(t)}
   \left\vert \frac{d^{l}}{dt^{l}}d(t)e^{-2\sqrt{2}d(t)}\right\vert\lesssim_{l} v^{4+l}\left(v\vert t\vert+\ln{\left(\frac{8}{v^{2}}\right)}\right)e^{-4\sqrt{2}\vert t\vert v}.
\end{equation}
\par Next, to simplify more our notation, we denote
\begin{multline}\label{nonlinear11}
    NL(t)= -\left\langle H^{'}_{0,1}(x),\,R_{1}(t,x)+R_{2}(t,x)\right\rangle+c_{1}d(t)e^{-2\sqrt{2}d(t)}+c_{2}e^{-2\sqrt{2}d(t)}\\
    {-}4\left[\exp\left(\frac{-\sqrt{2}d(t)}{\left(1-\frac{\dot d(t)^{2}}{4}\right)^{\frac{1}{2}}}\right)-\frac{e^{-\sqrt{2}d(t)}}{\sqrt{1-\frac{\dot d(t)^{2}}{4}}}\right].
\end{multline}
Using the variation of parameters technique, we can write any $C^{2}$ solution $r(t)$ of \eqref{odeasy} as $r(t)=\theta_{1}(t)sol_{1}(t)+\theta_{2}(t)sol_{2}(t)$ such that $\theta_{1}(t)$ and $\theta_{2}(t)$ satisfy for any $t\in\mathbb{R}$
\begin{equation*}
    \begin{bmatrix}
    sol_{1}(t) && sol_{2}(t)\\
    \dot sol_{1}(t) && \dot sol_{2}(t)
    \end{bmatrix}
    \begin{bmatrix}
    \dot \theta_{1}(t)\\
    \dot \theta_{2}(t)
    \end{bmatrix}=\frac{1}{\norm{ H^{'}_{0,1}}_{L^{2}_{x}}^{2}}
    \begin{bmatrix}
    0\\
    NL(t)
    \end{bmatrix}=\begin{bmatrix}
    0\\
    2\sqrt{2} NL(t)
    \end{bmatrix}.
\end{equation*}
In conclusion,
since for all $t\in\mathbb{R}$
\begin{equation*}
    \det \begin{bmatrix}
    sol_{1}(t) && sol_{2}(t)\\
    \dot sol_{1}(t) && \dot sol_{2}(t)
    \end{bmatrix}=\sqrt{2}v,
\end{equation*}
we have
\begin{equation}\label{deriv}
    \dot \theta_{2}(t)=\frac{2}{v}NL(t)\tanh{(\sqrt{2}vt)},\,\dot \theta_{1}(t)= \frac{{-}2}{v}NL(t)\left[\sqrt{2}vt\tanh{(\sqrt{2}vt)}-1\right].
\end{equation}
From Lemmas \ref{dt2kink} and \ref{LambdaG}, we have that $R_{1}(t,x)$ and $R_{2}(t,x)$ are even in $t,$ so $NL(t)$ is also even. Since we are interested in an even solution $r(t)$ of \eqref{odeasy}, we need $\theta_{1}$ odd and $\theta_{2}$ even, so
we must choose
\begin{equation}\label{0deriv}
    \theta_{2}(t)=\frac{1}{\sqrt{2}v}\int_{-\infty}^{t}NL(s)\tanh{(\sqrt{2}vs)}\,ds,\,\theta_{1}(t)=\frac{-1}{\sqrt{2}v}\int_{0}^{t}NL(s) \left[\sqrt{2}vs\tanh{(\sqrt{2}vs)}-1\right]\,ds.
\end{equation}

 \par From Lemmas \ref{dt2kink} and \ref{LambdaG}, we deduce for any $j\in\{1,2\}$ that if $0<v\ll 1,$ then
\begin{equation}\label{Rdecay}
   \left\vert \frac{d^{l}}{dt^{l}}\left\langle R_{j}(t,x),\, H^{'}_{0,1}(x) \right\rangle\right\vert\lesssim_{l} v^{4+l}\sech{\left(\sqrt{2}vt\right)}^{2} \text{ for all $l\in\mathbb{N}\cup\{0\},$}
\end{equation}
and so, from the equations \eqref{dd(t)},\eqref{nonlinear11} and Lemma \ref{explemma}, we deduce for all $0<v\ll 1$ and any $l\in\mathbb{N}\cup\{0\}$ that
\begin{equation}\label{4.8}
    \left\vert \frac{d^{l}}{dt^{l}}NL(t)\right\vert\lesssim_{l} v^{4+l}\left(v\vert t\vert +\ln{\left(\frac{1}{v}\right)}\right)e^{{-}2\sqrt{2}\vert t\vert v}.
\end{equation}
Therefore, from the definition of $d(t),$ the identities \eqref{nonlinear11}, \eqref{deriv} and the estimates \eqref{Rdecay}, \eqref{4.8}, using the Fundamental Theorem of Calculus, we deduce the existence of a constant $C>0$ such that if $0<v\ll 1,$ then
\begin{equation}\label{a}
\norm{\theta_{1}}_{L^{\infty}(\mathbb{R})}<C v^{2}\ln{\left(\frac{1}{v^{2}}\right)}.
\end{equation}
Furthermore, since $NL(t)$ is an even function and $\tanh{(\sqrt{2}vs)}$ is an odd function, we have that
\begin{equation*}
    \int_{-\infty}^{t}NL(s)\tanh{\left(\sqrt{2}vs\right)}\,ds=-\int_{t}^{+\infty}NL(s)\tanh{\left(\sqrt{2}vs\right)}\,ds,
\end{equation*} from which with identity \eqref{0deriv}, we deduce the following estimate
\begin{equation}\label{idb}
    \lvert \theta_{2}(t)\rvert\leq \frac{1}{\sqrt{2}v}\int_{\vert t\vert}^{+\infty}\lvert NL(s)\rvert \tanh{\left(\sqrt{2}vs\right)}\,ds \text{, for all $t\in \mathbb{R}.$}
\end{equation}
Therefore, the estimate \eqref{4.8} implies that
$
\lvert \theta_{2}(t)\rvert 
    \lesssim v^{2}\left[\ln{\left(\frac{8}{v^{2}}\right)}+v\vert t\vert \right]e^{-2\sqrt{2}\vert t\vert v},$ for any $t\in\mathbb{R}.$
\par Finally, since $r(t)=\theta_{1}(t)sol_{1}(t)+\theta_{2}(t)sol_{2}(t)$ and $\dot r(t)=\theta_{1}(t)\dot sol_{1}(t)+\theta_{2}(t)\dot sol_{2}(t),$ we deduce for all $t\in \mathbb{R}$ that
\begin{equation}\label{r,dr,k=1}
    \lvert r(t)\rvert \lesssim v^{2}\ln{\left(\frac{1}{v^{2}}\right)},
\, \lvert \dot r(t)\rvert\lesssim v^{3}\left[\ln{\left(\frac{1}{v^{2}}\right)}+\vert t\vert v\right]\sech{\left(\sqrt{2}vt\right)}^{2}.
\end{equation}
 Moreover, \eqref{4.8} and the definitions of $sol_{1}$ and $sol_{2},$ we can verify by induction on $l\in \mathbb{N}$ for any $0<v\ll 1$ that
\begin{equation}\label{geraldr}
     \left\vert \frac{ d^{l}r}{dt^{l}}(t)\right\vert \lesssim_{l} v^{l+2}\left[\ln{\left(\frac{1}{v^{2}}\right)}+\vert t \vert v\right]\sech{(\sqrt{2}vt)}^{2} \text{ for all integers $l\geq 1$ and $t\in \mathbb{R}.$}
\end{equation}
\textbf{Step 2.}(Estimate of $\Lambda(\varphi_{2})(t,x).$)
From now on, we define the function $w_{1}:\mathbb{R}^{2}\to\mathbb{R}$ as the unique function satisfying
\begin{equation*}
    w_{1}(t,x)=w_{0}(t,x+r_{v}(t))=\frac{x-\frac{d_{v}(t)}{2}+r_{v}(t)}{\sqrt{1-\frac{\dot d_{v}(t)^{2}}{4}}},
\end{equation*}
 for every $(t,x)\in\mathbb{R}^{2}.$
 Furthermore, similarly to the identity \eqref{sumbig0}, we have the following equation  
  \begin{align}\label{H1}
   \Lambda(\varphi_{2})(t,x)=&\Lambda( H_{0,1}\left(w_{1}(t,x)\right))-\Lambda\left( H_{0,1}\left(w_{1}(t,{-}x)\right)\right)\\ \label{linearg1} &{+}\left[\frac{\partial^{2}}{\partial t^{2}}-\frac{\partial^{2}}{\partial x^{2}}\right]\left[e^{{-}\sqrt{2}d(t)}\mathcal{G}^{w_{1}}(t,x)\right]+ \left[U^{''}\left(H_{0,1}\right)\mathcal{G}\right]^{w_{1}}(t,x)e^{{-}\sqrt{2}d(t)}\\ \label{interg1}
&{+}U^{''}\left(H^{w_{1}}_{0,1}(t,x)\right)\mathcal{G}^{w_{1}}(t,x)e^{{-}\sqrt{2}d(t)}-\left[ U^{''}\left(H_{0,1}\right)\mathcal{G}\right]^{w_{1}}(t,x)e^{{-}\sqrt{2}d(t)}
   \\\label{nolint1}
&{+}U^{'}\left(H^{w_{1}}_{0,1}(t,x)\right)
   -U^{'}\left(H_{0,1}\left(w_{1}(t,x)\right)\right)-U^{'}\left({-}H_{0,1}\left(w_{1}(t,{-}x)\right)\right)
  \\ \label{g31} &{+}\frac{U^{(3)}\left(H^{w_{1}}_{0,1}(t,x)\right)}{2}\left[e^{{-}\sqrt{2}d(t)}\mathcal{G}^{w_{1}}\left(t,x\right)\right]^{2}
  \\ \label{sumbig1} &{+}\sum_{j=4}^{6}\frac{U^{(j)}\left(H^{w_{1}}_{0,1}(t,x)\right)}{(j-1)!}\left[e^{{-}\sqrt{2}d(t)}\mathcal{G}^{w_{1}}\left(t,x\right)\right]^{(j-1)}.
\end{align}
\par From identity $\norm{H^{'}_{0,1}}_{L^{2}_{x}}^{2}=\frac{1}{2\sqrt{2}},$ the definitions of $M(x),\,N(x)$ in \eqref{def1} and identity \eqref{oioi}, we have
\begin{align*}
     \left \langle\left[24M(w_{0}(t,x))-30N(w_{0}(t,x))\right],  H^{'}_{0,1}(w_{0}(t,x))\right\rangle=4\sqrt{1-\frac{\dot d(t)^{2}}{4}}.
\end{align*}
Therefore, we deduce the following identity 
\begin{gather}\nonumber
     \exp\left(\frac{-\sqrt{2}(d(t)-2r(t))}{\sqrt{1-\frac{\dot d(t)^{2}}{4})}}\right)\left \langle24M(w_{0}(t,x))-30N(w_{0}(t,x)), H^{'}_{0,1}(w_{0}(t,x))\right\rangle-4e^{{-}\sqrt{2}d(t)}
     \\ \nonumber
   = 4\exp\left(\frac{-\sqrt{2}(d(t)-2r(t))}{\sqrt{1-\frac{\dot d(t)^{2}}{4}}}\right)\sqrt{1-\frac{\dot d(t)^{2}}{4}}-4e^{{-}\sqrt{2}d(t)}
      \\ \label{expp1}
      =4\left[\exp\left(\frac{-\sqrt{2}(d(t)-2r(t))}{\sqrt{1-\frac{\dot d(t)^{2}}{4}}}\right)-e^{-\sqrt{2}(d(t)-2r(t))}\right]\sqrt{1-\frac{\dot d(t)^{2}}{4}}\\ \label{expp2}
     {+}4e^{-\sqrt{2}(d(t)-2r(t))}\left[\sqrt{1-\frac{\dot d(t)^{2}}{4}}-1\right]
   \\ \label{expp3}
     {+}4\left[e^{-\sqrt{2}(d(t)-2r(t))}-e^{-\sqrt{2}d(t)}-2\sqrt{2}e^{-\sqrt{2}d(t)}r(t)\right]
     \\ \nonumber
     {+}8\sqrt{2}e^{-\sqrt{2}d(t)}r(t).
\end{gather}
Since $e^{{-}\sqrt{2}d(t)}=\frac{v^{2}}{8}\sech{\left(\sqrt{2}vt\right)}^{2},$ using estimates \eqref{r,dr,k=1} and \eqref{geraldr} of the function $r,$ we deduce from an application of Lemma \ref{explemma} in the expressions \eqref{expp1}, \eqref{expp2} and from an application of Taylor's Expansion Theorem in the term \eqref{expp3} that the following function
\begin{multline*}
   Rem(t)=\exp\left(\frac{-\sqrt{2}(d(t)-2r(t))}{\sqrt{1-\frac{\dot d(t)^{2}}{4}}}\right)\left\langle 24M(w_{0}(t,x))-30N(w_{0}(t,x)),\, H^{'}_{0,1}(w_{0}(t,x))\right\rangle \\{-}4e^{-\sqrt{2}d(t)}-8\sqrt{2}e^{-\sqrt{2}d(t)}r(t)
\end{multline*}
satisfies 
$
    \left \vert \frac{d^{l}Rem(t)}{dt^{l}}\right\vert\lesssim_{l} v^{l+4}\left[\left\vert t\right\vert v+\ln{\left(\frac{8}{v^{2}}\right)}\right]e^{-2\sqrt{2}v\vert t\vert }
$ for all $t\in\mathbb{R}$ and any $l\in\mathbb{N}\cup\{0\}.$\\
\textbf{Substep 2.1.}(Estimate of $\Lambda\left(H_{0,1}(w_{1}(t,x))\right).$)
 From now on, we use the following notation
\begin{equation*}
\varphi_{2}(t,x)=H^{w_{1}}_{0,1}(t,x)+e^{-\sqrt{2}d(t)}\mathcal{G}^{w_{1}}(t,x) \text{, for all $(t,x)\in\mathbb{R}^{2}.$}
\end{equation*}
First, for all $(t,x)\in\mathbb{R}^{2},$ the  following identity
\begin{align*}
    \frac{\partial^{2}}{\partial t^{2}}H_{0,1}(w_{1}(t,x))=&\frac{\partial^{2}}{\partial t_{1}^{2}}\Big\vert_{t_{1}=t}H_{0,1}(w_{0}(t_{1},x+r(t)))+\frac{\ddot r(t)}{\sqrt{1-\frac{\dot d(t)^{2}}{4}}} H^{'}_{0,1}(w_{1}(t,x))
    \\&{-}\frac{\dot d(t)\dot r(t)}{1-\frac{\dot d(t)^{2}}{4}}H^{''}_{0,1}(w_{1}(t,x))+\frac{8\sqrt{2}\dot d(t)\dot r(t)e^{{-}\sqrt{2}d(t)}}{\left(1-\frac{\dot d(t)^{2}}{4}\right)^{\frac{3}{2}}}w_{1}(t,x) H^{'}_{0,1}(w_{1}(t,x))\\
    &{+}\frac{8\sqrt{2}\dot r(t)\dot d(t)e^{{-}\sqrt{2}d(t)}}{\left(1-\frac{\dot d(t)^{2}}{4}\right)^{\frac{3}{2}}}w_{1}(t,x)H^{''}_{0,1}(w_{1}(t,x))
\end{align*}
implies with the product rule, estimates \eqref{mmma}, \eqref{r,dr,k=1}, \eqref{geraldr}, Lemmas \ref{dlemma} and Remark \ref{perturbt}
that
\begin{align*}
    \frac{\partial^{2}}{\partial t^{2}}H_{0,1}(w_{1}(t,x))\cong_{6}\frac{\partial^{2}}{\partial t_{1}^{2}}\Big\vert_{t_{1}=t}H_{0,1}(w_{0}(t_{1},x+r(t)))+\frac{\ddot r(t)}{\sqrt{1-\frac{\dot d(t)^{2}}{4}}} H^{'}_{0,1}(w_{1}(t,x))
    \\{-}\frac{\dot d(t)\dot r(t)}{1-\frac{\dot d(t)^{2}}{4}} H^{''}_{0,1}(w_{1}(t,x)).
\end{align*}
Therefore, from Lemma \ref{porrataylor}, we deduce from the estimate above and the decay estimates \eqref{r,dr,k=1}, \eqref{geraldr} of $r$  that
\begin{multline}\label{Hr}
\frac{\partial^{2}}{\partial t^{2}}H_{0,1}(w_{1}(t,x))\cong_{6}\frac{\partial^{2}}{\partial t_{1}^{2}}\Big\vert_{t_{1}=t}H_{0,1}(w_{0}(t_{1},x+r(t)))+\frac{\ddot r(t)}{\sqrt{1-\frac{\dot d(t)^{2}}{4}}} H^{'}_{0,1}(w_{0}(t,x))
    \\{-}\frac{\dot d(t)\dot r(t)}{1-\frac{\dot d(t)^{2}}{4}}H^{''}_{0,1}(w_{0}(t,x)).
\end{multline}
Moreover, Lemma \ref{dt2kink} implies that
\begin{multline*}
    \frac{\partial^{2}}{\partial t_{1}^{2}}\Big\vert_{t_{1}=t}H_{0,1}(w_{0}(t_{1},x+r(t)))-\frac{\partial^{2}}{\partial x^{2}}\left[ H_{0,1}\left(w_{1}(t,x)\right)\right]+ U^{'}\left(H_{0,1}\left(w_{1}(t,x)\right)\right)\\
    ={-}\frac{8\sqrt{2}e^{-\sqrt{2}d(t)}}{(1-\frac{\dot d(t)^{2}}{4})^{\frac{1}{2}}} H^{'}_{0,1}\left(w_{1}(t,x)\right)+R_{1}\left(t,w_{1}(t,x)\right),
\end{multline*}
from which with Lemma \ref{porrataylor} and estimates \eqref{r,dr,k=1}, \eqref{geraldr}, we obtain the following estimate
\begin{multline}\label{lambdanew}
\frac{\partial^{2}}{\partial t_{1}^{2}}\Big\vert_{t_{1}=t}H_{0,1}(w_{0}(t_{1},x+r(t)))-\frac{\partial^{2}}{\partial x^{2}}\left[ H_{0,1}\left(w_{1}(t,x)\right)\right]+U^{'}\left(H_{0,1}\left(w_{1}(t,x)\right)\right)\\
    \cong_{6}{-}\frac{8\sqrt{2}e^{-\sqrt{2}d(t)}}{\sqrt{1-\frac{\dot d(t)^{2}}{4}}} H^{'}_{0,1}\left(w_{0}(t,x)\right)-\frac{8\sqrt{2}r(t) e^{{-}\sqrt{2}d(t)} }{\sqrt{1-\frac{\dot d(t)^{2}}{4}}} H^{''}_{0,1}\left(w_{0}(t,x)\right)+R_{1}\left(t,w_{0}(t,x)\right).    
\end{multline}
Therefore, we obtain using estimates \eqref{Hr} and \eqref{lambdanew} that
\begin{multline*}
    \Lambda\left(H_{0,1}\left(w_{1}(t,x)\right)\right)\cong_{6}
    {-}\frac{8\sqrt{2}e^{{-}\sqrt{2}d(t)}}{\sqrt{1-\frac{\dot d(t)^{2}}{4}}}H^{'}_{0,1}\left(w_{0}(t,x)\right)-\frac{8\sqrt{2}r(t)e^{{-}\sqrt{2}d(t)}}{\sqrt{1-\frac{\dot d(t)^{2}}{4}}} H^{''}_{0,1}\left(w_{0}(t,x)\right)
\\{+}\frac{\ddot r(t)}{\sqrt{1-\frac{\dot d(t)^{2}}{4}}} H^{'}_{0,1}\left(w_{0}(t,x)\right)-\frac{\dot d(t)\dot r(t)}{1-\frac{\dot d(t)^{2}}{4}}H^{''}_{0,1}\left(w_{0}(t,x)\right)+R_{1}(t,w_{0}(t,x)).
\end{multline*}
Consequently, using Lemma \ref{dt2kink}, we deduce the following estimate
\begin{align}\nonumber
\Lambda\left(H_{0,1}\left(w_{1}(t,x)\right)\right)-\Lambda\left(H_{0,1}\left(w_{1}(t,{-}x)\right)\right)\cong_{6}&\Lambda\left(H_{0,1}\left(w_{0}(t,x)\right)\right)-\Lambda\left(H_{0,1}\left(w_{0}(t,{-}x)\right)\right)
    \\ \nonumber &{-}\frac{8\sqrt{2}r(t)e^{{-}\sqrt{2}d(t)}}{\sqrt{1-\frac{\dot d(t)^{2}}{4}}} \left(H^{''}_{0,1}\right)^{w_{0}}(t,x)
    \\\nonumber &{+}\frac{\ddot r(t)}{\sqrt{1-\frac{\dot d(t)^{2}}{4}}} \left(H^{'}_{0,1}\right)^{w_{0}}(t,x)\\ \label{estlambdaH} &{-}\frac{\dot d(t)\dot r(t)}{1-\frac{\dot d(t)^{2}}{4}} \left(H^{''}_{0,1}\right)^{w_{0}}\left(t,x\right).
\end{align}
\textbf{Substep 2.2.}(Estimate of \eqref{linearg1}.)
 Next, from Lemmas \ref{dlemma},  \ref{porrataylor}, we deduce with estimates \eqref{r,dr,k=1}, \eqref{geraldr} and the product rule that
\begin{align*}
    \frac{\partial^{2}}{\partial t^{2}}\left[e^{{-}\sqrt{2}d(t)}\mathcal{G}\left(w_{1}(t,x)\right)\right]\cong_{6} \frac{\partial^{2}}{\partial t^{2}}\left[e^{{-}\sqrt{2}d(t)}\mathcal{G}\left(w_{0}(t,x)\right)\right]\cong_{6}
    \frac{\partial^{2}}{\partial t_{1}^{2}}\Big\vert_{t_{1}=t}\left[e^{{-}\sqrt{2}d(t_{1})}\mathcal{G}(w_{0}(t_{1},x+r(t)))\right].
\end{align*}
Therefore, we deduce from
Lemma \ref{LambdaG} the following estimate
\begin{multline*}
    \left[ \frac{\partial^{2}}{\partial t^{2}}-\frac{\partial^{2}}{\partial x^{2}}+ U^{''}\left(H_{0,1}(w_{1}(t,x))\right)\right]\left(e^{-\sqrt{2}d(t)}\mathcal{G}(w_{1}(t,x))\right)\\\cong_{6}{-}\Big[24M\left(w_{1}(t,x)\right)-30N\left(w_{1}(t,x)\right)\Big]e^{{-}\sqrt{2}d(t)}{+}8\sqrt{2} H^{'}_{0,1}\left(w_{1}(t,x)\right)e^{{-}\sqrt{2}d(t)}\\{+}R_{2}(t,w_{1}(t,x)),
\end{multline*}
from which with Lemma \ref{porrataylor} and the decay estimates \eqref{r,dr,k=1}, \eqref{geraldr} of $r,$ we deduce that
\begin{multline*}
     \left[ \frac{\partial^{2}}{\partial t^{2}}-\frac{\partial^{2}}{\partial x^{2}}+ U^{''}\left(H_{0,1}(w_{1}(t,x))\right)\right]\left(e^{-\sqrt{2}d(t)}\mathcal{G}(w_{1}(t,x))\right)\\ 
     \begin{aligned}
     \cong_{6}&{-}\Big[24M\left(w_{0}(t,x)\right)-30N\left(w_{0}(t,x)\right)\Big]e^{{-}\sqrt{2}d(t)}+8\sqrt{2} H^{'}_{0,1}\left(w_{0}(t,x)\right)e^{{-}\sqrt{2}d(t)}\\ 
     &{-}\Big[24 M^{'}\left(w_{0}(t,x)\right)-30 N^{'}\left(w_{0}(t,x)\right)\Big]\frac{r(t)e^{{-}\sqrt{2}d(t)}}{\sqrt{1-\frac{\dot d(t)^{2}}{4}}}+\frac{8\sqrt{2}r(t)e^{{-}\sqrt{2}d(t)}}{\sqrt{1-\frac{\dot d(t)^{2}}{4}}} H^{''}_{0,1}\left(w_{0}(t,x)\right)
     \\ &{+}R_{2}(t,w_{0}(t,x)).
    \end{aligned}
    \end{multline*}
   Hence, using Lemma \ref{LambdaG}, we obtain the following estimate
    \begin{multline} \label{lambdagnew}
    \left[ \frac{\partial^{2}}{\partial t^{2}}-\frac{\partial^{2}}{\partial x^{2}}+U^{''}\left(H_{0,1}(w_{1}(t,x))\right)\right]\left(e^{-\sqrt{2}d(t)}\mathcal{G}(w_{1}(t,x))\right)\\
    \cong_{6}\left[\frac{\partial^{2}}{\partial t^{2}}-\frac{\partial^{2}}{\partial x^{2}}+U^{''}\left(H_{0,1}(w_{0}(t,x))\right)\right]\left(e^{{-}\sqrt{2}d(t)}\mathcal{G}\left(w_{0}(t,x)\right)\right)\\ 
    {-}\Big[24 M^{'}\left(w_{0}(t,x)\right)-30N^{'}\left(w_{0}(t,x)\right)\Big]\frac{r(t)e^{{-}\sqrt{2}d(t)}}{\sqrt{1-\frac{\dot d(t)^{2}}{4}}}\\{+}\frac{8\sqrt{2}r(t)e^{{-}\sqrt{2}d(t)}}{\sqrt{1-\frac{\dot d(t)^{2}}{4}}} H^{''}_{0,1}\left(w_{0}(t,x)\right) .
\end{multline}
\textbf{Substep 2.3.}(Estimate of \eqref{nolint1}.)
 In notation of Lemma \ref{l1}, we have the following identity
\begin{multline}\label{intnew242}
U^{'}\left(H^{w_{1}}_{0,1}(t,x)\right)- U^{'}\left(H_{0,1}(w_{1}(t,x))\right)- U^{'}\left(-H_{0,1}(w_{1}(t,{-}x))\right)\\ 
\begin{aligned}
     =&\exp\left(-\frac{\sqrt{2}(d(t)-2r(t))}{\sqrt{1-\frac{\dot d(t)^{2}}{4}}}\right)\left[24M^{w_{1}}(t,x)-30N^{w_{1}}(t,x)\right]\\ 
     &{+}\exp\left(-\frac{2\sqrt{2}(d(t)-2r(t))}{\sqrt{1-\frac{d(t)^{2}}{4})}}\right)\left[24V^{w_{1}}(t.x)+\frac{60}{\sqrt{2}} \left(H^{'}_{0,1}\right)^{w_{1}}(t,x)\right]\\&{+}R\left({-}w_{1}(t,x),\frac{d(t)-2r(t)}{\sqrt{1-\frac{\dot d(t)^{2}}{4}}}\right).
\end{aligned}
\end{multline}
Moreover, similarly to the proof of Remark \ref{reint}, Lemmas \ref{porrataylor} and \ref{l1} imply that
\begin{equation*}
    R\left({-}w_{1}(t,x),\frac{d(t)-2r(t)}{\sqrt{1-\frac{\dot d(t)^{2}}{4}}}\right)\cong_{6} 0.
\end{equation*}
Therefore, identity \eqref{intnew242} and Lemmas \ref{porrataylor}, \ref{explemma} imply the following estimate
\begin{multline}\label{intnew2}
U^{'}\left(H^{w_{1}}_{0,1}(t,x)\right)- U^{'}\left(H_{0,1}(w_{1}(t,x))\right)-U^{'}\left(-H_{0,1}(w_{1}(t,{-}x))\right)\\
   \cong_{6} \exp\left(-\frac{\sqrt{2}(d(t)-2r(t))}{\sqrt{1-\frac{\dot d(t)^{2}}{4}}}\right)\left[24M^{w_{1}}(t,x)-30N^{w_{1}}(t,x)\right]
     \\{+}e^{{-}2\sqrt{2}d(t)}\left[24V^{w_{0}}(t,x)+\frac{60}{\sqrt{2}} \left(H^{'}_{0,1}\right)^{w_{0}}(t,x)\right].
\end{multline}
\par Next, using the decay estimates \eqref{r,dr,k=1}, \eqref{geraldr} of $r$, we deduce from Lemma \ref{porrataylor} that
\begin{gather} \label{mexp}
    \left[M^{w_{1}}(t,x)-M^{w_{0}}(t,x)-\frac{r(t)}{\sqrt{1-\frac{\dot d(t)^{2}}{4}}}\left(M^{'}\right)^{w_{0}}(t,x)\right]\exp\left(\frac{{-}\sqrt{2}d(t)}{\sqrt{1-\frac{\dot d(t)^{2}}{4}}}\right)\cong_{6} 0,\\ \label{nexp}
    \left[N^{w_{1}}(t,x)-N^{w_{0}}(t,x)-\frac{r(t)}{\sqrt{1-\frac{\dot d(t)^{2}}{4}}}\left(N^{'}\right)^{w_{0}}(t,x)\right]\exp\left(\frac{{-}\sqrt{2}d(t)}{\sqrt{1-\frac{\dot d(t)^{2}}{4}}}\right)\cong_{6} 0.
\end{gather}
We also deduce from Taylor's Expansion Theorem and the decay  estimates \eqref{r,dr,k=1}, \eqref{geraldr} of the function $r$ that
\begin{align*}
    M^{w_{1}}(t,x)\exp\left(\frac{{-}\sqrt{2}(d(t)-2r(t))}{\sqrt{1-\frac{\dot d(t)^{2}}{4}}}\right)\cong_{6} M^{w_{1}}(t,x)\left[1+\frac{2r(t)}{\sqrt{1-\frac{\dot d(t)^{2}}{4}}}\right]\exp\left(\frac{{-}\sqrt{2}d(t)}{\sqrt{1-\frac{\dot d(t)^{2}}{4}}}\right)\\
    N^{w_{1}}(t,x)\exp\left(\frac{{-}\sqrt{2}(d(t)-2r(t))}{\sqrt{1-\frac{\dot d(t)^{2}}{4}}}\right)\cong_{6} N^{w_{1}}(t,x)\left[1+\frac{2r(t)}{\sqrt{1-\frac{\dot d(t)^{2}}{4}}}\right]\exp\left(\frac{{-}\sqrt{2}d(t)}{\sqrt{1-\frac{\dot d(t)^{2}}{4}}}\right),
\end{align*}
 therefore, using now Lemma \ref{porrataylor}, we conclude the following estimates
\begin{align*}
    M^{w_{1}}(t,x)\exp\left(\frac{{-}\sqrt{2}(d(t)-2r(t))}{\sqrt{1-\frac{\dot d(t)^{2}}{4}}}\right)\cong_{6} & M^{w_{1}}(t,x)\exp\left(\frac{{-}\sqrt{2}d(t)}{\sqrt{1-\frac{\dot d(t)^{2}}{4}}}\right)\\&{+}M^{w_{0}}(t,x)\frac{2r(t)}{\sqrt{1-\frac{\dot d(t)^{2}}{4}}}\exp\left(\frac{{-}\sqrt{2}d(t)}{\sqrt{1-\frac{\dot d(t)^{2}}{4}}}\right),\\
    N^{w_{1}}(t,x)\exp\left(\frac{{-}\sqrt{2}(d(t)-2r(t))}{\sqrt{1-\frac{\dot d(t)^{2}}{4}}}\right)\cong_{6} & N^{w_{1}}(t,x)\exp\left(\frac{{-}\sqrt{2}d(t)}{\sqrt{1-\frac{\dot d(t)^{2}}{4}}}\right)\\&{+}N^{w_{0}}(t,x)\frac{2r(t)}{\sqrt{1-\frac{\dot d(t)^{2}}{4}}}\exp\left(\frac{{-}\sqrt{2}d(t)}{\sqrt{1-\frac{\dot d(t)^{2}}{4}}}\right).
\end{align*}
As a consequence, we obtain from estimate \eqref{intnew242} and Lemma \ref{explemma} that
\begin{multline*}
U^{'}\left(H^{w_{1}}_{0,1}(t,x)\right)
- U^{'}\left(H_{0,1}\left(w_{1}(t,x)\right)\right)- U^{'}\left({-}H_{0,1}\left(w_{1}(t,{-}x)\right)\right)\\ 
\begin{aligned}
\cong_{6} &
\left[24M^{w_{0}}(t,x)-30N^{w_{0}}(t,x)\right]\frac{2r(t)}{\sqrt{1-\frac{\dot d(t)^{2}}{4}}}\exp\left(\frac{{-}\sqrt{2}d(t)}{\sqrt{1-\frac{\dot d(t)^{2}}{4}}}\right)\\ 
&{+}\left[24M^{w_{0}}(t,x)-30N^{w_{0}}(t,x)\right]\exp\left(\frac{{-}\sqrt{2}d(t)}{\sqrt{1-\frac{\dot d(t)^{2}}{4}}}\right)
\\ 
&{+}\frac{r(t)}{\sqrt{1-\frac{\dot d(t)^{2}}{4}}}\left[24 \left(M^{'}\right)^{w_{0}}(t,x)-30 \left(N^{'}\right)^{w_{0}}(t,x)\right]\exp\left(\frac{{-}\sqrt{2}d(t)}{\sqrt{1-\frac{\dot d(t)^{2}}{4}}}\right)
\\&{+}\left[24V^{w_{0}}(t,x)+\frac{60}{\sqrt{2}}\left(H^{'}_{0,1}\right)^{w_{0}}(t,x)\right]e^{{-}2\sqrt{2}d(t)}\\  
\cong_{6} & U^{'}\left(H^{w_{0}}_{0,1}(t,x)\right)- U^{'}\left(H_{0,1}\left(w_{0}(t,x)\right)\right)-U^{'}\left({-}H_{0,1}\left(w_{0}(t,{-}x)\right)\right)\\ 
&{+}\frac{r(t)}{\sqrt{1-\frac{\dot d(t)^{2}}{4}}}\left[24\left( M^{'}\right)^{w_{0}}(t,x)-30\left( N^{'}\right)^{w_{0}}(t,x)\right]\exp\left(\frac{{-}\sqrt{2}d(t)}{\sqrt{1-\frac{\dot d(t)^{2}}{4}}}\right)\\&{+}\frac{r(t)}{\sqrt{1-\frac{\dot d(t)^{2}}{4}}}\left[48M^{w_{0}}(t,x)-60N^{w_{0}}(t,x)\right]\exp\left(\frac{{-}\sqrt{2}d(t)}{\sqrt{1-\frac{\dot d(t)^{2}}{4}}}\right).
\end{aligned}
\end{multline*}
Therefore, using Remark \ref{reint}, we conclude that  
\begin{multline}\label{u-uinteract}
U^{'}\left(H^{w_{1}}_{0,1}(t,x)\right)
-U^{'}\left(H_{0,1}\left(w_{1}(t,x)\right)\right)- U^{'}\left({-}H_{0,1}\left(w_{1}(t,{-}x)\right)\right) \\ 
\begin{aligned}
\cong_{6}& 
U^{'}\left(H^{w_{0}}_{0,1}(t,x)\right)- U^{'}\left(H_{0,1}\left(w_{0}(t,x)\right)\right)- U^{'}\left({-}H_{0,1}\left(w_{0}(t,{-}x)\right)\right)\\   
&{+}\frac{r(t)e^{{-}\sqrt{2}d(t)}}{\sqrt{1-\frac{\dot d(t)^{2}}{4}}}\left[24\left( M^{'}\right)^{w_{0}}(t,x)-30\left( N^{'}\right)^{w_{0}}(t,x)\right]\\&{+}\frac{r(t)e^{{-}\sqrt{2}d(t)}}{\sqrt{1-\frac{\dot d(t)^{2}}{4}}}\left[48M^{w_{0}}(t,x)-60N^{w_{0}}(t,x)\right].
\end{aligned}
\end{multline}
\textbf{Substep 2.4.}(Estimate of \eqref{interg1}.)
\par Now, using identities \eqref{QQ} and \eqref{Rq}, Lemma \ref{interg} also implies the following equation
\begin{multline}\label{int2new}
    \left[ U^{''}\left(H^{w_{1}}_{0,1}(t,x)\right)- U^{''}\left(H_{0,1}(w_{1}(t,x))\right)\right]e^{-\sqrt{2}d(t)}\mathcal{G}(w_{1}(x,t))\\
    =\mathcal{Q}\left({-}w_{1}(t,{-}x),\frac{d_{v}(t)-2r(t)}{\sqrt{1-\frac{\dot d(t)^{2}}{4}}}\right)e^{{-}\sqrt{2}d(t)}
  +R_{q}\left({-}w_{1}(t,{-}x),\frac{d(t)-2r(t)}{\sqrt{1-\frac{\dot d(t)^{2}}{4}}}\right)e^{{-}\sqrt{2}d(t)}.
\end{multline}
Furthermore, from Lemma \ref{porrataylor} and the definition of $\mathcal{Q}$ in \eqref{QQ}, we deduce that
\begin{equation*}
     \mathcal{Q}\left({-}w_{1}(t,{-}x),\frac{d(t)-2r(t)}{\sqrt{1-\frac{\dot d(t)^{2}}{4}}}\right)e^{{-}\sqrt{2}d(t)}\cong_{6}   \mathcal{Q}\left({-}w_{0}(t,{-}x),\frac{d(t)-2r(t)}{\sqrt{1-\frac{\dot d(t)^{2}}{4}}}\right)e^{{-}\sqrt{2}d(t)},
\end{equation*}
from which with Lemmas \ref{dlemma}, \ref{explemma} and identity \eqref{QQ}, we obtain that
\begin{equation}\label{Qf}    
\mathcal{Q}\left({-}w_{1}(t,{-}x),\frac{d(t)-2r(t)}{\sqrt{1-\frac{\dot d(t)^{2}}{4}}}\right)e^{{-}\sqrt{2}d(t)}\cong_{6} \mathcal{Q}\left({-}w_{0}(t,{-}x),d(t)\right)e^{{-}\sqrt{2}d(t)}.
\end{equation}
Using identity \eqref{Rq} and Remark \ref{remarkinterg}, we can deduce similarly to the proof of estimate \eqref{Qf} that 
\begin{equation*}
   R_{q}\left({-}w_{1}(t,{-}x),\frac{d(t)-2r(t)}{\sqrt{1-\frac{\dot d(t)^{2}}{4}}}\right)e^{{-}\sqrt{2}d(t)}\cong_{6}  R_{q}\left({-}w_{0}(t,{-}x),d(t)\right)e^{{-}\sqrt{2}d(t)}\cong_{6} 0.
\end{equation*}
Consequently, in notation of Lemma \ref{interg}, we have from identity \eqref{QQ} that
\begin{multline*}
     \left[ U^{''}\left(H_{0,1}(w_{1}(t,x))-H_{0,1}(w_{1}(t,{-}x))\right)-U^{''}\left(H_{0,1}(w_{1}(t,x))\right)\right]e^{-\sqrt{2}d(t)}\mathcal{G}(w_{1}(t,x)) \\ \cong_{6}w_{0}(t,x)A(w_{0}(t,x))e^{-2\sqrt{2}d(t)}
    +w_{0}(t,x)B(w_{0}(t,{-}x))e^{-2\sqrt{2}d(t)}+C(w_{0}(t,x))e^{-2\sqrt{2}d(t)}\\{+}D(w_{0}(t,{-}x))e^{-2\sqrt{2}d(t)},
\end{multline*}
from which, using Remark \ref{remarkinterg}, we deduce
\begin{multline*}  
\left[ U^{''}\left(H_{0,1}(w_{1}(t,x))-H_{0,1}(w_{1}(t,{-}x))\right)- U^{''}\left(H_{0,1}(w_{1}(t,x))\right)\right]e^{-\sqrt{2}d(t)}\mathcal{G}(w_{1}(t,x))\\  \cong_{6}\left[U^{''}\left(H^{w_{0}}_{0,1}(t,x)\right)- U^{''}\left(H_{0,1}(w_{0}(t,x))\right)\right]e^{-\sqrt{2}d(t)}\mathcal{G}(w_{0}(t,x)).
\end{multline*}
In conclusion, since $ U^{''}$ is an even function, we have
\begin{multline}\label{intergnew}
   e^{{-}\sqrt{2}d_{v}(t)} U^{''}\left(H^{w_{1}}_{0,1}(t,x)\right)\mathcal{G}^{w_{1}}(t,x) -e^{{-}\sqrt{2}d_{v}(t)}\left[ U^{''}\left(H_{0,1}\right)\mathcal{G}\right]^{w_{1}}(t,x)\\
   \cong_{6} e^{{-}\sqrt{2}d_{v}(t)} U^{''}\left(H^{w_{0}}_{0,1}(t,x)\right)\mathcal{G}^{w_{0}}(t,x) -e^{{-}\sqrt{2}d_{v}(t)}\left[ U^{''}\left(H_{0,1}\right)\mathcal{G}\right]^{w_{0}}(t,x).
\end{multline}
\textbf{Substep 2.5.}(Estimate of \eqref{g31}.)
  Next, using Lemma \ref{porrataylor}, we can verify that
 \begin{multline*}
      \frac{1}{2}U^{(3)}\left(H_{0,1}(w_{1}(t,x))-H_{0,1}(w_{1}(t,{-}x))\right)\left[\mathcal{G}(w_{1}(t,x))-\mathcal{G}(w_{1}(t,{-}x))\right]^{2}e^{-2\sqrt{2}d(t)}
        \\
        \cong_{6}\frac{1}{2}U^{(3)}\left(H_{0,1}(w_{0}(t,x))-H_{0,1}(w_{0}(t,{-}x))\right)\left[\mathcal{G}(w_{0}(t,x))-\mathcal{G}(w_{0}(t,{-}x))\right]^{2}e^{-2\sqrt{2}d(t)}.
 \end{multline*}
Therefore, from Remark \ref{g3}, we obtain 
\begin{equation}\label{w1g3}
    \frac{1}{2}U^{(3)}\left(H^{w_{1}}_{0,1}(t,x)\right)\left[\mathcal{G}^{w_{1}}(t,x)\right]^{2}e^{-2\sqrt{2}d(t)}\cong_{6}\frac{1}{2} \left[U^{(3)}\left(H_{0,1}\right)\mathcal{G}^{2}\right]^{w_{0}}(t,x)e^{{-}2\sqrt{2}d(t)}.
\end{equation}
 \textbf{Substep 2.6.}(Estimate of \eqref{sumbig1}.)
 Furthermore, similarly to the proof of estimate \eqref{w1g3}, we can verify that
 \begin{equation*}
     \sum_{j=4}^{6}\frac{U^{(j)}\left(H^{w_{1}}_{0,1}(t,x)\right)}{(j-1)!}\left[\mathcal{G}^{w_{1}}(t,x)\right]^{j-1}e^{-\sqrt{2}d(t)(j-1)}\cong_{6} \sum_{j=4}^{6}\frac{U^{(j)}\left(H^{w_{0}}_{0,1}(t,x)\right)}{(j-1)!}\left[\mathcal{G}^{w_{0}}(t,x)\right]^{j-1}e^{-\sqrt{2}d(t)(j-1)}.
 \end{equation*}
Hence, we obtain using Remark \ref{remarktaylorr} that
\begin{equation}\label{taylorw1}
     \sum_{j=4}^{6}\frac{U^{(j)}\left(H^{w_{1}}_{0,1}(t,x)\right)}{(j-1)!}\left[\mathcal{G}^{w_{1}}(t,x)\right]^{j-1}e^{-\sqrt{2}d(t)(j-1)}\cong_{6} 0.
\end{equation}
\textbf{Substep 2.7.}(Conclusion of estimate of $\Lambda\left(\varphi_{2}\right)(t,x).$)
 From identity \eqref{sumbig0}, after we use the estimates \eqref{estlambdaH}, \eqref{lambdagnew}, \eqref{intergnew}, \eqref{u-uinteract}, \eqref{w1g3}, \eqref{taylorw1} respectively, in the terms \eqref{H1}, \eqref{linearg1}, \eqref{interg1}, \eqref{nolint1},  \eqref{g31}, \eqref{sumbig1}, we obtain
 \begin{align}\nonumber
    \Lambda(\varphi_{2})(t,x)-\Lambda(\varphi_{2,0})(t,x)\cong_{6} &\frac{\ddot r(t)}{(1-\frac{\dot d(t)^{2}}{4})^{\frac{1}{2}}}\left( H^{'}_{0,1}\right)^{w_{0}}(t,x)
    {-}\frac{\dot d(t)\dot r(t)}{1-\frac{\dot d(t)^{2}}{4}}\left( H^{''}_{0,1}\right)^{w_{0}}(t,x)
    \\ \nonumber &{+}\frac{8\sqrt{2}r(t)e^{{-}\sqrt{2}d(t)}}{\sqrt{1-\frac{\dot d(t)^{2}}{4}}}\left( H^{''}_{0,1}\right)^{w_{0}}(t,x)\\ \nonumber
    &{+}\left[48M^{w_{0}}(t,x)-60N^{w_{0}}(t,x)\right]\frac{r(t)}{\sqrt{1-\frac{\dot d(t)^{2}}{4}}}\exp\left(\frac{{-}\sqrt{2}d(t)}{\sqrt{1-\frac{\dot d(t)^{2}}{4}}}\right).
    \end{align}
    In conclusion, we deduce from Lemma \ref{explemma} and the estimate above that
    \begin{align}\nonumber
    \Lambda(\varphi_{2})(t,x)-\Lambda(\varphi_{2,0})(t,x)
    \cong_{6} 
    &\frac{\ddot r(t)}{\sqrt{1-\frac{\dot d(t)^{2}}{4}}}\left( H^{'}_{0,1}\right)^{w_{0}}(t,x)
    {-}\frac{\dot d(t)\dot r(t)}{1-\frac{\dot d(t)^{2}}{4}}\left( H^{''}_{0,1}\right)^{w_{0}}(t,x)
    \\ \nonumber &{+}\frac{8\sqrt{2}r(t)e^{{-}\sqrt{2}d(t)}}{\sqrt{1-\frac{\dot d(t)^{2}}{4}}}\left( H^{''}_{0,1}\right)^{w_{0}}(t,x)\\ \label{finalestimate} &{+}\left[48M^{w_{0}}(t,x)-60N^{w_{0}}(t,x)\right]\frac{r(t)}{\sqrt{1-\frac{\dot d(t)^{2}}{4}}}e^{{-}\sqrt{2}d(t)}.
\end{align}
\textbf{Step 3.}(Conclusion of the proof of Theorem \ref{k=2}.)
 \par Using Lemmas \ref{dlemma}, \ref{geraldt} and estimates \eqref{r,dr,k=1}, \eqref{geraldr}, we conclude from the product rule of derivative and estimate \eqref{finalestimate} that if $0<v\ll 1,$ then
\begin{equation}\label{firstconcl}
    \norm{\frac{\partial^{l}}{\partial t^{l}}\left[\Lambda(\varphi_{2})(t,x)-\Lambda(\varphi_{2,0})(t,x)\right]}_{H^{s}_{x}}\lesssim_{l,s}v^{4+l}\left(\vert t\vert v+\ln{\left(\frac{1}{v^{2}}\right)}\right)e^{-2\sqrt{2}\vert t\vert v}, 
\end{equation}
 for all $t\in\mathbb{R},\,s\geq 0$ and $l\in\mathbb{N}\cup\{0\}.$
\par Moreover, Remark \ref{elint} implies for all $m,l\in\mathbb{N}\cup\{0\}$ and $t\in\mathbb{R}$ that if $h\in S^{+}_{m},$ then
\begin{align*}
    \left\vert \frac{d^{l}}{dt^{l}}\left\langle h\left(w_{0}(t,x)\right), H^{'}_{0,1}(w_{0}(t,{-}x))\right\rangle \right\vert\lesssim_{h,l} v^{2+l}\left[\vert t\vert v+\ln{\left(\frac{1}{v}\right)}\right]^{m+1}e^{-2\sqrt{2}\vert t\vert v}.
\end{align*}
Consequently, using Lemma \ref{phi222}, the ordinary differential equation \eqref{ode1}, identity \eqref{oioi} and estimate \eqref{finalestimate},  if $0<v\ll 1,$ there exists $n_{2}\in\mathbb{N}$ satisfying for all $l\in \mathbb{N}\cup\{0\}$ the following estimate
 \begin{equation*}
 \left\vert \frac{d^{l}}{dt^{l}} \left\langle \Lambda(\varphi_{2})(t,x),\, H^{'}_{0,1}\left(\frac{x-\frac{d(t)}{2}}{\sqrt{1-\frac{\dot d(t)^{2}}{4})}}\right)\right\rangle\right\vert\lesssim_{l} v^{l+6}\left[\ln{\left(\frac{1}{v^{2}}\right)}+\vert t\vert v\right]^{n_{2}+1}e^{-2\sqrt{2}\vert t\vert v}.
 \end{equation*}
Therefore, if $0<v\ll 1,$ Lemmas \ref{porrataylor}, \ref{phi222}, inequality \eqref{firstconcl} and estimates \eqref{r,dr,k=1}, \eqref{geraldr} of $r(t)$ imply for any $l\in\mathbb{N}\cup\{0\}$ and all $t\in\mathbb{R}$ that
\begin{equation*}
 \left\vert \frac{d^{l}}{dt^{l}} \left\langle \Lambda(\varphi_{2})(t,x),\, H^{'}_{0,1}\left(\frac{x+r(t)-\frac{d(t)}{2}}{\sqrt{1-\frac{\dot d(t)^{2}}{4}}}\right)\right\rangle\right\vert\lesssim_{l} v^{l+6}\left[\ln{\left(\frac{1}{v^{2}}\right)}+\vert t\vert v\right]^{n_{2}+1}e^{-2\sqrt{2}\vert t\vert v}.
 \end{equation*}
Since $\varphi_{2}$ is an odd function on $x,$ the estimate above implies \eqref{orthodec2}.
\par Finally, since $d(t)$ and $r(t)$ are even functions and $\lim_{t\to+\infty}r(t)$ exists, there exists a number $e(v)$ such that $\phi_{2}(v,t,x)=\varphi_{2}(t+e(v),x)$ satisfies Theorem \ref{approximated theorem} for $k=2.$ More precisely, because
$ d(t)=2vt+\frac{1}{\sqrt{2}}\ln{\left(\frac{8}{v^{2}}\right)}+O(e^{-2\sqrt{2}vt})
$ when $t\gg 1$
and $\lim_{s\to\pm\infty} r(s) =e_{r}=O(v^{2}\ln{\left(\frac{1}{v}\right)}^{2}),$ we consider
$e(v)=\frac{-1}{2v}\left[\frac{1}{\sqrt{2}}\ln{\left(\frac{8}{v^{2}}\right)}+e_{r}\right].
$
\end{proof}
\section{Approximate solutions for $k>2$}\label{sub32}
\par We will prove the following theorem, which implies Theorem \ref{approximated theorem}: 
\begin{theorem}\label{strongerr}
There exist a sequence of approximate solutions $\left(\varphi_{k,v}(t,x)\right)_{k\geq 2},$ functions $r_{k}(v,t)$ that are smooth and even in $t,$ and numbers $n_{k}\in\mathbb{N}$ such that if $0<v\ll 1,$ then for any $k\in\mathbb{N}_{\geq 2},$ $m\in\mathbb{N}$
\begin{equation}\label{aproxr}
    \vert r_{k}(v,t)\vert\lesssim_{k} v^{2(k-1)} \left[\ln{\left(\frac{1}{v}\right)}\right]^{n_{k}},\,\left\vert\frac{\partial^{m}}{\partial t^{m}}r_{k}(v,t)\right\vert\lesssim_{k,m} v^{2(k-1)+m}\left[\ln{\left(\frac{1}{v}\right)}+\vert t\vert v\right]^{n_{k}}e^{-2\sqrt{2}\vert t\vert v},
\end{equation}
$\varphi_{k,v}(t,x)$ satisfies for $\rho_{k}(v,t)=-\frac{d_{v}(t)}{2}+\sum_{j=2}^{k}r_{j}(v,t)$ the identity
\begin{multline}\label{aproxfor}
    \varphi_{k,v}(t,x)=H_{0,1}\left(\frac{x+\rho_{k}(v,t)}{\sqrt{1-\frac{\dot d_{v}(t)^{2}}{4}}}\right)+H_{-1,0}\left(\frac{x-\rho_{k}(v,t)}{\sqrt{1-\frac{\dot d_{v}(t)^{2}}{4}}}\right)\\+e^{-\sqrt{2}d_{v}(t)}\left[\mathcal{G}\left(\frac{x+\rho_{k}(v,t)}{\sqrt{1-\frac{\dot d_{v}(t)^{2}}{4}}}\right)-\mathcal{G}\left(\frac{-x+\rho_{k}(v,t)}{\sqrt{1-\frac{\dot d_{v}(t)^{2}}{4}}}\right)\right]
    \\{+}\mathcal{T}_{k,v}\left(vt,\frac{x+\rho_{k}(v,t)}{\sqrt{1-\frac{\dot d_{v}(t)^{2}}{4}}}\right)-\mathcal{T}_{k,v}\left(vt,\frac{-x+\rho_{k}(v,t)}{\sqrt{1-\frac{\dot d_{v}(t)^{2}}{4}}}\right),\end{multline}
the following estimates for any $l\in\mathbb{N}\cup\{0\}$ and $s\geq 1$
\begin{equation}\label{geraldecay}
    \norm{\frac{\partial^{l}}{\partial t^{l}}\Lambda(\varphi_{k,v}(t,x)}_{H^{s}_{x}}\lesssim_{k,l,s}v^{2k+l}\left[\ln{\left(\frac{1}{v^{2}}\right)}+\vert t\vert v\right]^{n_{k}}e^{-2\sqrt{2}\vert t\vert v},
\end{equation}
and 
\begin{equation}\label{orthodecay}
 \left\vert\frac{d^{l}}{dt^{l}}\left[\left\langle \Lambda(\varphi_{k,v})(t,x),\, H^{'}_{0,1}\left(\frac{\pm x+\rho_{k}(v,t)}{(1-\frac{\dot d_{v}(t)^{2}}{4})^{\frac{1}{2}}}\right)\right\rangle\right]\right\vert\lesssim_{k,l} v^{2k+l+2}\left[\ln{\left(\frac{1}{v^{2}}\right)}+\vert t\vert v\right]^{n_{k}+1}e^{-2\sqrt{2}\vert t\vert v},    
\end{equation}
where $\mathcal{T}_{k}(t,x)$ is a finite sum of functions $p_{k,i,v}(t)h_{k,i}(x)$ with $h_{k,i,v}\in\mathscr{S}(\mathbb{R})\cap S^{+}_{\infty}$ and each $p_{k,i,v}(t)$ being an even function satisfying
\begin{equation*}
\left \vert\frac{d^{m}p_{k,i,v}(t)}{dt^{m}} \right\vert\lesssim_{k,m} v^{4}\left(\ln{\left(\frac{1}{v^{2}}\right)}+\vert t\vert \right)^{n_{k,i}}e^{-2\sqrt{2}\vert t\vert }
\end{equation*}
for a positive number $n_{k,i}\in\mathbb{N}$ and all $m\in \mathbb{N}\cup\{0\}.$ 
\end{theorem}
\begin{remark}\label{veryimportant}
From the result of the subsection before, we have that $\varphi_{2}(t,x)$ and $r(t)$ satisfy all the properties \eqref{aproxfor}, \eqref{aproxr} and \eqref{orthodecay} for $k=2$ if $v\ll1,$ so Theorem \ref{strongerr} is true for $k=2.$ We are going to prove that if, for any $2\leq k\leq \mathcal{M}$, there exists a smooth function $\varphi_{k,v}(t,x)$ denoted by \eqref{aproxfor} that satisfies the conclusion of Theorem \ref{strongerr}  if $0<v\ll 1,$ then there exists also $\varphi_{\mathcal{M}+1,v}(t,x)$ satisfying \eqref{aproxfor}, \eqref{orthodecay} and Theorem \ref{strongerr} if $v\ll 1.$ Next, after a time translation of order $O\left(\frac{\ln{\left(\frac{1}{v}\right)}}{v}\right),$ this function will satisfy Theorem \ref{approximated theorem}. 
\end{remark}
\begin{remark}\label{n20}
 Furthermore, from Theorem \ref{k=2}, we also have that $r_{2}$ satisfies, if $v>0$ is small enough, the following estimates
 \begin{equation*}
      \norm{r_{2}(v,\cdot)}_{L^{\infty}(\mathbb{R})}\lesssim v^{2}\ln{\left(\frac{1}{v^{2}}\right)},\,\, \left\vert \frac{\partial^{l}}{\partial t^{l}}r_{2}(v,t)\ \right\vert \lesssim_{l} v^{2+l}\left[\ln{\left(\frac{1}{v}\right)}+\vert t \vert v\right]e^{{-}2\sqrt{2}\vert t\vert v},
 \end{equation*}
for all $l\in\mathbb{N}.$
\end{remark}
\subsection{Auxiliary lemmas.}\label{subss}
\par  From now on, we assume that Theorem \ref{strongerr} is true for $2\leq k\leq \mathcal{M}.$ We also consider the following defintion.

\begin{definition}\label{negldef}
    We say that function $\mathcal{F}:(0,1)\times\mathbb{R}^{2}\to \mathbb{R}$ is negligible of order $(n,m)\in\mathbb{N}^{2}$ if there exist a constant $M(n)$ satisfying such that $\mathcal{F}$ satisfies for any $v\in (0,1)$ small enough the following estimate
\begin{equation*}
\norm{\frac{\partial^{l}}{\partial t^{l}}\mathcal{F}(v,t,x)}_{H^{s}_{x}}\lesssim_{l,s} v^{n+l}\left(\vert t\vert v+\ln{\left(\frac{1}{v}\right)}\right)^{m}e^{{-}2\sqrt{2}\vert t\vert v},
\end{equation*}
for all $t\in\mathbb{R},$ any $l\in\mathbb{N}$ and all $s\geq 0.$ Moreover, we also say for any $n\in\mathbb{N}_{>6}$ that any two real functions $f,g:(0,1)\times\mathbb{R}^{2}\to \mathbb{R}^{2}$ satisfy $f\cong_{n}g$ if $f-g$ is a negligible function of order $(n,m)$ for some $m\in\mathbb{N}.$
\end{definition}
The demonstration of Theorem \ref{strongerr} will be done by induction on $k.$ However, before the beginning of this proof, we need to prove three lemmas necessary to demonstrate Theorem \ref{strongerr}. The first lemma is the following:
\begin{lemma}\label{represent1} In notation of Theorem \ref{strongerr}, there exist natural numbers $N_{1},N_{2}$ satisfying, for $0<v\ll 1,$ the following estimate
\begin{equation*}
    \Lambda(\varphi_{\mathcal{M},v})(t,x)\cong_{2\mathcal{M}+4}\sum_{i=1}^{N_{1}}s_{i,v}(\sqrt{2}vt)\left[\mathcal{R}_{i}\left(\frac{x+\rho_{\mathcal{M}}(v,t)}{\sqrt{1-\frac{\dot d_{v}(t)^{2}}{4}}}\right)-\mathcal{R}_{i}\left(\frac{-x+\rho_{\mathcal{M}}(v,t)}{\sqrt{1-\frac{\dot d_{v}(t)^{2}}{4}}}\right)\right]
\end{equation*}
such that for all $1\leq i,j\leq N_{1}$ we have $\left\langle \mathcal{R}_{i},\,\mathcal{R}_{j}\right\rangle=\delta_{i,j},\,\mathcal{R}_{i}\in S^{+}_{\infty}\cap\mathscr{S}(\mathbb{R}),$  $s_{i,v}\in C^{\infty}(\mathbb{R})$ satisfies, for all $l\in\mathbb{N}\cup\{0\},$ $\left\vert\frac{d^{l}}{dt^{l}}s_{i,v}(t)\right\vert\lesssim_{l}v^{2\mathcal{M}}\left[\vert t \vert+\ln{\left(\frac{1}{v^{2}}\right)}\right]^{n_{\mathcal{M}}}e^{-2\sqrt{2}\vert t\vert}.$
\end{lemma}
Our demonstration of Lemma \ref{represent1} will need the following result.
\begin{lemma}\label{application1}
    For any $\zeta>1,$ let $\phi:\mathbb{R}_{\geq 1}\times\mathbb{R}^{2}\to\mathbb{R}$ be a function of the form
    \begin{equation*}
        \phi(\zeta,t,x)=H_{0,1}\left(x-\zeta\right)-H_{0,1}\left({-}x\right)+\sum_{i=1}^{\mathcal{N}} p_{i}(t)\left[I_{i}\left(x-\zeta\right)-I_{i}\left({-}x\right)\right],
    \end{equation*}
where $\mathcal{N}<{+}\infty,$ all the functions $p_{i}(t)$ are smooth with all their non-zero derivatives being in $\mathscr{S}(\mathbb{R}),$ and for all $1\leq i\leq \mathcal{N},\,I_{i}\in \mathscr{S}(\mathbb{R})\cap S^{+,m_{i}}$ for some $m_{i}\in\mathbb{N}\cup\{0\}.$
Let $Z_{\zeta}:\mathbb{R}^{2}\to\mathbb{R}$ be the following function
\begin{equation*}
    \mathcal{Z}_{\zeta}(t,x)= U^{'}(\phi(\zeta,t,x))- U^{'}\left(H_{0,1}(x-\zeta)\right)- U^{'}\left(H_{{-}1,0}(x)\right),
\end{equation*}
for any $(t,x)\in\mathbb{R}^{2},$ and $\zeta>1.$ For any $k\in\mathbb{N},$ there exist $\mathcal{N}_{1}(k)\in\mathbb{N},$ functions $h_{i}\in S^{+}_{\infty},$ and numbers $n_{i},\,l_{i}\in\mathbb{N}\cup\{0\},\,\alpha_{i,j}\in\mathbb{N}\cup\{0\}$ for all $1\leq i\leq \mathcal{N}_{1}(k)$ and $1\leq j\leq \mathcal{N}$ such that the following function 
\begin{equation*}
    \mathcal{Z}_{k,\zeta}(t,x)=\sum_{i=1}^{\mathcal{N}_{1}(k)}\left[\zeta^{l_{i}}e^{{-}\sqrt{2}n_{i}\zeta}\left( h_{i}(x-\zeta)-h_{i}({-}x)\right)\prod_{j=1}^{\mathcal{N}}p_{j}(t)^{\alpha_{i,j}}\right] \text{, for all $(\zeta,x)\in\mathbb{R}_{\geq 1}\times\mathbb{R},$}
\end{equation*}
satisfies for any $s\geq 0$ and every $(\zeta,t)\in\mathbb{R}_{\geq 1}\times \mathbb{R}$ the estimate
\begin{equation*}
    \norm{\mathcal{Z}_{\zeta}(t,x)-\mathcal{Z}_{k,\zeta}(t,x)}_{H^{s}_{x}}\leq C(\phi,s,k) e^{{-}\sqrt{2}k\zeta},
\end{equation*}
where $C(\phi,s,k)$ is a positive value depending only on $k$ and $s$ and the function $\phi.$
\end{lemma}
\begin{proof}
     Proposition \ref{separation} and Remarks \ref{sepc}, \ref{reflection} can be applied to estimate with higher precision the function
\begin{equation}\label{worstinteraction}
   \mathcal{Z}_{\zeta}(t,x)= U^{'}\left(\phi(\zeta,t,x)\right)- U^{'}\left(H_{0,1}\left(t,x-\zeta\right)\right)- U^{'}\left({-}H_{0,1}\left({-}x\right)\right), 
\end{equation}
since $ U^{'}(\phi)=2\phi-8\phi^{3}+6\phi^{5}.$ 
More precisely, since $ U^{'}$ is an odd polynomial, it is not difficult to verify from the definition of $\phi(\zeta,t,x)$ and the multinomial formula that $\mathcal{Z}_{\zeta}(t,x)$ is a finite sum of functions of the following kind
\begin{align*}
   \mathcal{X}_{\zeta}(t,x)=&\Bigg[H_{0,1}\left(x-\zeta\right)^{\alpha_{0}}\Big({-} H_{0,1}\left({-}x\right)\Big)^{\beta_{0}} \prod_{i,j=1}^{\mathcal{N}}p_{j}(t)^{\alpha_{j}}I_{j}\left(x-\zeta\right)^{\alpha_{j}}p_{i}(t)^{\beta_{i}}\Big({-}I_{i}\left({-}x\right)\Big)^{\beta_{i}}\Bigg]\\
   &{+}\Bigg[
   H_{0,1}\left(x-\zeta\right)^{\beta_{0}} \Big({-} H_{0,1}\left({-}x\right)\Big)^{\alpha_{0}}\prod_{i,j=1}^{\mathcal{N}}p_{i}(t)^{\beta_{i}}I_{i}\left(x-\zeta\right)^{\beta_{i}}p_{j}(t)^{\alpha_{j}}\Big({-}I_{j}\left({-}x\right)\Big)^{\alpha_{j}}\Bigg],
\end{align*}
such that
\begin{itemize}
    \item $\alpha_{i},\,\beta_{i}\in\mathbb{N}\cup\{0\}$ for all $0\leq i\leq \mathcal{N},$
    \item $\sum_{i=0}^{\mathcal{N}}\alpha_{i}+\beta_{i}$ is odd,
    \item either $\sum_{i=1}^{\mathcal{N}}\alpha_{i}+\beta_{i}\neq 0$ or $\min\left(\alpha_{0},\beta_{0}\right)>0.$
\end{itemize}  
Since every $I_{j}\in S^{+}_{\infty}$, we can apply Lemma \ref{multiplicative} and deduce for any natural number $1\leq j\leq \mathcal{N}$ and any $k\in\mathbb{N}$ that $I_{j}({-}x)^{2k}\in S^{-}_{\infty}$ and $I_{j}(x)^{2k-1}\in S^{+}_{\infty}.$ Moreover, Lemma \ref{multiplicative} also implies for all $k\in\mathbb{N}$ that if $(f_{i})_{1\leq i\leq 2k-1}\subset S^{+}_{\infty},$ then $\prod_{i=1}^{2k-1}f_{i}\in S^{+}_{\infty},$ and if $(f_{i})_{1\leq i\leq 2k}\subset S^{-}_{\infty},$ then
$\prod_{i=1}^{2k}f_{i}\in S^{-}_{\infty}.$
Therefore, we deduce that either
\begin{align*}
     H_{0,1}\left(x\right)^{\alpha_{0}}\prod_{j=1}^{\mathcal{N}}I_{j}\left(x\right)^{\alpha_{j}}\in S^{+}_{\infty},\,\, H_{0,1}\left({-}x\right)^{\beta_{0}}\prod_{i=1}^{\mathcal{N}}I_{i}\left({-}x\right)^{\beta_{i}} \in S^{-}_{\infty}\cup\{1\} \text{ or }\\
       H_{0,1}\left({-}x\right)^{\alpha_{0}}\prod_{j=1}^{\mathcal{N}}I_{j}\left({-}x\right)^{\alpha_{j}}\in S^{-}_{\infty}\cup\{1\},\,\, H_{0,1}\left(x\right)^{\beta_{0}}\prod_{i=1}^{\mathcal{N}}I_{i}\left(x\right)^{\beta_{i}} \in S^{+}_{\infty}.
\end{align*}
\par Consequently, we can apply the Separation Lemma and Remark \ref{reflection} in the expression 
\begin{align*}
   H_{0,1}\left(x-\zeta\right)^{\alpha_{0}}\prod_{j=1}^{\mathcal{N}}I_{j}\left(x-\zeta\right)^{\alpha_{j}}\left[\left({-}H_{0,1}\left({-}x\right)\right)^{\beta_{0}}\prod_{i=1}^{\mathcal{N}}\left({-}I_{i}\left({-}x\right)\right)^{\beta_{i}}\right]\\
   {+}\left[\left({-}H_{0,1}\left({-}x\right)\right)^{\alpha_{0}}\prod_{j=1}^{\mathcal{N}}\left({-}I_{j}\left({-}x\right)\right)^{\alpha_{j}}\right]H_{0,1}\left(x-\zeta\right)^{\beta_{0}}\prod_{i=1}^{\mathcal{N}}I_{i}\left(x-\zeta\right)^{\beta_{i}},
\end{align*} and deduce for any $k\in\mathbb{N}$ the existence of $\mathcal{N}_{2}(k)\in\mathbb{N},$ a set of numbers $l_{i,1},\,n_{i,1}\in \mathbb{N}\cup\{0\}$  and a set of functions $h_{i,1}\in S^{+}_{\infty}\cap\mathscr{S}(\mathbb{R}),$ such that the function 
\begin{equation*}
    \mathcal{X}_{k,\zeta}(t,x)=\left[\sum_{i=1}^{\mathcal{N}_{2}(k)}\zeta^{l_{i,1}}e^{{-}\sqrt{2}n_{i,1}\zeta}\Big(h_{i,1}\left(x-\zeta\right){-}h_{i,1}\left({-}x\right)\Big)\right]\prod_{j=1}^{\mathcal{N}}p_{j}(t)^{\alpha_{j}+\beta_{j}}
\end{equation*}
satisfies, if $\zeta$ is large enough, the estimate 
\begin{equation*}
    \norm{\mathcal{X}_{\zeta}(t,x)-\mathcal{X}_{k,\zeta}(t,x)}_{H^{s}_{x}}\lesssim_{s,k}
e^{-\sqrt{2}(k+1)\zeta}\prod_{j=1}^{\mathcal{N}}\left\vert p_{j}(t)\right\vert^{\alpha_{j}+\beta_{j}}.
\end{equation*}
In conclusion, using triangle inequality, we obtain the result of Lemma \ref{application1}. 
\end{proof}
\begin{corollary}\label{coapplication}
 Let the functions $I_{i}\in \mathscr{S}(\mathbb{R}),\, p_{i}\in C^{\infty}(\mathbb{R})$ be as defined in the statement of Lemma \ref{application1}. Let $\gamma:(0,1)\times\mathbb{R}\to\mathbb{R}$ be a function satisfying
\begin{equation*}
    \norm{\frac{\partial^{l}}{\partial t^{l}}\gamma(v,t)}_{L^{\infty}_{t}(\mathbb{R})}\lesssim_{l} v^{l} \text{, for any $l\in\mathbb{N}\cup\{0\},$ if $0<v\ll 1,$}
\end{equation*} and $w:(0,1)\times\mathbb{R}^{2}\to\mathbb{R}$ be the following smooth function \begin{equation*}
    \omega(v,t,x)=\frac{x-\frac{d_{v}(t)}{2}+\gamma(v,t)}{\sqrt{1-\frac{\dot d_{v}(t)^{2}}{4}}}.
\end{equation*} In addition,  let $\phi_{app}:\mathbb{R}^{2}\to\mathbb{R}$ be the following function
 \begin{equation*}
    \phi_{app}(t,x)= H_{0,1}\left(w(v,t,x)\right)-H_{0,1}\left(w(v,t,{-}x)\right)+\sum_{i=1}^{\mathcal{N}} p_{i}(t)\left[I_{i}\left(w(v,t,x)\right)-I_{i}\left(w(v,t,{-}x)\right)\right],
\end{equation*} 
for all $(t,x)\in\mathbb{R}^{2}$ and $\mathcal{Z}(t,x)$ be denoted by
\begin{equation*}
    \mathcal{Z}(t,x)= U^{'}(\phi_{app}(t,x))- U^{'}\left(H_{0,1}\left(w(v,t,x)\right)\right)-U^{'}\left({-}H_{0,1}\left(w(v,t,{-}x)\right)\right),
\end{equation*}
for any $(t,x)\in\mathbb{R}^{2}.$ If $v\ll 1$ and the functions $p_{i}$ also satisfy the following decay estimate \begin{equation*}
    \max_{1\leq i\leq \mathcal{N}} \norm{p_{i}^{(l)}(t)}\lesssim_{l} v^{l} \text{, for every $l\in\mathbb{N},$}
\end{equation*}
then, for any $k\in\mathbb{N}_{\geq 2},$ there exist $\mathcal{N}_{1}(k)\in\mathbb{N},$ functions $h_{i}\in S^{+}_{\infty},$ and numbers $n_{i},\,l_{i}\in\mathbb{N}\cup\{0\},\,\alpha_{i,j}\in\mathbb{N}\cup\{0\}$ for all $1\leq i\leq \mathcal{N}_{1}(k)$ and $1\leq j\leq \mathcal{N}$ such that the following function
\begin{align*}
   \mathcal{Z}_{k}(t,x)=\Bigg[\sum_{i=1}^{\mathcal{N}_{1}(k)}\left[\frac{d_{v}(t)-2\gamma(v,t)}{\sqrt{1-\dot d_{v}(t)^{2}}}\right]^{l_{i}}\exp\left(\frac{{-}2\sqrt{2}n_{i}\left[d_{v}(t)-2\gamma(v,t)\right]}{\sqrt{1-v(t)^{2}}}\right)\prod_{j=1}^{\mathcal{N}}p_{j}(t)^{\alpha_{j,i}}\Big( h_{i}\left(w(v,t,x)\right)\\{-}h_{i}\left(w(v,t,{-}x)\right)\Big)\Bigg] \text{, for any $(t,x)\in\mathbb{R}^{2},$}
\end{align*}
satisfies
\begin{equation*}
    \norm{\frac{\partial^{l}}{\partial t^{l}}\left[\mathcal{Z}_{k}(t,x)-\mathcal{Z}(t,x)\right]}_{H^{s}_{x}}\leq \hat{C}v^{l}e^{-2\sqrt{2}k d_{v}(t)}d_{v}(t)^{M_{2}(k)},
\end{equation*}
for every $l\in\mathbb{N}\cup\{0\}$ and $s\geq 0,$ where $\hat{C}>0$ is a constant depending only on the functions $(p_{i})_{1\leq i\leq \mathcal{N}}$ and the numbers $l,\,s$ and $k.$ 
\end{corollary}
\begin{proof}[Proof of Corollary \ref{coapplication}.] 
First, from Lemma \ref{application1}, if we replace the variables $x$ and $\zeta,$ respectively, with
${-}w(t,{-}x)$ and 
\begin{equation*}
    \frac{d_{v}(t)-2\gamma(v,t)}{\sqrt{1-\frac{\dot d_{v}(t)^{2}}{4}}},
\end{equation*} we deduce for any $k\in\mathbb{N}_{\geq 2}$ the existence of a set of functions $\left(h_{i}\right)_{i\in\mathbb{N}}\subset S^{+}_{\infty},$ a set of numbers $(\alpha_{j,i})_{(j,i)\in\mathbb{N}^{2}}\subset\mathbb{N}\cup\{0\}$ and two sequences of numbers $(l_{i})_{\in\mathbb{N}}\subset \mathbb{N}\cup\{0\},\,(n_{i})_{i\in\mathbb{N}}\subset\mathbb{N}$ such that if $0<v\ll 1,$ the following function  
\begin{align*}
    \mathcal{Z}_{k}(t,x)=\Bigg[\sum_{i=1}^{\mathcal{N}_{1}(k)}\left[\frac{d_{v}(t)-2\gamma(v,t)}{\sqrt{1-\frac{\dot d_{v}(t)^{2}}{4}}}\right]^{l_{i}}\exp\left(\frac{{-}2\sqrt{2}n_{i}\left(d(t)-2\gamma(v,t)\right)}{\sqrt{1-\frac{\dot d_{v}(t)^{2}}{4}}}\right)\prod_{j=1}^{\mathcal{N}}p_{j}(t)^{\alpha_{j,i}}\Big( h_{i}\left(w(t,x)\right)\\{-}h_{i}\left(w(t,{-}x)\right)\Big)\Bigg]
\end{align*}
satisfies, for a constant $M_{2}(k)\in\mathbb{N}$ any $m\in\mathbb{N},$ the following estimate
\begin{equation*}
    \norm{\mathcal{Z}_{k}(t,x)-\mathcal{Z}(t,x)}_{H^{m}_{x}}\lesssim_{m,k}e^{-2\sqrt{2}k y(t)}\left(1+y(t)\right)^{M_{2}(k)}.
\end{equation*}
\par Furthermore, Separation Lemma also implies the existence of $M_{1}(k)\in\mathbb{N},$ for any $k\in\mathbb{N},$ such that 
\begin{multline}\label{identofremark}
  \mathcal{Z}(t,x)-\mathcal{Z}_{k}(t,x)=\\ \sum_{i=1}^{M_{1}(k)} \exp\left(\frac{{-}\sqrt{2}N_{i}\left(d_{v}(t)-2\gamma(v,t)\right)}{\sqrt{1-\frac{d_{v}(t)^{2}}{4}}}\right)\left(\frac{d_{v}(t)-2\gamma(v,t)}{\sqrt{1-\frac{\dot d_{v}(t)^{2}}{4}}}\right)^{n_{i}} \\ \times \prod_{j=1}^{\mathcal{N}}p_{j}(t)^{\beta_{j,i}}h_{i,1}\left(w(t,x)\right)h_{i,2}\left(w(t,{-}x)\right),
\end{multline}
where for any $1\leq i\leq M_{1}(k),\,n_{i}\in\mathbb{N}\cup\{0\}$ and $\,N_{i}$ in $\mathbb{N}_{\geq k},$ the functions $h_{i,1},\,h_{i,2}\in L^{\infty}_{x}(\mathbb{R})$ are smooth and all $\beta_{j,i}\in \mathbb{N}\cup\{0\}.$   
\par In fact, from Proposition \ref{separation}, we could also say for all $1\leq i\leq M_{1}(k)$ that $2k \leq N_{i},\,n_{i}\in\mathbb{N}\cup\{0\},$ either $h_{i,1}$ or $h_{i,2}$ is in $\mathscr{S}(\mathbb{R})$ and either $h_{i,1}(x)\in S^{+}\cup S^{+}_{\infty},\,h_{j,2}(x)\in S^{-}\cup S^{-}_{\infty}$ or $h_{i,1}({-}x)\in S^{+}\cup S^{+}_{\infty},\,h_{i,2}({-}x)\in S^{-}\cup S^{-}_{\infty}.$ Moreover, since $\gamma$ satisfies the condition of Remark \ref{perturbt},
and  
\begin{align*}
\max_{1 \leq j\leq \mathcal{N}}\left\vert \frac{d^{l}}{dt^{l}}p_{j}(t) \right\vert\lesssim_{l} v^{l} \text{, for all $l\in\mathbb{N}$ and $t\in\mathbb{R},$}
\end{align*}
we deduce from Remark \ref{perturbt} and the product rule of derivative that if $v>0$ is small enough, then
\begin{equation}\label{fimcoap}
\norm{\frac{\partial^{l}}{\partial t^{l}}\left[\mathcal{Z}_{k}(t,x)-\mathcal{Z}(t,x)\right]}_{H^{s}_{x}}\lesssim_{s,k,l} v^{k+l}e^{{-}2\sqrt{2}\vert t\vert v} \text{, for any $l\in\mathbb{N}\cup\{0\}$ and $s\geq 0.$} 
\end{equation}
\par Actually, using the product rule of derivative, for every $1\leq i\leq M_{1}(k),$ we have if $v>0$ is small enough that
\begin{equation*}
    \left\vert \frac{d^{l}}{dt^{l}}\left[\prod_{j=1}^{\mathcal{N}}p_{j}(t)^{\beta_{j,i}}\exp\left(\frac{{-}2\sqrt{2}N_{i}(d_{v}(t)-2\gamma(v,t))}{\sqrt{1-\frac{\dot d_{v}(t)^{2}}{4}}}\right)\left(\frac{d_{v}(t)-2\gamma(v,t)}{\sqrt{1-\frac{d_{v}(t)^{2}}{4}}}\right)^{n_{i}} \right] \right\vert\lesssim_{l,k} v^{k+l}e^{{-}2\sqrt{2}\vert t\vert v},
\end{equation*}
for all $l\in\mathbb{N}\cup\{0\}$ and every $t\in\mathbb{R}.$ Therefore, since Remark \ref{perturbt} implies 
\begin{equation*}
\norm{\frac{\partial^{l}}{\partial t^{l}}h_{i,1}(w(t,x))}_{H^{s}_{x}}+\norm{\frac{\partial^{l}}{\partial t^{l}}h_{i,2}(w(t,x))}_{H^{s}_{x}}\lesssim_{l,s} v^{l},
\end{equation*}
for every $1\leq i\leq M_{1}(k),$ we conclude estimate \eqref{fimcoap} from the product rule, triangle inequality and identity \eqref{identofremark}.
\end{proof}

\begin{proof}[Proof of Lemma \ref{represent1}.]
First, we consider $0<v\ll 1$ and recall that $\Lambda(\cdot)= \frac{\partial^{2}}{\partial t^{2}}-\frac{\partial^{2}}{\partial x^{2}}+ U^{'}(\cdot).$ From Lemma \ref{geraldt} and Remark \ref{perturbt}, if $h\in S^{+}_{\infty}$ and $p_{v}(t)$ satisfies for constants $q_{1},\,q_{2}\in\mathbb{N}$ the following estimate
\begin{equation*}
   \left\vert\frac{d^{l}}{dt^{l}}p_{v}(t)\right\vert\lesssim_{l}v^{2q_{1}}\left[\ln{\left(\frac{1}{v}\right)}+\vert t\vert \right]^{q_{2}}e^{-2\sqrt{2}\vert t\vert } \text{, for all $l\in\mathbb{N}\cup\{0\},$}
\end{equation*}
then
\begin{equation*}
    \left[\frac{\partial^{2}}{\partial t^{2}}-\frac{\partial^{2}}{\partial x^{2}}\right]\left(p_{v}(\sqrt{2}vt)h\left(\frac{x+\rho_{M}(v,t)}{\sqrt{1-\frac{\dot d(t)^{2}}{4}}}\right)\right)
\end{equation*}
is a finite sum of functions $p_{i,v}(\sqrt{2}vt)h_{i}\left(\frac{x+\rho_{\mathcal{M}}(v,t)}{\sqrt{1-\frac{\dot d(t)^{2}}{4}}}\right)$ with $h_{i}\in S^{+}_{\infty}$ and $p_{i,v}$ satisfying for some natural numbers $m_{i}>0,\,w_{i}$ the following decay 
\begin{equation}\label{decf}
    \left\vert\frac{d^{l}}{dt^{l}}\left[p_{i}(\sqrt{2}vt)\right]\right\vert\lesssim_{l}v^{2m_{i}+l}\left[\ln{\left(\frac{1}{v}\right)}+\vert t\vert v\right]^{w_{i}}e^{-2\sqrt{2}\vert t\vert v}\text{, for all $l\in\mathbb{N}\cup\{0\}.$}
\end{equation}
\par Next, using Lemma \ref{dlemma}, Remark \ref{perturbt} and identity $ H^{''}_{0,1}(x)= U^{'}(H_{0,1}(x)),$ we can verify similarly to the proof of Lemma \ref{dt2kink} the following estimate
\begin{equation}
   \left[\frac{\partial^{2}}{\partial t^{2}}-\frac{\partial^{2}}{\partial x^{2}}\right]H_{0,1}\left(\frac{x-\rho_{\mathcal{M}}(v,t)}{\sqrt{1-\frac{\dot d(t)^{2}}{4}}}\right)={-} U^{'}\left(H_{0,1}\left(\frac{x-\rho_{\mathcal{M}}(v,t)}{\sqrt{1-\frac{\dot d(t)^{2}}{4}}}\right)\right)+residue_{0}(t,x),
\end{equation}
where $residue_{0}(t,x)$ is a finite sum of functions 
\begin{equation*}
    q_{i,v}(\sqrt{2}vt)h_{i}\left(\frac{x-\rho_{\mathcal{M}}(v,t)}{\sqrt{1-\frac{\dot d(t)^{2}}{4}}}\right),
\end{equation*}
with $h_{i} \in S^{2}_{+}$ and 
\begin{equation*}
    \left\vert\frac{d^{l}q_{i,v}(t)}{dt^{l}} \right\vert\lesssim_{l} v^{2}\left(\vert t\vert +\ln{\left(\frac{1}{v^{2}}\right)}\right)e^{-2\vert t\vert } \text{, for all $l\in\mathbb{N}\cup\{0\}.$}
\end{equation*}
Therefore, to finish the proof of Lemma \ref{represent1} we need only to study the expression
\begin{equation}\label{DU}
    DU(t,x)= U^{'}(\varphi_{\mathcal{M},v}(t,x))- U^{'}\left(H_{0,1}\left(\frac{x-\rho_{\mathcal{M}}(v,t)}{\sqrt{1-\frac{\dot d(t)^{2}}{4}}}\right)\right)- U^{'}\left(H_{-1,0}\left(\frac{x+\rho_{\mathcal{M}}(v,t)}{\sqrt{1-\frac{\dot d(t)^{2}}{4}}}\right)\right).
\end{equation}
\par Furthermore, from Corollary \ref{coapplication}, we can obtain for any natural $N\gg 1$ the existence of natural numbers $N_{1},\,N_{2},$ a set of functions $h_{\mathcal{M},j} \in S^{+}_{\infty}$ and a set of functions $p_{\mathcal{M},j,v}(t)$ satisfying property \eqref{decf} such that $DU(t,x)$ satisfies
\begin{equation}
    DU(t,x) \cong_{2N}\sum_{j=1}^{N_{1}} p_{\mathcal{M},j,v}(\sqrt{2}vt)\left[h_{\mathcal{M},j}\left(\frac{x+\rho_{\mathcal{M}}(v,t)}{\sqrt{1-\frac{\dot d(t)^{2}}{4}}}\right)-h_{\mathcal{M},j}\left(\frac{-x+\rho_{\mathcal{M}}(v,t)}{\sqrt{1-\frac{\dot d(t)^{2}}{4}}}\right)\right].
\end{equation}
Moreover, if two functions $p_{1}(t),\,p_{2}(t)$ satisfy property \eqref{decf}, then, from the product rule of derivative, $p_{1}(t)p_{2}(t)$ have much smaller decay than the right-hand side of \eqref{decf} as $\vert t\vert\to+\infty,$ because of the $e^{-4\sqrt{2}\vert t\vert}$ contribution obtained in the product of these functions.     
\par In conclusion, we proved that there exist a finite subset $I_{0}$ of $\mathbb{N},$ functions $p_{j,v}$ satisfying property \eqref{decf} and $h_{j}\in\mathscr{S}(\mathbb{R})\cap S^{+}_{\infty}$ such that
\begin{equation}\label{repfinal}
    \Lambda(\varphi_{\mathcal{M},v})(t,x)\cong_{2N}\sum_{j\in I_{0}} p_{j,v}(\sqrt{2}vt)\left[h_{j}\left(\frac{x+\rho_{\mathcal{M}}(v,t)}{\sqrt{1-\frac{\dot d(t)^{2}}{4}}}\right)-h_{j}\left(\frac{-x+\rho_{\mathcal{M}}(v,t)}{\sqrt{1-\frac{\dot d(t)^{2}}{4}}}\right)\right].
\end{equation}
 Moreover, after a finite number of applications of Proposition \ref{separation}, it is possible to obtain an estimate of the form \eqref{repfinal} for any $N\gg 1$ if we assume $v\ll 1.$ 
\par From Gram-Schmidt, we can exchange the functions $h_{j}$ in \eqref{repfinal} by functions $\mathcal{R}_{j}\in S^{+}_{\infty}\cap \mathscr{S}(\mathbb{R})$ such that $\langle \mathcal{R}_{j},\,\mathcal{R}_{i}\rangle=\delta_{i,j}$ and
\begin{equation}\label{repfinal2}
    \Lambda(\varphi_{\mathcal{M},v})(t,x)\cong_{2N}\sum_{j\in I} s_{j,v}(\sqrt{2}vt)\left[\mathcal{R}_{j}\left(\frac{x+\rho_{\mathcal{M}}(v,t)}{\sqrt{1-\frac{\dot d(t)^{2}}{4}}}\right)-\mathcal{R}_{j}\left(\frac{-x+\rho_{\mathcal{M}}(v,t)}{\sqrt{1-\frac{\dot d(t)^{2}}{4}}}\right)\right],
\end{equation}
for a finite set $I$ with the functions $s_{j,v}(t)$ also satisfying property \eqref{decf}.  In conclusion, from the assumption that the conclusion of Theorem \ref{strongerr} is true when $k=\mathcal{M},$ we deduce from Lemma \ref{interactt} and condition $\langle \mathcal{R}_{j},\,\mathcal{R}_{i}\rangle=\delta_{i,j}$ that, for any $j\in I,$ we have 
\begin{align} \nonumber
    \left\langle \Lambda(\varphi_{\mathcal{M},v})(t,x),\mathcal{R}_{j}\left(\frac{x+\rho_{\mathcal{M}}(v,t)}{\sqrt{1-\frac{\dot d(t)^{2}}{4}}}\right)\right\rangle=&\left[\left(1-\frac{\dot d(t)^{2}}{4}\right)^{\frac{1}{2}}+O(v)\right]s_{j,v}(\sqrt{2}vt)\\ \nonumber &{+}\sum_{i\neq j,\,i\in I}s_{i,v}(\sqrt{2}vt)O(v)
    \\ \label{linearAlg} &{+}O\left(v^{2N}\left(\vert t\vert v+\ln{\left(\frac{1}{v}\right)}\right)^{N_{2}}e^{{-}2\sqrt{2}\vert t\vert v}\right).
\end{align}
 Since $N> \mathcal{M}+1,$ using the identities\eqref{linearAlg} for all $j\in I$ and estimate \eqref{geraldecay}, we deduce that 
\begin{equation}\label{0s0s}
\vert s_{j,v}(t) \vert\lesssim v^{2\mathcal{M}}\left[\vert t\vert +\ln{\left(\frac{1}{v^{2}}\right)}\right]^{n_{\mathcal{M}}}e^{-2\sqrt{2}\vert t\vert}, 
\end{equation}
for all $j\in I ,$ and $t\in \mathbb{R}.$ 
\par Furthermore, we can assume the existence of $m_{0}\in\mathbb{N}\cup\{0\}$ such that 
\begin{equation}\label{hypopos}
    \left\vert\frac{d^{l}s_{j,v}(t)}{dt^{l}}\right\vert\lesssim_{l} v^{2\mathcal{M}}\left[\vert t\vert +\ln{\left(\frac{1}{v^{2}}\right)}\right]^{n_{\mathcal{M}}}e^{-2\sqrt{2}\vert t\vert},
\end{equation}
 for all $j\in I,\,l\in\mathbb{N}\cup\{0\}$ satisfying $0\leq l\leq m_{0}.$
 But, from estimate \eqref{repfinal2}, assumption \eqref{hypopos}, Lemma \ref{dlemma} and Remark \ref{perturbt}, we deduce using the product rule of derivative that
\begin{align*}
   \sum_{j\in I} 2^{\frac{m_{0}+1}{2}}v^{m_{0}+1}s^{(m_{0}+1)}_{j,v}\left(\sqrt{2}vt\right)\left[\mathcal{R}_{j}\left(\frac{x+\rho_{\mathcal{M}}(v,t)}{\sqrt{1-\frac{\dot d(t)^{2}}{4}}}\right)-\mathcal{R}_{j}\left(\frac{-x+\rho_{\mathcal{M}}(v,t)}{\sqrt{1-\frac{\dot d(t)^{2}}{4}}}\right)\right]\\\cong_{2\mathcal{M}+m_{0}+1}\frac{\partial^{m_{0}+1}\Lambda\left(\phi_{\mathcal{M},v}\right)(t,x)}{\partial t^{m_{0}+1}}.
\end{align*}
Therefore, similarly to the proof of \eqref{0s0s} for all $j\in I$ and using Remark \ref{elint} in the expressions
\begin{equation*}
   \left\langle \mathcal{R}_{j}\left(\frac{x+\rho_{\mathcal{M}}(v,t)}{\sqrt{1-\frac{\dot d(t)^{2}}{4}}}\right),\mathcal{R}_{i}\left(\frac{{-}x+\rho_{\mathcal{M}}(v,t)}{\sqrt{1-\frac{\dot d(t)^{2}}{4}}}\right)\right\rangle \text{ for all $i,\,j\in I,$}
\end{equation*}
we obtain the following estimate
\begin{equation*}
    \left\vert\frac{d^{m_{0}+1}s_{j,v}(t)}{dt^{l}}\right\vert\lesssim_{m_{0}+1} v^{2\mathcal{M}}\left[\vert t\vert +\ln{\left(\frac{1}{v^{2}}\right)}\right]^{n_{\mathcal{M}}}e^{-2\sqrt{2}\vert t\vert}.
\end{equation*}
In conclusion, from induction on $l$, the estimate of the decay of the derivatives of $s_{j,v}$ in Lemma \ref{represent1} is true for all $l\in\mathbb{N}\cup\{0\}.$   
\end{proof}
The third lemma necessary to the proof of the existence of $\phi_{\mathcal{M}+1,v}(t,x)$ is the following:
\begin{lemma}\label{projectionl}
In notation of Lemma \ref{represent1}, there is a positive number $n_{\mathcal{M}}$ such that the following function
\begin{equation*}
   Proj(t)=\left(1-\frac{\dot d(t)^{2}}{4}\right)^{\frac{1}{2}} \left[\sum_{i=1}^{N_{1}}s_{i,v}(\sqrt{2}vt)\left\langle \mathcal{R}_{i}(x),H^{'}_{0,1}(x)\right\rangle\right]
\end{equation*}
satisfies
\begin{equation*}
    \left\vert\frac{d^{l}}{dt^{l}}Proj(t)\right\vert\lesssim_{l}v^{2M+l+2}\left[\ln{\left(\frac{1}{v^{2}}\right)}+\vert t\vert v\right]^{n_{\mathcal{M}}}e^{-2\sqrt{2}\vert t\vert v} \text{ for all $l\in\mathbb{N}\cup\{0\}.$}
\end{equation*}
\end{lemma}
\begin{proof}
From Lemma \ref{represent1}, there exists a function $res:(0,1)\times\mathbb{R}^{2}\to\mathbb{R}$ such that $res\cong_{2\mathcal{M}+4} 0$ and
\begin{equation*}
    \Lambda(\varphi_{\mathcal{M},v})(t,x)=\sum_{j\in I} s_{j,v}(\sqrt{2}vt)\left[\mathcal{R}_{j}\left(\frac{x+\rho_{\mathcal{M}}(v,t)}{\sqrt{1-\frac{\dot d(t)^{2}}{4}}}\right)-\mathcal{R}_{j}\left(\frac{-x+\rho_{\mathcal{M}}(v,t)}{\sqrt{1-\frac{\dot d(t)^{2}}{4}}}\right)\right]+res(v,t,x).
\end{equation*}
Therefore, we have the following identity
\begin{multline}
    \left\langle\Lambda(\varphi_{\mathcal{M},v})(t,x),\,H^{'}_{0,1}\left(\frac{x+\rho_{\mathcal{M}}(v,t)}{\sqrt{1-\frac{\dot d(t)^{2}}{4}}}\right)\right\rangle \\
    \begin{aligned}
    =& Proj(t)+\left\langle H^{'}_{0,1}\left(\frac{x+\rho_{\mathcal{M}}(v,t)}{\sqrt{1-\frac{\dot d(t)^{2}}{4}}}\right),res(v,t,x)\right\rangle\\&{-}\sum_{j}s_{j,v}(\sqrt{2}tv)\left(1-\frac{\dot d(t)^{2}}{4}\right)^{\frac{1}{2}}\left\langle H^{'}_{0,1}(x),\,\mathcal{R}_{j}\left(-x+2\frac{\rho_{\mathcal{M}}(v,t)}{\sqrt{1-\frac{\dot d(t)^{2}}{4}}}\right)\right\rangle.
\end{aligned}
\end{multline}
First, we recall the function $d(t)=\frac{1}{\sqrt{2}}\ln{\left(\frac{8}{v^{2}}\cosh{(\sqrt{2}vt)}^{2}\right)},$ which satisfies
\begin{equation*}
    \norm{\dot d(t)}_{L^{\infty}(\mathbb{R})}\lesssim v,\,\norm{d^{(k)}(t)}_{L^{\infty}(\mathbb{R})}\lesssim v^{k}e^{-2\sqrt{2}\vert t\vert v}\text{ if $k\geq 2.$}
\end{equation*}
We also recall $\rho_{\mathcal{M}}(v,t)=\sum_{j=2}^{\mathcal{M}}r_{j,v}(t)-\frac{d(t)}{2}.$ Since we are assuming the veracity of estimates \eqref{aproxr} for any natural number $k$ satisfying $2\leq k\leq \mathcal{M},$ we deduce, from Remark \ref{perturbt}, Lemma \ref{represent1}, the product rule of derivative and Cauchy-Schwarz inequality, the existence of $N_{2}\geq 0$ satisfying for any $l\in\mathbb{N}\cup\{0\}$ the following inequalities
\begin{align*}
   \left \vert \frac{d^{l}}{dt^{l}}\left[\left\langle  H^{'}_{0,1}\left(\frac{x+\rho_{\mathcal{M}}(v,t)}{\sqrt{1-\frac{\dot d(t)^{2}}{4}}}\right),res(v,t,x)\right\rangle\right]\right\vert\lesssim_{l}& \sum_{j=0}^{l}\norm{\frac{\partial^{j}}{\partial t^{j}}res(v,t,x)}_{L^{2}_{x}}v^{l-j}\\\lesssim_{l} &v^{2\mathcal{M}+4+l}\left[\ln{\left(\frac{1}{v^{2}}\right)}+\vert t\vert v\right]^{N_{2}}e^{-2\sqrt{2}\vert t\vert v}.
\end{align*}
Furthermore, Lemma \ref{dlemma} and Remark \ref{elint} imply for any $n\in\mathbb{N}\cup\{0\}$ that if $\mathcal{R}_{j}\in S^{+}_{n}$ and $0<v\ll 1,$ then  
\begin{align*}
  \left\vert \frac{d^{l}}{d t^{l}}\left[\left(1-\frac{\dot d(t)^{2}}{4}\right)^{\frac{1}{2}}\left\langle H^{'}_{0,1}(x),\,\mathcal{R}_{j}\left(-x+2\frac{\rho_{\mathcal{M}}(v,t)}{\sqrt{1-\frac{\dot d(t)^{2}}{4}}}\right)\right\rangle\right]\right\vert \\ \lesssim_{l} v^{2+l} \left[\vert t\vert v+\ln{\left(\frac{1}{v^{2}}\right)}\right]^{n+1}e^{-2\sqrt{2}v\vert t\vert} \text{, for all $l\in\mathbb{N}\cup\{0\}.$}
\end{align*}
Consequently, from Lemma \ref{represent1} and the product rule of derivative, we deduce the existence of a sufficiently large number $n_{\mathcal{M}}$ satisfying for $v\ll 1$ the following inequality
\begin{multline}
    \left\vert\frac{d^{l}}{dt^{l}}\left[\sum_{j}s_{j,v}(\sqrt{2}tv)\left(1-\frac{\dot d(t)^{2}}{4}\right)^{\frac{1}{2}}\left\langle H^{'}_{0,1}(x),\,\mathcal{R}_{j}\left(-x+2\frac{\rho_{\mathcal{M}}(v,t)}{\sqrt{1-\frac{\dot d(t)^{2}}{4}}}\right)\right\rangle\right]\right\vert\\ \lesssim_{l,\mathcal{M}} v^{2\mathcal{M}+2+l}\left[\vert t\vert v+\ln{\left(\frac{1}{v^{2}}\right)}\right]^{n_{\mathcal{M}}}e^{-2\sqrt{2}\vert t\vert v}, 
\end{multline}
for all $l\in\mathbb{N}\cup\{0\}.$ 
\par In conclusion, we obtain Lemma \ref{projectionl} from the estimates above, Lemma \ref{represent1} and triangle inequality.
\end{proof}
From now on, for $\rho_{\mathcal{M}}(v,t)=-\frac{d(t)}{2}+\sum_{j=2}^{\mathcal{M}}r_{j}(v,t),$ we consider
\begin{equation}\label{wmwm}
w_{\mathcal{M}}(t,x)=\frac{x+\rho_{\mathcal{M}}(v,t)}{\sqrt{1-\frac{\dot d(t)^{2}}{4}}}.
\end{equation}
To simplify our notation, we denote the function $r_{l,v}$ as $r_{l}$ for every $l\in\mathbb{N}_{\geq 2}.$ Using the notation of Lemma \ref{represent1} and Lemma \ref{projectionl}, we define the function
\begin{equation}\label{Bequation}
    \Gamma(t,x)=\sum_{i=1}^{N_{1}}s_{i,v}(\sqrt{2}vt)\mathcal{R}_{i}(x)- H^{'}_{0,1}(x)\frac{Proj(t)}{\norm{ H^{'}_{0,1}}_{L^{2}}^{2}\sqrt{1-\frac{\dot d(t)^{2}}{4}}}.
\end{equation}
Lemmas \ref{represent1} and \ref{projectionl} imply 
$
    \left\langle \Gamma\left(t,x\right),\,H^{'}_{0,1}\left(w_{\mathcal{M}}\left(x\right)\right)\right\rangle=0
$ 
for all $t\in\mathbb{R},$ and for any $(t,x)\in\mathbb{R}^{2}$ 
\begin{multline}\label{firstfinal}
\Lambda(\varphi_{M,v})(t,x)\cong_{2\mathcal{M}+4}\Bigg[ H^{'}_{0,1}\left(w_{\mathcal{M}}(t,x)\right)\frac{Proj(t)}{\norm{ H^{'}_{0,1}}_{L^{2}}^{2}\sqrt{1-\frac{\dot d(t)^{2}}{4}}}\\{-} H^{'}_{0,1}\left(w_{\mathcal{M}}(t,{-}x)\right)\frac{Proj(t)}{\norm{ H^{'}_{0,1}}_{L^{2}}^{2}\sqrt{1-\frac{\dot d(t)^{2}}{4}}}\Bigg]
    \\{+}\Gamma\left(t,w_{\mathcal{M}}(t,x)\right)-\Gamma\left(t,w_{\mathcal{M}}(t,{-}x)\right).
\end{multline}

 Moreover, from Lemma \ref{firstinvert} and Lemma \ref{secondinvert}, we can define the function $L_{1}(\Gamma(t,\cdot))(x)\in\mathscr{S}(\mathbb{R})\cap S^{+}_{\infty},$ more precisely, from the linearity of $L_{1},$ we have for any $(t,x)\in\mathbb{R}^{2}$ the following identity
\begin{equation}\label{explicityB}
    L_{1}(\Gamma(t,\cdot))(x)=\sum_{i=1}^{N_{1}}s_{i,v}(\sqrt{2}vt)L_{1}\left(\mathcal{R}_{i}-\frac{ H^{'}_{0,1}}{\norm{H^{'}_{0,1}}_{L^{2}}^{2}}\langle H^{'}_{0,1},\,\mathcal{R}_{i}\rangle\right)(x),
\end{equation}
and so, from Lemma \ref{represent1}, we have for any $t\in\mathbb{R},\,s>0$ and $l\in\mathbb{N}\cup\{0\}$ that
\begin{equation}\label{L1B}
    \norm{\frac{\partial^{l}}{\partial t^{l}}L_{1}(\Gamma(t,\cdot))(x)}_{H^{s}_{x}}\lesssim_{s,l} v^{2\mathcal{M}+l}\left(\vert t\vert v+\ln{\left(\frac{8}{v^{2}}\right)}\right)^{n_{\mathcal{M}}}e^{-2\sqrt{2}\vert t\vert v}.
\end{equation}
\par Next, we recall from the inductive hypothesis of Theorem \ref{strongerr} that $\varphi_{\mathcal{M},v}(t,x)$ also has the representation  \eqref{aproxfor} given by 
\begin{multline}\label{phiMM}
    \varphi_{\mathcal{M},v}(t,x)=H_{0,1}\left(w_{\mathcal{M}}(t,x)\right)-H_{0,1}\left(w_{\mathcal{M}}(t,{-}x)\right)+e^{-\sqrt{2}d(t)}\left[\mathcal{G}\left(w_{\mathcal{M}}(t,x)\right)-\mathcal{G}\left(w_{\mathcal{M}}(t,{-}x)\right)\right]\\
    +\mathcal{T}_{\mathcal{M}}\left(vt,w_{\mathcal{M}}(t,x)\right)-\mathcal{T}_{\mathcal{M}}\left(vt,w_{\mathcal{M}}(t,{-}x)\right),
\end{multline}
 where $\mathcal{T}_{\mathcal{M}}(t,x)$ is a function even on $t$ satisfying for a sufficiently large number $n_{\mathcal{M},1}\in\mathbb{N}$ and any $s>0$ the following inequality
\begin{equation}\label{RM}
    \norm{\frac{\partial^{l}}{\partial t^{l}}\mathcal{T}_{\mathcal{M}}(t,x)}_{H^{s}_{x}}\lesssim_{l,s} v^{4}\left(\vert t\vert +\ln{\left(\frac{1}{v^{2}}\right)}\right)^{n_{\mathcal{M},1}}e^{-2\sqrt{2}\vert t\vert} \text{ for all $l\in\mathbb{N}\cup\{0\},$ if $0<v\ll 1.$}
\end{equation}
\subsection{Construction of $r_{\mathcal{M}+1}(v,t).$}
From now on, for $j\in\{1,2,3,4\},$ we consider the smooth functions $I_{j}:\mathbb{R}\to\mathbb{R}$ defined by 
\begin{align}\label{I1}
    I_{1}(t)=&e^{{-}\sqrt{2}d(t)}\left\langle U^{(3)}(H_{0,1}(x))e^{{-}\sqrt{2}x}L_{1}\left(\Gamma(t,\cdot)\right)(x),H^{'}_{0,1}(x) \right\rangle,\\ \label{I2}
    I_{2}(t)=&\left\langle L_{1}\left(\Gamma(t,\cdot)\right)\left({-}x+\frac{d(t)}{\sqrt{1-\frac{\dot d(t)^{2}}{4}}}\right),\left[2- U^{''}(H_{0,1}\left(x\right))\right]H^{'}_{0,1}\left(x\right)\right\rangle\\
    \label{I3}
I_{3}(t)=&{-}e^{{-}\sqrt{2}d(t)}\left\langle U^{(3)}\left(H_{0,1}(x)\right)\mathcal{G}(x)L_{1}\left(\Gamma(t,\cdot)\right)(x), H^{'}_{0,1}(x)\right\rangle,\\ \label{I4}
    I_{4}(t)=&{-}\left\langle  \left[\frac{\partial^{2}}{\partial t^{2}}-\frac{\dot d(t)^{2}}{4-\dot d(t)^{2}}\frac{\partial^{2}}{\partial x^{2}}\right]L_{1}\left(\Gamma(t,\cdot)\right)(x),H^{'}_{0,1}(x)\right\rangle.
\end{align}
\par Denoting the function $NL_{\mathcal{M}}:\mathbb{R}\to\mathbb{R}$ by
\begin{equation*}
    NL_{\mathcal{M}}(t)=\sum_{i=1}^{4}I_{i}(t) \text{, for any $t\in\mathbb{R},$}
\end{equation*}
and recalling the function $Proj:\mathbb{R}
\to\mathbb{R}$ defined in Lemma \ref{projectionl}, we consider
\begin{equation*}
    Res_{\mathcal{M}}(t)=NL_{\mathcal{M}}(t)-Proj(t),
\end{equation*}
for any $t\in\mathbb{R},$ and the following
ordinary differential equation
\begin{multline}\label{odek}
    \begin{cases}
   \begin{aligned}\norm{ H^{'}_{0,1}}_{L^{2}_{x}}^{2} \ddot r_{\mathcal{M}+1}(t)=&{-}32\norm{ H^{'}_{0,1}}_{L^{2}_{x}}^{2} e^{-\sqrt{2}d(t)}r_{\mathcal{M}+1}(t)+Res_{\mathcal{M}}(t),
   \end{aligned}\\
   r_{\mathcal{M}+1}(t)=r_{\mathcal{M}+1}({-}t).
    \end{cases}
\end{multline}
  \par From Lemma \ref{represent1}, we recall the existence of $n_{\mathcal{M}}>0$ such that, for any $l\in\mathbb{N}\cup\{0\}$ and $1\leq i\leq N_{1},$ $\left\vert\frac{d^{l}}{dt^{l}}s_{i,v}(t)\right\vert\lesssim_{l}v^{2\mathcal{M}}\left[\vert t \vert+\ln{\left(\frac{1}{v^{2}}\right)}\right]^{n_{\mathcal{M}}}e^{-2\sqrt{2}\vert t\vert},$ if $0<v\ll 1.$ Therefore, for $0<v\ll 1$ and using Remark \ref{elint} and 
identities \eqref{explicityB}, \eqref{I2}, we deduce the existence of $n_{\mathcal{M},2}\in\mathbb{N}\cup\{0\}$ satisfying  
\begin{equation*}
\left\vert I^{(l)}_{2}(t)\right\vert\lesssim_{l}v^{2\mathcal{M}+2+l}\left(\vert t\vert v+\ln{\left(\frac{1}{v}\right)}\right)^{n_{\mathcal{M},2}}e^{{-}2\sqrt{2}\vert t\vert v} \text{, for every $t\in\mathbb{R}$ and any $l\in\mathbb{N}\cup\{0\}$.}
\end{equation*}
\par Next, from estimate \eqref{L1B}, Lemma \ref{dlemma}, identity \eqref{I1} and Cauchy-Schwarz
inequality, we obtain using the product rule of derivative that
\begin{equation}\label{I1decay}
   \left\vert I^{(l)}_{1}(t)\right\vert\lesssim_{l}v^{2\mathcal{M}+2+l}\left(\vert t\vert v+\ln{\left(\frac{1}{v}\right)}\right)^{n_{\mathcal{M}}}e^{{-}2\sqrt{2}\vert t\vert v} \text{, for every $t\in\mathbb{R}$ and any $l\in\mathbb{N}\cup\{0\}.$}
\end{equation}
Similarly to the proof of estimate \eqref{I1decay}, we deduce that
\begin{equation*}
    \left\vert I^{(l)}_{3}(t)\right\vert\lesssim_{l}v^{2\mathcal{M}+2+l}\left(\vert t\vert v+\ln{\left(\frac{1}{v}\right)}\right)^{n_{\mathcal{M}}}e^{{-}2\sqrt{2}\vert t\vert v} \text{, for every $t\in\mathbb{R}$ and any $l\in\mathbb{N}\cup\{0\}.$}
\end{equation*}
\par Furthermore, using Lemma \ref{dlemma}, estimate \eqref{L1B} and the product rule of derivative, we obtain the following decay estimate
\begin{equation*}
    \left\vert I^{(l)}_{4}(t)\right\vert\lesssim_{l}v^{2\mathcal{M}+2+l}\left(\vert t\vert v+\ln{\left(\frac{1}{v}\right)}\right)^{n_{\mathcal{M}}}e^{{-}2\sqrt{2}\vert t\vert v} \text{, for every $t\in\mathbb{R}$ and any $l\in\mathbb{N}\cup\{0\}.$}
\end{equation*}
\par In conclusion, using Lemma \ref{projectionl}, we obtain that the function $Res_{\mathcal{M}}(t)$ defined in the ordinary differential equation \eqref{odek} satisfies for some number $n_{\mathcal{M}+1}\geq 0$ the following decay estimate
\begin{equation}\label{restM}
    \left\vert\frac{d^{l}}{dt^{l}}Res_{\mathcal{M}}(t)\right\vert\lesssim_{l} v^{l+2\mathcal{M}+2}\left(\vert t\vert v+\ln{\left(\frac{1}{v^{2}}\right)}\right)^{n_{\mathcal{M}+1}}e^{-2\sqrt{2}v\vert t\vert},
\end{equation}
for every $t\in\mathbb{R}$ and any $l\in\mathbb{N}\cup\{0\}.$
\par Repeating the argument in the first step of the proof of Theorem \ref{k=2}, we have for the following functions 
\begin{align}\label{theta11}
    \theta_{\mathcal{M}+1,2}(t)=&\frac{1}{\sqrt{2}v}\int_{-\infty}^{t}Res_{\mathcal{M}}(s)\tanh{(\sqrt{2}vs)}\,ds,\\  \label{theta12} \theta_{\mathcal{M}+1,1}(t)=&\frac{-1}{\sqrt{2}v}\int_{0}^{t}Res_{\mathcal{M}}(s)\left[\sqrt{2}vs\tanh{(\sqrt{2}vs)}-1\right]\,ds,
\end{align}
that $r_{\mathcal{M}+1}(t)=\theta_{\mathcal{M}+1,1}(t)\tanh{(\sqrt{2}vt)}+\theta_{\mathcal{M}+1,2}(t)\left[\sqrt{2}vt\tanh{(\sqrt{2}vt)}-1\right]$ is even and satisfies the ordinary differential equation \eqref{odek}. Moreover, from the decay estimates of $Res_{\mathcal{M}}(t)$ in \eqref{restM}, we can deduce by induction on $l\in\mathbb{N}$ the existence of a number $n_{\mathcal{M}+1}\geq 0$ satisfying
\begin{equation}\label{rmd}
   \left\vert \frac{d^{l}}{dt^{l}}r_{\mathcal{M}+1}(t)\right\vert\lesssim_{l} v^{2\mathcal{M}+l}\left(\vert t\vert v+\ln{\left(\frac{1}{v^{2}}\right)}\right)^{n_{\mathcal{M}+1}}e^{-2\sqrt{2}\vert t\vert v} \text{, for every $t\in\mathbb{R}$ and any $l\in\mathbb{N},$}
\end{equation} so $\lim_{t\to{+}\infty}r_{\mathcal{M}+1}(t)$ exists and $\norm{r_{\mathcal{M}+1}(t)}_{L^{\infty}_{t}(\mathbb{R})}\lesssim v^{2\mathcal{M}}\ln{\left(\frac{1}{v}\right)}^{n_{\mathcal{M}+1}}.$ 
 \par Next, we are going to denote, for all $(t,x)\in\mathbb{R}^{2}$ and $0<v\ll1,$ $\phi_{\mathcal{M}+1,v,0}:\mathbb{R}^{2}\to\mathbb{R}$ by
\begin{multline}\label{phim10}
    \begin{aligned}
    \phi_{\mathcal{M}+1,v,0}(t,x)=&H_{0,1}\left(w_{\mathcal{M}}(t,x+r_{\mathcal{M}+1}(t))\right)-H_{0,1}\left(w_{\mathcal{M}}(t,{-}x+r_{\mathcal{M}+1}(t))\right)\\&{+}e^{-\sqrt{2}d(t)}\left[\mathcal{G}\left(w_{\mathcal{M}}(t,x+r_{\mathcal{M}+1}(t))\right)-\mathcal{G}\left(w_{\mathcal{M}}(t,{-}x+r_{\mathcal{M}+1}(t))\right)\right]
    \\&{+}\mathcal{T}_{\mathcal{M}}\left(vt,w_{\mathcal{M}}(t,x+r_{\mathcal{M}+1}(t))\right)-\mathcal{T}_{\mathcal{M}}\left(vt,w_{\mathcal{M}}(t,{-}x+r_{\mathcal{M}+1}(t))\right),
    \end{aligned}
\end{multline}
and use this function to construct $\varphi_{\mathcal{M}+1,v}:\mathbb{R}^{2}\to\mathbb{R}$ satisfying Theorem \ref{strongerr} for $k=\mathcal{M}+1,$ which will imply the statement of this theorem for all $k\in\mathbb{N}_{\geq 2}$ by induction. 
\par Since we assume Theorem \ref{strongerr} is true for $k=\mathcal{M},$ we deduce from Lemma \ref{porrataylor} and estimates \eqref{rmd} of $r_{\mathcal{M}+1}$ that the following function
\begin{multline}\label{phim111}
   \begin{aligned}
   \phi_{\mathcal{M}+1,v,1}(t,x)=& H_{0,1}\left(w_{\mathcal{M}}(t,x+r_{\mathcal{M}+1}(t))\right)-H_{0,1}\left(w_{\mathcal{M}}(t,{-}x+r_{\mathcal{M}+1}(t))\right)\\&{+}e^{-\sqrt{2}d(t)}\left[\mathcal{G}\left(w_{\mathcal{M}}(t,x+r_{\mathcal{M}+1}(t))\right)-\mathcal{G}\left(w_{\mathcal{M}}(t,{-}x+r_{\mathcal{M}+1}(t))\right)\right]
\\&{+}\mathcal{T}_{\mathcal{M}}\left(vt,w_{\mathcal{M}}(t,x)\right)-\mathcal{T}_{\mathcal{M}}\left(vt,w_{\mathcal{M}}(t,{-}x)\right) \text{, for every $(t,x)\in\mathbb{R}^{2},$}
\end{aligned}
\end{multline}
satisfies
\begin{equation}\label{lambdamv10}\Lambda(\phi_{\mathcal{M}+1,v,1})(t,x)\cong_{2\mathcal{M}+4} \Lambda\left(\phi_{\mathcal{M}+1,v,0}(t,x)\right).
\end{equation}
 \begin{lemma}\label{hm+1}
For any function $h \in L^{\infty}(\mathbb{R})$ such that $ h^{'}\in \mathscr{S}(\mathbb{R}),$ we have
\begin{multline}\label{dt2hM1}
   \begin{aligned}
    \frac{\partial^{2}}{\partial t^{2}}\left[h\left(w_{\mathcal{M}}\left(t,x+r_{\mathcal{M}+1}(t)\right)\right)\right]\cong_{2\mathcal{M}+4} &\frac{\partial^{2}}{\partial t_{1}^{2}}\Big\vert_{t_{1}=t}\left[h\left(w_{\mathcal{M}}\left(t_{1},x+r_{\mathcal{M}+1}(t)\right)\right)\right]\\&{+}\frac{\ddot r_{\mathcal{M}+1}(t)}{\sqrt{1-\frac{\dot d(t)^{2}}{4}}}h^{'}\left(w_{\mathcal{M}}\left(t,x+r_{\mathcal{M}+1}(t)\right)\right)\\&{-}\frac{\dot r_{\mathcal{M}+1}(t)\dot d(t)}{1-\frac{\dot d(t)^{2}}{4}}h^{''}\left(w_{\mathcal{M}}\left(t,x+r_{\mathcal{M}+1}(t)\right)\right).
\end{aligned}
\end{multline}
\end{lemma}
\begin{proof}[Proof of Lemma \ref{hm+1}.]
First, using \eqref{wmwm} and the product rule of derivative, we can verify the following identity
\begin{align*}
\frac{\partial^{2}}{\partial t^{2}}\left[h\left(w_{\mathcal{M}}\left(t,x+r_{\mathcal{M}+1}(t)\right)\right)\right]=&\frac{\partial^{2}}{\partial t_{1}^{2}}\Big\vert_{t_{1}=t}\left[h\left(w_{\mathcal{M}}\left(t_{1},x+r_{\mathcal{M}+1}(t)\right)\right)\right]\\&{+}2\frac{\dot r_{\mathcal{M}+1}(t)}{\sqrt{1-\frac{\dot d(t)^{2}}{4}}}\frac{\partial}{\partial t_{1}}\Big\vert_{t_{1}=t}\left[ h^{'}\left(w_{\mathcal{M}}(t_{1},x+r_{\mathcal{M}+1}(t))\right)\right]\\&{+}2\dot r_{\mathcal{M}+1}(t)\left[\frac{d}{dt}\left(1-\frac{\dot d(t)^{2}}{4}\right)^{{-}\frac{1}{2}}\right]h^{'}\left(w_{\mathcal{M}}(t,x+r_{\mathcal{M}+1}(t))\right)\\&{+}\frac{\dot r_{\mathcal{M}+1}(t)^{2}}{1-\frac{\dot d(t)^{2}}{4}} h^{''}\left(w_{\mathcal{M}}(t,x+r_{\mathcal{M}+1})\right)\\&{+}\frac{\ddot r_{\mathcal{M}+1}(t)}{\sqrt{1-\frac{\dot d(t)^{2}}{4}}} h^{'}\left(w_{\mathcal{M}}(t,x+r_{\mathcal{M}+1})\right). 
\end{align*}
We recall that the function $w_{\mathcal{M}}$ satisfies, for all $(t,x)\in\mathbb{R}^{2},$ the equation
\begin{equation*}w_{\mathcal{M}}(t,x)=w_{0}\left(t,x-\frac{d(t)}{2}+\sum_{j=2}^{\mathcal{M}}r_{j}(t)\right), 
\end{equation*}
and the estimates in \eqref{aproxr} are true for any $2\leq k\leq \mathcal{M}$ from the inductive hypotheses of Theorem \ref{approximated theorem}.
\par Using estimate \eqref{rmd} and the product rule of derivative, we deduce that
\begin{equation*}
    \left\vert \frac{d^{l}}{dt^{l}}\left[\dot r_{\mathcal{M}+1}(t)^{2}\right] \right\vert\lesssim_{l} v^{4\mathcal{M}+2+l}\left(\vert t\vert +\ln{\left(\frac{1}{v}\right)}\right)^{2n_{\mathcal{M}+1}}e^{{-}4\sqrt{2}\vert t\vert v} \text{, for every $t\in\mathbb{R}$ and any $l\in\mathbb{N}\cup\{0\}.$}
\end{equation*}
Therefore, the estimate above, Lemma \ref{dlemma}, Remark \ref{perturbt} and the product rule of derivative imply that
\begin{equation*}
    \frac{\dot r_{\mathcal{M}+1}(t)^{2}}{2-\frac{\dot d(t)^{2}}{4}} h^{''}\left(w_{\mathcal{M}}\left(t,x+r_{\mathcal{M}+1}\right)\right)\cong_{2\mathcal{M}+4} 0.
\end{equation*}

\par Moreover, from estimates \eqref{rmd}, we deduce using Lemma \ref{dlemma}, the chain and product rule of derivative that if $0<v\ll 1,$ then
\begin{equation*}
   \left\vert \frac{d^{l}}{dt^{l}}\left[\dot r_{\mathcal{M}+1}(t)\left[\frac{d}{dt}\left(1-\frac{\dot d(t)^{2}}{4}\right)^{{-}\frac{1}{2}}\right]\right]\right\vert\lesssim_{l}v^{2\mathcal{M}+4+l}\left(\vert t\vert v+\ln{\left(\frac{1}{v}\right)} \right)^{n_{\mathcal{M}+1}}e^{{-}2\sqrt{2}\vert t\vert v}, 
\end{equation*}
for every $t\in\mathbb{R}$ and any $l\in\mathbb{N}\cup\{0\}.$
So, using Remark \ref{perturbt} and the product rule of derivative, we obtain that
\begin{equation*}
    \dot r_{\mathcal{M}+1}(t)\left[\frac{d}{dt}\left(1-\frac{\dot d(t)^{2}}{4}\right)^{{-}\frac{1}{2}}\right]h^{'}\left(w_{\mathcal{M}}(t,x+r_{\mathcal{M}+1}(t))\right)\cong_{2\mathcal{M}+4} 0.
\end{equation*}
  
\par Next, from estimates \eqref{rmd} and $\norm{r_{\mathcal{M}+1}(t)}_{L^{\infty}}\lesssim v^{2\mathcal{M}}\ln{\left(\frac{1}{v}\right)}^{n_{\mathcal{M}+1}},$ we deduce using Lemma \ref{porrataylor} for all $s\geq 1$ and $l\in\mathbb{N}\cup\{0\}$ that
\begin{align*}
    \norm{\frac{\partial^{l}}{\partial t^{l}}\left[w_{\mathcal{M}}(t,x+r_{\mathcal{M}+1}(t)) h^{''}\left(w_{\mathcal{M}}(t,x+r_{\mathcal{M}+1}(t))\right)-w_{\mathcal{M}}(t,x)h^{''}\left(w_{\mathcal{M}}(t,x\right)\right]}_{H^{s}_{x}}\lesssim_{s,l}\\ v^{2\mathcal{M}+l}\left[\ln{\frac{1}{v}}\right]^{n_{\mathcal{M}+1}},\\
    \norm{\frac{\partial^{l}}{\partial t^{l}}\left[ h^{''}\left(w_{\mathcal{M}}(t,x+r_{\mathcal{M}+1}(t))\right)- h^{''}\left(w_{\mathcal{M}}(t,x\right)\right]}_{H^{s}_{x}}\lesssim_{s,l} v^{2\mathcal{M}+l}\left[\ln{\frac{1}{v}}\right]^{n_{\mathcal{M}+1}}.
\end{align*}
Therefore, since we are assuming the veracity of estimates  \eqref{aproxr} for any $2\leq j\leq \mathcal{M},$ using the identity
\begin{multline*}
    \frac{\partial}{\partial t_{1}}\Big\vert_{t_{1}=t} h^{'}\left(w_{\mathcal{M}}(t,x+r_{\mathcal{M}+1}(t))\right)\\=
   \left[\frac{d}{dt}\left(1-\frac{\dot d(t)^{2}}{4}\right)^{\frac{{-}1}{2}}\right]\left(1-\frac{\dot d(t)^{2}}{4}\right)^{\frac{1}{2}}w_{\mathcal{M}}(t,x+r_{\mathcal{M}+1}(t))h^{''}\left(w_{\mathcal{M}}(t,x+r_{\mathcal{M}+1}(t))\right)\\{+}\frac{\dot r_{\mathcal{M}+1}(t)+\sum_{j=2}^{\mathcal{M}}\dot r_{j}(t)-\frac{\dot d(t)}{2}}{\sqrt{1-\frac{\dot d(t)^{2}}{4}}}h^{''}\left(w_{\mathcal{M}}(t,x+r_{\mathcal{M}+1}(t))\right),
\end{multline*}
 estimate
\eqref{rmd}, Lemma \ref{dlemma} and the product rule of derivative, we deduce that
\begin{equation*}
    \frac{2\dot r_{\mathcal{M}+1}(t)}{\sqrt{1-\frac{\dot d(t)^{2}}{4}}}\frac{\partial}{\partial t_{1}}\Big\vert_{t_{1}=t}\left[h^{'
    }\left(w_{\mathcal{M}}(t_{1},x+r_{\mathcal{M}+1}(t))\right)\right]\cong_{2\mathcal{M}+4}  \frac{{-}\dot r_{\mathcal{M}+1}(t)\dot d(t)}{1-\frac{\dot d(t)^{2}}{4}}h^{''}\left(w_{\mathcal{M}}(t_{1},x)\right).
\end{equation*}
In conclusion, estimate \eqref{dt2hM1} is true. 
\end{proof}
\subsection{Proof of Theorem \ref{strongerr}.}
\begin{proof}[Proof of Theorem \ref{strongerr}]
From the observations made at the beginning of this section, we need only to construct $\varphi_{\mathcal{M}+1,v}$ satisfying Theorem \ref{strongerr} from the function $\varphi_{\mathcal{M},v}$ denoted in \eqref{phiMM}. Let $\varphi_{\mathcal{M}+1,v}:\mathbb{R}^{2}\to\mathbb{R}$ be the function satisfying the following identity
\begin{align*}
\varphi_{\mathcal{M}+1,v}(t,x)=&\phi_{\mathcal{M}+1,v,0}(t,x)-L_{1}(\Gamma(t,\cdot))\left(w_{\mathcal{M}}(t,x+r_{\mathcal{M}+1}(t))\right)\\&{+}L_{1}(\Gamma(t,\cdot))\left(w_{\mathcal{M}}(t,{-}x+r_{\mathcal{M}+1}(t)\right),
\end{align*}
for every $(t,x)\in\mathbb{R}^{2},$ where $\phi_{\mathcal{M}+1,v,0}(t,x)$ is defined in \eqref{phim10}.
\par From the definition of $\Lambda,$ we have that
\begin{multline}\label{pep1}
\begin{aligned}
    \Lambda\left(\varphi_{\mathcal{M}+1,v}\right)(t,x)=&\left[\frac{\partial^{2}}{\partial t^{2}}-\frac{\partial^{2}}{\partial x^{2}}\right]\phi_{\mathcal{M}+1,v,0}(t,x)+U^{'}\left(\varphi_{\mathcal{M}+1,v}(t,x)\right)\\
&{+}\left[\frac{\partial^{2}}{\partial t^{2}}-\frac{\partial^{2}}{\partial x^{2}}\right]\left[{-}L_{1}(\Gamma(t,\cdot))\left(w_{\mathcal{M}}(t,x+r_{\mathcal{M}+1}(t))\right)\right]\\
&{+}\left[\frac{\partial^{2}}{\partial t^{2}}-\frac{\partial^{2}}{\partial x^{2}}\right]L_{1}(\Gamma(t,\cdot))\left(w_{\mathcal{M}}(t,{-}x+r_{\mathcal{M}+1}(t)\right)
\end{aligned}
\end{multline}
Moreover, since
$
    \left[{-}\frac{\partial^{2}}{\partial x^{2}}+ U^{''}\left(H_{0,1}(x)\right)\right]L_{1}\left(\Gamma(t,\cdot)\right)(x)=\Gamma(t,x),
$
and $w_{\mathcal{M}}(t,x)=\frac{x-\rho_{\mathcal{M}}(v,t)}{\sqrt{1-\frac{\dot d(t)^{2}}{4}}},$ we have the following identity
\begin{multline}\label{pep2}
    \left[{-}\frac{4-\dot d(t)^{2}}{4}\frac{\partial^{2}}{\partial x^{2}}+ U^{''}\left(H_{0,1}\left(w_{\mathcal{M}}(t,x+r_{\mathcal{M}+1}(t))\right)\right)\right]L_{1}\left(\Gamma(t,\cdot)\right)(w_{\mathcal{M}}(t,x+r_{\mathcal{M}+1}(t)))\\ =\Gamma(t,w_{\mathcal{M}}(t,x+r_{\mathcal{M}+1}(t))).
\end{multline}
Moreover, from identity \eqref{pep1},we deduce that $\varphi_{\mathcal{M}+1,v}(t,x)$ satisfies
\begin{multline}\label{Lambdaphi0,v}
\Lambda(\varphi_{\mathcal{M}+1,v})(t,x)-\Lambda(\phi_{\mathcal{M}+1,v,0})(t,x)\\
\begin{aligned}
=&\mathcal{L}_{\mathcal{M}+1,0}(t,x)+\mathcal{L}_{\mathcal{M}+1,1}(t,x)-\mathcal{L}_{\mathcal{M}+1,1}(t,{-}x)
+\mathcal{L}_{\mathcal{M}+1,2}(t,x)
-\mathcal{L}_{\mathcal{M}+1,2}(t,{-}x),
\end{aligned}
\end{multline}
for all $(t,x)\in\mathbb{R}^{2},$
where, for $0\leq j\leq 2,$ the functions $\mathcal{L}_{\mathcal{M}+1,j}:\mathbb{R}^{2}\to\mathbb{R}$ satisfy for any $(t,x)\in\mathbb{R}^{2}$ the following identities:
\begin{align}
 \begin{split}  \label{lm0}
   \mathcal{L}_{M+1,0}(t,x)= &  U^{'}\left(\varphi_{\mathcal{M}+1,v}(t,x)\right)-U^{'}\left(\phi_{\mathcal{M}+1,v,0}(t,x)\right)\\&{-}U^{''}\left(\phi_{\mathcal{M}+1,v,0}(t,x)\right)\Big[L_{1}(\Gamma(t,\cdot))\left(w_{\mathcal{M}}(t,{-}x+r_{\mathcal{M}+1}(t))\right)\\&{-}L_{1}(\Gamma(t,\cdot))\left(w_{\mathcal{M}}(t,x+r_{\mathcal{M}+1}(t)\right)\Big],
   \end{split}\\
\begin{split}\label{lm1}
    \mathcal{L}_{\mathcal{M}+1,1}(t,x)=&{-} \left[\frac{\partial^{2}}{\partial t^{2}}-\frac{\partial^{2}}{\partial x^{2}}+ U^{''}\left(H_{0,1}\left(w_{\mathcal{M}}(t,x+r_{\mathcal{M}+1}(t))\right)\right)\right]\\ &\times L_{1}\left(\Gamma(t,\cdot)\right)\left(w_{\mathcal{M}}(t,x+r_{\mathcal{M}+1})\right),
    \end{split}
    \\
  \begin{split}\label{lm2}
    \mathcal{L}_{M+1,2}(t,x)=&{-}\Big[ U^{''}\left(\phi_{\mathcal{M}+1,v,0}(t,x)\right)\\&{-} U^{''}\left(H_{0,1}\left(w_{\mathcal{M}}(t,x+r_{\mathcal{M}+1}(t))\right)\right)\Big]L_{1}\left(\Gamma(t,\cdot)\right)\left(w_{\mathcal{M}}(t,x+r_{\mathcal{M}+1})\right).
  \end{split}
\end{align}
\par Next, for $3\leq j\leq 6,$ we denote the functions $\mathcal{L}_{\mathcal{M}+1,j}:\mathbb{R}\to\mathbb{R}$  by
\begin{align}
\begin{split}
\mathcal{L}_{\mathcal{M}+1,3}(t,x)=& U^{'}\left(H_{0,1}\left(w_{\mathcal{M}}(t,x+r_{\mathcal{M}+1})\right)-H_{0,1}\left(w_{\mathcal{M}}(t,{-}x+r_{\mathcal{M}+1})\right)\right)\\&{-} U^{'}\left(H_{0,1}\left(w_{\mathcal{M}}(t,x+r_{\mathcal{M}+1}(t))\right)\right)\\&{-}U^{'}\left({-}H_{0,1}\left(w_{\mathcal{M}}(t,{-}x+r_{\mathcal{M}+1}(t))\right)\right),
\end{split}\\
\begin{split}\label{lm4}
\mathcal{L}_{\mathcal{M}+1,4}(t,x)=&\left[\frac{\partial^{2}}{\partial t^{2}}-\frac{\partial^{2}}{\partial x^{2}}\right]\left[e^{{-}\sqrt{2}d(t)}\mathcal{G}\left(w_{\mathcal{M}}(t,x+r_{\mathcal{M}+1}(t))\right)-\mathcal{G}\left(w_{\mathcal{M}}(t,{-}x+r_{\mathcal{M}+1}(t))\right)\right]\\
&{+} U^{''}\left(H_{0,1}\left(w_{\mathcal{M}}(t,x+r_{\mathcal{M}+1}(t))\right)\right)\\&{-}U^{''}\left(H_{0,1}\left(w_{\mathcal{M}}(t,{}-x+r_{\mathcal{M}+1}(t))\right)\right)e^{{-}\sqrt{2}d(t)}\mathcal{G}\left(w_{\mathcal{M}}(t,{-}x+r_{\mathcal{M}+1}(t))\right),
    \end{split}
     \\
\begin{split}\label{lm5}
\mathcal{L}_{M+1,5}(t,x)={}&\Big[U^{''}\left(H_{0,1}\left(w_{\mathcal{M}}(t,x+r_{\mathcal{M}+1}(t))-H_{0,1}\left(t,{-}x+r_{\mathcal{M}+1}(t)\right)\right)\right)\\&{-} U^{''}\left(H_{0,1}\left(w_{\mathcal{M}}(t,x+r_{\mathcal{M}+1}(t))\right)\right)\Big]e^{{-}\sqrt{2}d(t)}\mathcal{G}\left(w_{\mathcal{M}}(t,x+r_{\mathcal{M}+1}(t))\right)\\
&{-}\Big[U^{''}\left(H_{0,1}\left(w_{\mathcal{M}}(t,x+r_{\mathcal{M}+1}(t))-H_{0,1}\left(t,{-}x+r_{\mathcal{M}+1}(t)\right)\right)\right)\\&{-}U^{''}\left({-}H_{0,1}\left(w_{\mathcal{M}}(t,{-}x+r_{\mathcal{M}+1}(t))\right)\right)\Big]e^{{-}\sqrt{2}d(t)}\mathcal{G}\left(w_{\mathcal{M}}(t,{-}x+r_{\mathcal{M}+1}(t))\right),
\end{split}\\
\begin{split}\label{lm6}
\mathcal{L}_{\mathcal{M}+1,6}(t,x)=&  U^{'}\left(\phi_{\mathcal{M}+1,v,0}(t,x)\right)\\&{-} U^{'}\left(H_{0,1}\left(w_{\mathcal{M}}(t,x+r_{\mathcal{M}+1})\right)-H_{0,1}\left(w_{\mathcal{M}}(t,{-}x+r_{\mathcal{M}+1})\right)\right)\\
   &{-} U^{''}\left(H_{0,1}\left(w_{\mathcal{M}}(t,x+r_{\mathcal{M}+1}(t))\right)-H_{0,1}\left(w_{\mathcal{M}}(t,{-}x+r_{\mathcal{M}+1}(t))\right)\right)\times\\& \Big[\mathcal{G}\left(w_{\mathcal{M}}(t,x+r_{\mathcal{M}+1}(t))\right)-\mathcal{G}\left(w_{\mathcal{M}}(t,{-}x+r_{\mathcal{M}+1}(t))\right)\Big]e^{{-}\sqrt{2}d(t)},
   \end{split}
\end{align}
and they satisfy the following equation
\begin{multline*}
    \sum_{j=3}^{6}\mathcal{L}_{\mathcal{M}+1,j}(t,x)\\= U^{'}\left(\phi_{\mathcal{M}+1,v,0}(t,x)\right)- U^{'}\left(H_{0,1}\left(w_{\mathcal{M}}(t,x+r_{\mathcal{M}+1})\right)\right)
   - U^{'}\left({-}H_{0,1}\left(w_{\mathcal{M}}(t,{-}x+r_{\mathcal{M}+1})\right)\right)\\{+}\left[\frac{\partial^{2}}{\partial t^{2}}-\frac{\partial^{2}}{\partial x^{2}}\right]\left[e^{{-}\sqrt{2}d(t)}\left(\mathcal{G}\left(w_{\mathcal{M}}(t,x+r_{\mathcal{M}+1}(t))\right)-\mathcal{G}\left(w_{\mathcal{M}}(t,{-}x+r_{\mathcal{M}+1}(t))\right)\right)\right].
\end{multline*}
We recall the function $\phi_{\mathcal{M}+1,v,1}$ defined in \eqref{phim111} and obtain from \eqref{phim10}, the identity above and estimate \eqref{lambdamv10} that
\begin{align}\label{LambdahM1}
\Lambda(\phi_{\mathcal{M}+1,v,1})(t,x)\cong_{2\mathcal{M}+4}& \Lambda(\phi_{\mathcal{M}+1,v,0})(t,x)\\ \nonumber \cong_{2\mathcal{M}+4} & \Lambda\left(H_{0,1}\left(w_{\mathcal{M}}(t,x+r_{\mathcal{M}+1})\right)\right)+\Lambda\left({-}H_{0,1}\left(w_{\mathcal{M}+1}(t,{-}x+r_{\mathcal{M}+1})\right)\right)\\ \nonumber &{+}\left[\frac{\partial^{2}}{\partial t^{2}}-\frac{\partial^{2}}{\partial x^{2}}\right]\left(\mathcal{T}_{\mathcal{M}}(vt,x+r_{\mathcal{M}+1})-\mathcal{T}_{\mathcal{M}}(vt,{-}x+r_{\mathcal{M}+1})\right)
    \\ \nonumber &{+}\sum_{j=3}^{6}\mathcal{L}_{\mathcal{M}+1,j}(t,x).
\end{align}
\par Next, using Lemma \ref{hm+1}, we obtain the following estimate
\begin{align}
    \Lambda\left(H_{0,1}\left(w_{\mathcal{M}}(t,x+r_{\mathcal{M}+1}(t))\right)\right)\cong_{2\mathcal{M}+4} \nonumber  &\left[\frac{\partial^{2}}{\partial t_{1}^{2}}\Big\vert_{t_{1}=t}-\frac{\partial^{2}}{\partial x^{2}}\right]H_{0,1}\left(w_{\mathcal{M}}(t_{1},x+r_{\mathcal{M}+1}(t))\right) \\ \nonumber&{+} U^{'}\left(H_{0,1}\left(w_{\mathcal{M}}(t,x+r_{\mathcal{M}+1}(t))\right)\right)\\ \nonumber &{+}\frac{\ddot r_{\mathcal{M}+1}(t)}{\sqrt{1-\frac{\dot d(t)^{2}}{4}}} H^{'}_{0,1}\left(w_{\mathcal{M}}\left(t,x+r_{\mathcal{M}+1}(t)\right)\right)\\ \label{LambdahM00}&{-}\frac{\dot r_{\mathcal{M}+1}(t)\dot d(t)}{1-\frac{\dot d(t)^{2}}{4}} H^{''}_{0,1}\left(w_{\mathcal{M}}\left(t,x+r_{\mathcal{M}+1}(t)\right)\right).
\end{align}

 Consequently, using  estimates \eqref{rmd} and Lemma \ref{porrataylor} in the right-hand side of \eqref{LambdahM00}, we obtain the following estimate 
 \begin{multline}\label{LambdahM001}
     \begin{aligned}
     \Lambda\left(H_{0,1}\left(w_{\mathcal{M}}(t,x+r_{\mathcal{M}+1}(t))\right)\right)\cong_{2\mathcal{M}+4} & \Lambda\left(H_{0,1}\left(w_{\mathcal{M}}(t,x)\right)\right) \\&{+}\frac{\ddot r_{\mathcal{M}+1}(t)}{\sqrt{1-\frac{\dot d(t)^{2}}{4}}}H^{'}_{0,1}\left(w_{\mathcal{M}}\left(t,x\right)\right)\\&{-}\frac{\dot r_{\mathcal{M}+1}(t)\dot d(t)}{1-\frac{\dot d(t)^{2}}{4}}H^{''}_{0,1}\left(w_{\mathcal{M}}\left(t,x\right)\right)
     \\&{+}r_{\mathcal{M}+1}(t)\frac{\partial}{\partial x}\Lambda\left(H_{0,1}\left(w_{\mathcal{M}}(t,x)\right)\right).
 \end{aligned}
 \end{multline}
\par Actually, since $ H^{''}_{0,1}(x)=U^{'}(H_{0,1}),$ we have
\begin{equation*}
    {-}\frac{\partial^{2}}{\partial  x^{2}}H_{0,1}\left(w_{\mathcal{M}}(t,x)\right)+ U^{'}\left(H_{0,1}\left(w_{\mathcal{M}}(t,x)\right)\right)=\frac{{-}\dot d(t)^{2}}{4-\dot d(t)^{2}} H^{''}_{0,1}\left(w_{\mathcal{M}}(t,x)\right),
\end{equation*}
which implies the following equation
\begin{equation*}
    \Lambda\left(H_{0,1}\left(w_{\mathcal{M}}(t,x)\right)\right)=\frac{{-}\dot d(t)^{2}}{4-\dot d(t)^{2}} H^{''}_{0,1}\left(w_{\mathcal{M}}(t,x)\right)+\frac{\partial^{2}}{\partial t^{2}}H_{0,1}\left(w_{\mathcal{M}}(t,x)\right).
\end{equation*}
Consequently, since we are assuming that the estimates in \eqref{aproxr} are true every $k\in\mathbb{N}$ satisfying $2\leq k\leq \mathcal{M},$ we deduce from Lemma \ref{porrataylor} and estimate \eqref{rmd} that
\begin{align*}
r_{\mathcal{M}+1}(t)\Lambda\left(H_{0,1}\left(w_{\mathcal{M}}(t,x)\right)\right)=&\frac{{-}r_{\mathcal{M}+1}(t)\dot d(t)^{2}}{4-\dot d(t)^{2}}H^{''}_{0,1}\left(w_{\mathcal{M}}(t,x)\right)+r_{\mathcal{M}+1}(t)\frac{\partial^{2}}{\partial t^{2}}H_{0,1}\left(w_{\mathcal{M}}(t,x)\right)\\ \cong &_{2\mathcal{M}+4}  \,\, r_{\mathcal{M}+1}(t)\Lambda\left(H_{0,1}\left(w_{0}(t,x)\right)\right).
\end{align*}
Therefore, from Lemma \ref{dt2kink} and the above estimate above, we deduce that
\begin{align*}
    r_{\mathcal{M}+1}(t)\frac{\partial}{\partial x}\Lambda\left(H_{0,1}\left(w_{\mathcal{M}}(t,x)\right)\right)
\cong_{2\mathcal{M}+4} &\,\,{-} \frac{r_{\mathcal{M}+1}(t)8\sqrt{2}e^{{-}\sqrt{2}d(t)}}{1-\frac{\dot d(t)^{2}}{4}} H^{''}_{0,1}\left(w_{\mathcal{M}}(t,x)\right)
\end{align*}
 due to Lemma \ref{porrataylor} and the assumption that estimates \eqref{aproxr} are true for $2\leq k\leq \mathcal{M}.$  
In conclusion, we have the following estimate
\begin{multline}
    \Lambda\left(H_{0,1}\left(w_{\mathcal{M}}(t,x+r_{\mathcal{M}+1}(t))\right)\right) \cong_{2\mathcal{M}+4}\, \Lambda\left(H_{0,1}\left(w_{\mathcal{M}}(t,x)\right)\right)+\frac{\ddot r_{\mathcal{M}+1}(t)}{\sqrt{1-\frac{\dot d(t)^{2}}{4}}} H^{'}_{0,1}\left(w_{\mathcal{M}}(t,x)\right)\\
    {-}\frac{\dot r_{\mathcal{M}+1}(t)\dot d(t)+r_{\mathcal{M}+1}(t)8\sqrt{2}e^{{-}\sqrt{2}d(t)}}{1-\frac{\dot d(t)^{2}}{4}} H^{''}_{0,1}\left(w_{\mathcal{M}}(t,x)\right).
\end{multline}
\par From now, we are going to divide the remaining part of the proof on different steps.\\
\textbf{Step.1}(Estimate of $\mathcal{L}_{\mathcal{M}+1,0}(t,x).$)
First, we recall the inequality $\norm{fg}_{H^{s}_{s}}\lesssim_{s}\norm{f}_{H^{s}_{s}}\norm{g}_{H^{s}_{x}}$ for all $s\geq 1.$ So, using Remark \ref{perturbt}, Lemma \ref{interactionsize}, estimate \eqref{L1B} and the facts that $U\in C^{\infty}(\mathbb{R})$ and $\phi_{\mathcal{M}+1,v,0}\in L^{\infty}(\mathbb{R}^{2})\cap C^{\infty}(\mathbb{R}^{2}),$ we obtain for any natural number $j\geq 3$ that the function
\begin{multline*}
   \mathcal{E}_{j,\mathcal{M}}(t,x)=U^{(j)}\left(\phi_{\mathcal{M}+1,v,0}(t,x)\right)\Big[L_{1}(\Gamma(t,\cdot))\left(w_{\mathcal{M}}(t,x+r_{\mathcal{M}+1}(t))\right)\\-L_{1}(\Gamma(t,\cdot))\left(w_{\mathcal{M}}(t,{-}x+r_{\mathcal{M}+1}(t)\right)\Big]^{j-1}
\end{multline*}
satisfies, for all $s\geq 1,$ $\norm{ \mathcal{E}_{j,\mathcal{M}}(t,x)}_{H^{s}_{x}}\lesssim_{s}\norm{L_{1}\left(\Gamma(t,\cdot)\right)(x)}_{H^{x}_{s}}^{(j-1)}\lesssim v^{2\mathcal{M}+4}\left(\ln{\left(\frac{1}{v}\right)}\right)^{2n_{\mathcal{M}}}e^{-2\sqrt{2}\vert t\vert v(j-1)}$ if $0<v\ll 1.$ Indeed, using Remark \ref{perturbt}, estimate \eqref{L1B} and the product rule of derivative, we obtain similarly for all natural number $j\geq 3$ that
\begin{equation*}
    \norm{\frac{\partial^{l}}{\partial t^{l}}\mathcal{E}_{j,\mathcal{M}}(t,x)}_{H^{s}_{x}}\lesssim_{s,l} v^{2\mathcal{M}+4+l}\left(\ln{\left(\frac{1}{v}\right)}\right)^{2n_{\mathcal{M}}}e^{-2\sqrt{2}\vert t\vert v} \text{ for all $l\in\mathbb{N}\cup\{0\},$ if $0<v\ll 1.$}
\end{equation*}
Therefore, from \eqref{lm0}, we have $\mathcal{L}_{\mathcal{M}+1,0}\cong_{2\mathcal{M}+4} 0.$\\
\textbf{Step 2.}(Estimate of $\mathcal{L}_{\mathcal{M}+1,3}.$)
In notation of Lemma \ref{l1}, from the definition of $w_{\mathcal{M}}$ in \eqref{wmwm} and Remark \ref{reint}, we have
\begin{multline*}
U^{'}\left(H_{0,1}\left(w_{\mathcal{M}}(t,x+r_{\mathcal{M}+1})\right)-H_{0,1}\left(w_{\mathcal{M}}(t,{-}x+r_{\mathcal{M}+1})\right)\right)\\{-} U^{'}\left(H_{0,1}\left(w_{\mathcal{M}}(t,x+r_{\mathcal{M}+1})\right)\right)- U^{'}\left({-}H_{0,1}\left(w_{\mathcal{M}}(t,{-}x+r_{\mathcal{M}+1})\right)\right) \\
\begin{aligned}
   =&24\exp\left(\frac{2\sqrt{2}\left(\rho_{\mathcal{M}}+r_{\mathcal{M}+1}\right)}{\sqrt{1-\frac{\dot d(t)^{2}}{4}}}\right)\left[M\left(w_{\mathcal{M}}(t,x+r_{\mathcal{M}+1})\right)-M\left(w_{\mathcal{M}}(t,{-}x+r_{\mathcal{M}+1})\right)\right]\\&{-}30\exp\left(\frac{2\sqrt{2}\left(\rho_{\mathcal{M}}+r_{\mathcal{M}+1}\right)}{\sqrt{1-\frac{\dot d(t)^{2}}{4}}}\right)\left[N\left(w_{\mathcal{M}}(t,x+r_{\mathcal{M}+1})\right)-N\left(w_{\mathcal{M}}(t,{-}x+r_{\mathcal{M}+1})\right)\right]\\&{+}24\exp\left(\frac{4\sqrt{2}\left(\rho_{\mathcal{M}}+r_{\mathcal{M}+1}\right)}{\sqrt{1-\frac{\dot d(t)^{2}}{4}}}\right)\left[V\left(w_{0}(t,x+r_{\mathcal{M}+1})\right)-V\left(w_{\mathcal{M}}(t,{-}x+r_{\mathcal{M}+1})\right)\right]\\&{+}\frac{60}{\sqrt{2}}\exp\left(\frac{4\sqrt{2}\left(\rho_{\mathcal{M}}+r_{\mathcal{M}+1}\right)}{\sqrt{1-\frac{\dot d(t)^{2}}{4}}}\right)\left[ H^{'}_{0,1}\left(w_{\mathcal{M}}(t,x+r_{\mathcal{M}+1})\right)-H^{'}_{0,1}\left(w_{\mathcal{M}}(t,x+r_{\mathcal{M}+1})\right)\right]\\&{+}R\left(w_{\mathcal{M}}(t,x+r_{\mathcal{M}+1}),\frac{{-}4\rho_{\mathcal{M}}-4r_{\mathcal{M}+1}}{\sqrt{4-\dot d(t)^{2}}}\right).
\end{aligned}
\end{multline*}
\par Moreover, Lemma \ref{l1} implies that $R\left(w_{\mathcal{M}}(t,x+r_{\mathcal{M}+1}),\frac{{-}4\rho_{\mathcal{M}}(v,t)-4r_{\mathcal{M}+1}(t)}{\sqrt{4-\dot d(t)^{2}}}\right)$ is a finite sum of functions
\begin{multline*}
    \exp\left(\frac{{-}4\left(2+d_{i}\right)\sqrt{2}\left(\rho_{\mathcal{M}}(v,t)-r_{\mathcal{M}+1}(t)\right)}{\sqrt{1-\frac{\dot d(t)^{2}}{4}}}\right)m_{i}\left(w_{\mathcal{M}}(t,x+r_{\mathcal{M}+1}(t))\right)\\ \times n_{i}\left({-}w_{\mathcal{M}}(t,{-}x+r_{\mathcal{M}+1}(t))\right),
\end{multline*}
where any $d_{i}\in\mathbb{N},$ every $m_{i}\in S^{+}$ and every $n_{i}\in S^{-}.$ Consequently, using the decay estimates \ref{rmd} of $r_{\mathcal{M}+1}$ and estimate \eqref{aproxr} for any $2\leq k\leq \mathcal{M},$ Lemmas \ref{porrataylor} and \ref{explemma} imply that
\begin{equation*}
    R\left(w_{\mathcal{M}}(t,x+r_{\mathcal{M}+1}),\frac{{-}4\rho_{\mathcal{M}}(v,t)-4r_{\mathcal{M}+1}(t)}{\sqrt{4-\dot d(t)^{2}}}\right)\cong_{2\mathcal{M}+4} R\left(w_{\mathcal{M}}(t,x),\frac{{-}4\rho_{\mathcal{M}}(t)}{\sqrt{4-\dot d(t)^{2}}}\right).
\end{equation*}
\par Furthermore, since we are assuming the veracity of Theorem \ref{strongerr} for any $k\leq \mathcal{M}$ belonging to $\mathbb{N}_{\geq 2},$ we deduce from the Fundamental Theorem of Calculus, Lemma \ref{dlemma}, estimates \eqref{aproxr} for $2\leq k\leq \mathcal{M}$ and estimate \eqref{rmd} that if $v\ll 1,$ then  
\begin{gather*}
   \left\vert \frac{d^{l}}{dt^{l}}\left[ e^{{-}\sqrt{2}(2\rho_{\mathcal{M}}(v,t)-2r_{\mathcal{M}+1}(t))}-e^{{-}2\sqrt{2}\rho_{\mathcal{M}}(v,t)}-2\sqrt{2}r_{\mathcal{M}+1}(t)e^{{-}2\sqrt{2}\rho_{\mathcal{M}}(v,t)}\right]\right\vert\\ \lesssim_{l} v^{4\mathcal{M}+2+l}e^{-2\sqrt{2} v\vert t\vert},\\
   \left\vert \frac{d^{l}}{dt^{l}}\left[e^{{-}2\sqrt{2}\rho_{\mathcal{M}}(v,t)}-e^{{-}\sqrt{2}d(t)}\right] \right\vert\lesssim_{l}  v^{4}\left(\ln{\frac{1}{v}}\right)^{n_{2}}e^{{-}2\sqrt{2}v\vert t\vert },
\end{gather*}
for any $l\in\mathbb{N}\cup\{0\}.$
Therefore, using estimates $\norm{r_{\mathcal{M}}(t)}_{L^{\infty}}\lesssim v^{2\mathcal{M}}\ln{\left(\frac{1}{v}\right)}^{n_{\mathcal{M}}},\,\eqref{rmd}$ and \eqref{aproxr} for $2\leq k\leq \mathcal{M},$  we deduce from Lemmas \ref{explemma}, \ref{porrataylor} the following estimate
\begin{align}\nonumber
\mathcal{L}_{\mathcal{M}+1,3}(t,x) \cong_{2\mathcal{M}+4} &
U^{'}\left(H^{w_{\mathcal{M}}}_{0,1}(t,x)\right)- U^{'}\left(H_{0,1}\left(w_{\mathcal{M}}(t,x)\right)\right)-U^{'}\left({-}H_{0,1}\left(w_{\mathcal{M}}(t,{-}x)\right)\right)\\ \nonumber
&{+}2\sqrt{2}r_{\mathcal{M}+1}(t) e^{{-}\sqrt{2}d(t)}\left[24 M\left(w_{\mathcal{M}}(t,x)\right)-30N\left(w_{\mathcal{M}}(t,x)\right)\right]
    \\ \nonumber &{-}2\sqrt{2}r_{\mathcal{M}+1}(t) e^{{-}\sqrt{2}d(t)}\left[24 M\left(w_{\mathcal{M}}(t,{-}x)\right)-30N\left(w_{\mathcal{M}}(t,{-}x)\right)\right]\\ \nonumber &{+}\frac{r_{\mathcal{M}+1}e^{{-}\sqrt{2}d(t)}}{\sqrt{1-\frac{\dot d(t)^{2}}{4}}}\left[24 M^{'}\left(w_{\mathcal{M}}(t,x)\right)-30 N^{'}\left(w_{\mathcal{M}}(t,x)\right)\right]\\ \label{intwM} & {-}\frac{r_{\mathcal{M}+1}e^{{-}\sqrt{2}d(t)}}{\sqrt{1-\frac{\dot d(t)^{2}}{4}}}\left[24 M^{'}\left(w_{\mathcal{M}}(t,{-}x)\right)-30 N^{'}\left(w_{\mathcal{M}}(t,{-}x)\right)\right].
\end{align}
\textbf{Step 3.}(Estimate of $\mathcal{L}_{\mathcal{M}+1,4}.$)
From Lemma \ref{porrataylor}, if $0<v\ll 1,$ we deduce for every $s\geq 1$ and every $l\in\mathbb{N}\cup\{0\}$ that
\begin{equation*}
    \norm{\frac{\partial^{l}}{\partial t^{l}}\left[\mathcal{G}\left(w_{\mathcal{M}}(t,x+r_{\mathcal{M}+1}(t))\right)-\mathcal{G}\left(w_{\mathcal{M}}(t,x)\right)\right]}_{H^{s}_{x}}\lesssim_{s,l} v^{2\mathcal{M}+l}\left(\ln{\left(\frac{1}{v}\right)}+\vert t\vert v\right)^{n_{\mathcal{M}+1}},
\end{equation*}
which implies with Lemma \ref{dlemma} the following estimate
\begin{equation*}
\frac{\partial^{2}}{\partial t^{2}}\left[e^{{-}\sqrt{2}d(t)}\mathcal{G}\left(w_{\mathcal{M}}(t,x+r_{\mathcal{M}+1}(t))\right)\right]\cong
\frac{\partial ^{2}}{\partial t^{2}}\left[e^{{-}\sqrt{2}d(t)}\mathcal{G}\left(w_{\mathcal{M}}(t,x)\right)\right].
\end{equation*}
Moreover, 
using Lemma \ref{dlemma} and estimate \ref{rmd}, Lemma \ref{porrataylor} also implies
\begin{multline*}
    e^{{-}\sqrt{2}d(t)}\left(\left[{-}\frac{\partial^{2}}{\partial x^{2}}+ U^{''}\left(H_{0,1}\left(w_{\mathcal{M}}(t,x+r_{\mathcal{M}+1}(t))\right)\right)\right]\mathcal{G}\left(w_{\mathcal{M}}(t,x+r_{\mathcal{M}+1}(t))\right)\right)\\
    \cong_{2\mathcal{M}+4} e^{{-}\sqrt{2}d(t)}\left(\left[{-}\frac{\partial^{2}}{\partial x^{2}}+U^{''}\left(H_{0,1}\left(w_{\mathcal{M}}(t,x)\right)\right)\right]\mathcal{G}\left(w_{\mathcal{M}}(t,x)\right)\right)\\
    +r_{\mathcal{M}+1}(t)e^{{-}\sqrt{2}d(t)}\frac{\partial}{\partial x}\left(\left[{-}\frac{\partial^{2}}{\partial x^{2}}+ U^{''}\left(H_{0,1}\left(w_{\mathcal{M}}(t,x)\right)\right)\right]\mathcal{G}\left(w_{\mathcal{M}}(t,x)\right)\right).
\end{multline*}
Therefore, 
\begin{multline*}
\left[\frac{\partial^{2}}{\partial t^{2}}-\frac{\partial^{2}}{\partial x^{2}}+ U^{''}\left(H_{0,1}\left(w_{\mathcal{M}}(t,x+r_{\mathcal{M}+1}(t))\right)\right)\right]\left(e^{{-}\sqrt{2}d(t)}\mathcal{G}\left(w_{\mathcal{M}}(t,x+r_{\mathcal{M}+1}(t))\right)\right)\\
    \begin{aligned}
     \cong_{2\mathcal{M}+4}& 
     \left[\frac{\partial^{2}}{\partial t^{2}}-\frac{\partial^{2}}{\partial x^{2}}+ U^{''}\left(H_{0,1}\left(w_{\mathcal{M}}(t,x)\right)\right)\right]\left(e^{{-}\sqrt{2}d(t)}\mathcal{G}\left(w_{\mathcal{M}}(t,x)\right)\right)\\
     &{+}r_{\mathcal{M}+1}(t)e^{{-\sqrt{2}}d(t)}\frac{\partial}{\partial x}\left(\left[{-}\frac{\partial^{2}}{\partial x^{2}}+ U^{''}\left(H_{0,1}\left(w_{\mathcal{M}}(t,x)\right)\right)\right]\mathcal{G}\left(w_{\mathcal{M}}(t,x)\right)\right).
     \end{aligned}
\end{multline*}
In conclusion, recalling the notation $h^{w_{\mathcal{M}}}(t,x)=f\left(w_{\mathcal{M}}(t,x)\right)-h\left(w_{\mathcal{M}}(t,{-}x)\right)$ for any function $h:\mathbb{R}\to\mathbb{R},$ using Lemma \ref{dlemma}, identity \eqref{lm4} and estimate \eqref{rmd}, we obtain the following estimate
\begin{multline}
  \begin{aligned}
  \mathcal{L}_{\mathcal{M}+1,4}(t,x)\cong_{2\mathcal{M}+4}&   \,r_{\mathcal{M}+1}(t)e^{{-}\sqrt{2}d(t)}\frac{\partial}{\partial x}\left(\left[{-}\mathcal{G}^{(2)}+ U^{''}(H_{0,1})\mathcal{G}\right]^{w_{\mathcal{M}}}(t,x)\right)\\
  &{+}e^{{-}\sqrt{2}d(t)}\left[{-}\mathcal{G}^{(2)}+ U^{''}(H_{0,1})\mathcal{G}\right]^{w_{\mathcal{M}}}(t,x)+\frac{\partial^{2}}{\partial t^{2}}\left(e^{{-}\sqrt{2}d(t)}\mathcal{G}^{w_{\mathcal{M}}}(t,x)\right).
\end{aligned}
\end{multline}
\textbf{Step 4.}(Estimate of $\mathcal{L}_{\mathcal{M}+1,1}.$)
Since Lemma \ref{represent1} implies for all $s\geq 1,\,l\in\mathbb{N}\cup\{0\}$ that $\norm{\frac{\partial^{l}}{\partial t^{l}}L_{1}\left(\Gamma(t,\cdot)\right)(x)}_{H^{s}}\lesssim_{s,l} v^{2\mathcal{M}+l}\left(v\vert t\vert+\ln{\left(\frac{1}{v}\right)}\right)^{n_{\mathcal{M}}}e^{{-}2\sqrt{2}v\vert t \vert}$ if $0<v\ll 1,$ we can repeat the argument in the second step and obtain, from identity \eqref{lm1}, Lemma \ref{porrataylor} and the estimates in \ref{rmd}, that
\begin{align*}
 \mathcal{L}_{\mathcal{M}+1,1}(t,x)=\cong_{2\mathcal{M}+4}&
   {-}\left[\frac{\partial^{2}}{\partial t^{2}}-\frac{\partial^{2}}{\partial x^{2}}+U^{''}\left(H_{0,1}\left(w_{\mathcal{M}}(t,x)\right)\right)\right]L_{1}\left(\Gamma(t,\cdot)\right)\left(w_{\mathcal{M}}(t,x)\right)\\ \cong_{2\mathcal{M}+4} & {-}\Gamma\left(t,w_{\mathcal{M}}(t,x)\right)+\frac{\dot d(t)^{2}}{4-\dot d(t)^{2}}\frac{\partial^{2}}{\partial y^{2}}\Big\vert_{y=w_{\mathcal{M}}(t,x)}L_{1}\left(\Gamma(t,\cdot)\right)(y)\\&{-}\frac{\partial^{2}}{\partial t^{2}}L_{1}\left(\Gamma(t,\cdot)\right)\left(w_{\mathcal{M}}(t,x)\right).
\end{align*}
\textbf{Step 5.}(Estimate of $\mathcal{L}_{\mathcal{M}+1,5}.$) Lemma \ref{porrataylor} and estimate \eqref{rmd} imply for all $m\in\mathbb{N},\,l\in\mathbb{N}\cup\{0\}$ the following estimates
\begin{align}\label{nhhm}
   \norm{\frac{\partial^{l}}{\partial t^{l}}\Big[H_{0,1}\left(w_{\mathcal{M}}(t,\pm x+r_{\mathcal{M}+1})\right)^{m}-H_{0,1}\left(w_{\mathcal{M}}(t,\pm x)\right)^{m}\Big]}_{H^{s}_{x}}\lesssim_{m,s,l} v^{2\mathcal{M}+l}\left[\ln{\frac{1}{v}}\right]^{n_{\mathcal{M}+1}},\\ \label{nggm}
   \norm{\frac{\partial^{l}}{\partial t^{l}}\Big[\mathcal{G}\left(w_{\mathcal{M}}(t,x+r_{\mathcal{M}+1})\right)^{m}-\mathcal{G}\left(w_{\mathcal{M}}(t,x)\right)^{m}\Big]}_{H^{s}_{x}}\lesssim_{m,s,l} v^{2\mathcal{M}+l}\left[\ln{\frac{1}{v}}\right]^{n_{\mathcal{M}+1}},
\end{align}
if $0<v\ll 1.$ Therefore, since 
\begin{equation*}
 U^{''}\left(H_{0,1}\left(w_{\mathcal{M}}(t,x+r_{\mathcal{M}+1})-H_{0,1}\left(t,{-}x+r_{\mathcal{M}+1}\right)\right)\right)- U^{''}\left(H_{0,1}\left(w_{\mathcal{M}}(t,x+r_{\mathcal{M}+1})\right)\right)
\end{equation*} 
is a real linear combination of functions $H_{0,1}\left(w_{\mathcal{M}}(t,x+r_{\mathcal{M}+1}\right)^{m}H_{0,1}\left(w_{\mathcal{M}}(t,{-}x+r_{\mathcal{M}+1}\right)^{n}$ such that $m\in\mathbb{N}\cup\{0\}$ and $n\in\mathbb{N},$ we deduce using the identity \eqref{lm5} and Lemma \ref{dlemma} the following estimate
\begin{equation*}
    \mathcal{L}_{M+1,5}(t,x)\cong_{2\mathcal{M}+4}
    e^{{-}\sqrt{2}d(t)} U^{''}\left(H^{w_{\mathcal{M}}}_{0,1}(t,x)\right)\mathcal{G}^{w_{\mathcal{M}}}(t,x)-e^{{-}\sqrt{2}d(t)}\left[U^{''}(H_{0,1})\mathcal{G}\right]^{w_{\mathcal{M}}}(t,x),
\end{equation*}
where $f^{w_{\mathcal{M}}}(t,x)=f\left(w_{\mathcal{M}}(t,x)\right)-f\left(w_{\mathcal{M}}(t,{-}x)\right)$ for any function $f:\mathbb{R}\to\mathbb{R}$ and $(t,x)\in\mathbb{R}^{2}.$\\
\textbf{Step 6.}(Estimate of $\mathcal{L}_{\mathcal{M}+1,6}.$)
From the definition of the functions $\varphi_{\mathcal{M},v},\,\phi_{\mathcal{M}+1,0,v},\,\mathcal{L}_{\mathcal{M}+1,6}$ respectively in \eqref{phiMM}, \eqref{phim10},  \eqref{lm6} and using the notation
\begin{equation*}
    \mathcal{S}(v,t,x)=\phi_{\mathcal{M}+1,v,0}(t,x)-H_{0,1}\left(w_{\mathcal{M}}(t,x+r_{\mathcal{M}+1})\right)+H_{0,1}\left(w_{\mathcal{M}}(t,{-}x+r_{\mathcal{M}+1})\right),
\end{equation*}
we have the following identity 
\begin{multline*}
  \begin{aligned} \mathcal{L}_{\mathcal{M}+1,6}(t,x)
=&\sum_{j=3}^{6}\frac{U^{(j)}\left(H_{0,1}\left(w_{\mathcal{M}}(t,x+r_{\mathcal{M}+1})\right)-H_{0,1}\left(w_{\mathcal{M}}(t,{-}x+r_{\mathcal{M}+1})\right)\right)}{(j-1)!} \Big[\mathcal{S}(v,t,x)\Big]^{j-1}
    \\
   &{+}  U^{''}\left(H_{0,1}\left(w_{\mathcal{M}}(t,x+r_{\mathcal{M}+1})\right)-H_{0,1}\left(w_{\mathcal{M}}(t,{-}x+r_{\mathcal{M}+1})\right)\right)\\&\times\left[\mathcal{T}_{\mathcal{M}}(vt,w_{\mathcal{M}}(t,x+r_{\mathcal{M}+1}))-\mathcal{T}_{\mathcal{M}}(vt,w_{\mathcal{M}}(t,{-}x+r_{\mathcal{M}+1}))\right].   
\end{aligned}
\end{multline*}
\par Furthermore, from the assumption that Theorem \ref{strongerr} is true for any $k\in\mathbb{N}$ satisfying $2\leq k\leq \mathcal{M},$ we have the following estimate
\begin{equation*}
    \norm{\frac{\partial^{l}}{\partial t^{l}}\left[\mathcal{T}_{\mathcal{M}}\left(vt,w_{\mathcal{M}}(t,x)\right)\right]}_{H^{s}_{x}}\lesssim_{s,l} v^{4+l}\left(\vert t\vert v+\ln{\left(\frac{1}{v}\right)}\right)^{c_{k}}e^{{-}2\sqrt{2}\vert t\vert v},
\end{equation*}
for some positive constant $c_{k},$ all $s\geq 0,$ and any $l\in\mathbb{N}\cup\{0\}$ if $0<v\ll 1.$
Therefore, using Lemma \ref{porrataylor} and estimate \eqref{rmd}, we obtain that if $0<v\ll 1,$ then the following inequality 
\begin{equation*}
    \norm{\frac{\partial^{l}}{\partial t^{l}}\left[\mathcal{T}_{\mathcal{M}}\left(vt,w_{\mathcal{M}}(t,\pm x+r_{\mathcal{M}+1}(t))\right)\right]}_{H^{s}_{x}}\lesssim_{s,l} v^{4+l}\left(\vert t\vert v+\ln{\left(\frac{1}{v}\right)}\right)^{c_{k}}e^{{-}2\sqrt{2}\vert t\vert v}  
\end{equation*}
is true for every $s\geq 0$ and any $l\in\mathbb{N}\cup\{0\}.$
Thus, using estimates \eqref{nhhm}, \eqref{nggm} and the inequality
$
    \norm{fg}_{H^{s}_{x}}\lesssim_{s}\norm{f}_{H^{s}_{x}}\norm{g}_{H^{s}_{x}}$ , for all $f,\,g\in H^{s}_{x},$
for any $s>\frac{1}{2},$
we deduce that
\begin{multline}\label{lm5es}
    \mathcal{L}_{M+1,6}(t,x)\cong_{2\mathcal{M}+4}\,
U^{''}\left(H^{w_{\mathcal{M}}}_{0,1}(t,x)\right)\Big[\mathcal{T}_{\mathcal{M}}(vt,w_{\mathcal{M}}(t,x))
   -\mathcal{T}_{\mathcal{M}}(vt,w_{\mathcal{M}}(t,{-}x))
    \Big]\\
    {+}\sum_{j=3}^{6}\frac{U^{(j)}\left(H^{w_{\mathcal{M}}}_{0,1}(t,x)\right)\left[\phi_{\mathcal{M},v}(t,x)-H^{w_{\mathcal{M}}}_{0,1}(t,x)\right]^{(j-1)}}{(j-1)!} .
\end{multline}
\textbf{Step 7.}(Estimate of $\mathcal{L}_{\mathcal{M}+1,2}.$)
We recall that $\mathcal{L}_{M+1,2}$ is defined at \eqref{lm2}.
To simplify the estimate of this function, we are going to estimate separately the functions
\begin{multline*}
\mathcal{L}_{\mathcal{M}+1,2,1}(t,x)={-}\Big[ U^{''}\left(H_{0,1}\left(w_{\mathcal{M}}(t,x+r_{\mathcal{M}+1})\right)-H_{0,1}\left(w_{\mathcal{M}}(t,{-}x+r_{\mathcal{M}+1})\right)\right)\\- U^{''}\left(H_{0,1}\left(w_{\mathcal{M}}(t,x+r_{\mathcal{M}+1})\right)\right)\Big]L_{1}\left(\Gamma(t,\cdot)\right)\left(w_{\mathcal{M}}(t,x+r_{\mathcal{M}+1})\right),
\end{multline*}
and
\begin{multline*}
\mathcal{L}_{\mathcal{M}+1,2,2}(t,x)=\Big[U^{''}\Big(H_{0,1}\left(w_{\mathcal{M}}(t,x+r_{\mathcal{M}+1})\right)-H_{0,1}\left(w_{\mathcal{M}}(t,{-}x+r_{\mathcal{M}+1})\right)\Big)\\{-}U^{''}\Big(\phi_{\mathcal{M},v,0}(t,x)\Big)\Big]L_{1}\left(\Gamma(t,\cdot)\right)\left(w_{\mathcal{M}}(t,x+r_{\mathcal{M}+1})\right),
\end{multline*}
the sum of these functions is $\mathcal{L}_{M+1,2}(t,x).$ 
\par First, from Taylor's Theorem, we have
\begin{multline*}
\mathcal{L}_{\mathcal{M}+1,2,1}(t,x)\\=L_{1}\left(\Gamma(t,\cdot)\right)\left(w_{\mathcal{M}}(t,x+r_{\mathcal{M}+1})\right)\Big[U^{(3)}\left(H_{0,1}\left(w_{\mathcal{M}}(t,x+r_{\mathcal{M}+1})\right)\right) H_{0,1}\left(w_{\mathcal{M}}(t,{-}x+r_{\mathcal{M}+1})\right)\\
    +\sum_{j=4}^{6}\frac{({-}1)^{(j-1)}}{(j-2)!}U^{(j)}\left(H_{0,1}\left(w_{\mathcal{M}}(t,x+r_{\mathcal{M}+1})\right)\right)H_{0,1}\left(w_{\mathcal{M}}(t,{-}x+r_{\mathcal{M}+1})\right)^{j-2}\Big]
\end{multline*}
Moreover, using Lemmas \ref{porrataylor},  \ref{represent1}, identity \eqref{explicityB}, estimate \eqref{rmd} and the product rule of derivative, we can verify that
\begin{multline*}
\mathcal{L}_{\mathcal{M}+1,2,1}(t,x) \cong_{2\mathcal{M}+4}L_{1}\left(\Gamma(t,\cdot)\right)\left(w_{\mathcal{M}}(t,x+r_{\mathcal{M}+1})\right)\\ \times U^{(3)}\left(H_{0,1}\left(w_{\mathcal{M}}(t,x+r_{\mathcal{M}+1})\right)\right) H_{0,1}\left(w_{\mathcal{M}}(t,{-}x+r_{\mathcal{M}+1})\right)\\ \cong_{2\mathcal{M}+4}
L_{1}\left(\Gamma(t,\cdot)\right)\left(w_{\mathcal{M}}(t,x)\right)U^{(3)}\left(H_{0,1}\left(w_{\mathcal{M}}(t,x)\right)\right) H_{0,1}\left(w_{\mathcal{M}}(t,{-}x)\right).
\end{multline*}
Furthermore, since $\left\vert \frac{d^{k}}{dx^{k}}\left[H_{0,1}(x)-e^{\sqrt{2}x}\right]\right\vert\lesssim_{k}\min\left(e^{2\sqrt{2}x},e^{\sqrt{2}x}\right),$ we can deduce using
Lemmas \ref{interactt}, \ref{porrataylor}, \ref{interactionsize} and estimate \eqref{nhhm} that
\begin{align*}
\mathcal{L}_{\mathcal{M}+1,2,1}(t,x)
\cong_{2\mathcal{M}+4} & L_{1}\left(\Gamma(t,\cdot)\right)\left(w_{\mathcal{M}}(t,x)\right)U^{(3)}\left(H_{0,1}\left(w_{\mathcal{M}}(t,x)\right)\right)e^{{-}\sqrt{2}w_{\mathcal{M}}(t,x)}e^{{-}\sqrt{2}d(t)},
\end{align*}
so $\norm{\frac{\partial^{l}}{\partial t^{l}}\mathcal{L}_{\mathcal{M}+1,2,1}(t,x)}_{H^{s}_{x}}\lesssim_{s,l} v^{2\mathcal{M}+2}\left(\vert t\vert v+\ln{\left(\frac{1}{v}\right)}\right)^{n_{\mathcal{M}}}e^{{-}2\sqrt{2}\vert t \vert v}.$
\par Next, let $w_{\mathcal{M}+1}:\mathbb{R}^{2}\to\mathbb{R}$ be the unique function satisfying
\begin{equation}\label{wM1}
    w_{\mathcal{M}+1}(t,x)=w_{\mathcal{M}}(t,x+r_{\mathcal{M}+1}(t)) \text{, for all $(t,x)\in\mathbb{R}^{2}.$}
\end{equation}
Since we are assuming that Theorem \ref{strongerr} is true for $k=\mathcal{M},$ Lemmas \ref{interactionsize}, \ref{represent1} and the following identity 
\begin{multline*}
\begin{aligned}
U^{''}\left(\phi_{\mathcal{M},v,0}(t,x)\right)
    =&  U^{''}\left(H^{w_{\mathcal{M}+1}}_{0,1}(t,x)\right)
    +e^{{-}\sqrt{2}d(t)}U^{(3)}\left(H^{w_{\mathcal{M}+1}}(t,x)\right)\mathcal{G}^{w_{\mathcal{M}+1}}(t,x)\\&{+}U^{(3)}\left(H^{w_{\mathcal{M}+1}}_{0,1}(t,x)\right)\left[\mathcal{T}_{\mathcal{M}}\left(vt,w_{\mathcal{M}+1}(t,x)\right)
    {-}\mathcal{T}_{\mathcal{M}}\left(vt,w_{\mathcal{M}+1}(t,{-}x)\right)\right]\\
    &{+}\sum_{j=4}^{6}\frac{1}{(j-2)!}U^{(j)}\left(H^{w_{\mathcal{M}+1}}_{0,1}(t,x)\right)\left[\phi_{\mathcal{M},v,0}(t,x)
-H^{w_{\mathcal{M}+1}}_{0,1}(t,x)
    \right]^{j-2}
\end{aligned}
\end{multline*}
imply 
\begin{multline*}
    L_{1}\left(\Gamma(t,\cdot)\right)\left(w_{\mathcal{M}}(t,x+r_{\mathcal{M}+1})\right) U^{''}\left(\phi_{\mathcal{M},v,0}(t,x)\right)\\
\begin{aligned}
    \cong_{2\mathcal{M}+4} & \,  U^{''}\left(H^{w_{\mathcal{M}+1}}_{0,1}(t,x)\right)L_{1}\left(\Gamma(t,\cdot)\right)\left(w_{\mathcal{M}}(t,x+r_{\mathcal{M}+1})\right)\\
&{+}e^{{-}\sqrt{2}d(t)}U^{(3)}\left(H^{w_{\mathcal{M}+1}}_{0,1}(t,x)\right)\mathcal{G}^{w_{\mathcal{M}+1}}(t,x)L_{1}\left(\Gamma(t,\cdot)\right)\left(w_{\mathcal{M}}(t,x+r_{\mathcal{M}+1})\right).
\end{aligned}
\end{multline*}
Thus, we obtain that
\begin{equation*}
    \mathcal{L}_{\mathcal{M}+1,2,2}(t,x)\cong_{2\mathcal{M}+4}{-} e^{{-}\sqrt{2}d(t)}U^{(3)}\left(H^{w_{\mathcal{M}+1}}_{0,1}(t,x)\right)\mathcal{G}^{w_{\mathcal{M}+1}}(t,x)L_{1}\left(\Gamma(t,\cdot)\right)\left(w_{\mathcal{M}}(t,x+r_{\mathcal{M}+1})\right).
\end{equation*}
Indeed, using Lemma \ref{porrataylor} and estimates \eqref{rmd}, we deduce from the estimate above that
\begin{equation*}
    \mathcal{L}_{\mathcal{M}+1,2,2}(t,x)\cong_{2\mathcal{M}+4}{-} e^{{-}\sqrt{2}d(t)}U^{(3)}\left(H^{w_{\mathcal{M}}}_{0,1}(t,x)\right)\mathcal{G}^{w_{\mathcal{M}}}(t,x)L_{1}\left(\Gamma(t,\cdot)\right)\left(w_{\mathcal{M}}(t,x)\right).
\end{equation*}
Furthermore, Lemmas \ref{represent1} and \ref{dlemma} implies, for any $1\leq i\leq N_{1},$
\begin{equation*}
\left\vert \frac{d^{l}}{dt^{l}}\left[ s_{i,v}\left(\sqrt{2}vt\right) e^{{-}\sqrt{2}d(t)} \right]\right\vert\lesssim_{l}v^{2\mathcal{M}+2+l}\left(\vert t\vert v+\ln{\left(\frac{1}{v}\right)}\right)^{n_{\mathcal{M}}}e^{{-}2\sqrt{2}\vert t\vert v},
\end{equation*}
for all $l\in\mathbb{N}\cup\{0\},$ if $0<v\ll 1.$ Moreover, using Lemma \ref{interactionsize}, estimate \eqref{L1B} and the inequality
$\norm{fg}_{H^{s}_{x}}\lesssim_{s} \norm{f}_{H^{s+1}_{x}}\norm{g}_{H^{s+1}_{x}},
$
for any $f,\,g\in\mathscr{S}(\mathbb{R})$ and all $s\geq 0,$ we deduce that
\begin{equation*}
\mathcal{L}_{\mathcal{M}+1,2,2}(t,x)\cong_{2\mathcal{M}+4} {-}e^{{-}\sqrt{2}d(t)}U^{(3)}\left(H_{0,1}\left(w_{\mathcal{M}}(t,x)\right)\right)\mathcal{G}\left(w_{\mathcal{M}}(t,x)\right)L_{1}\left(\Gamma(t,\cdot)\right)\left(w_{\mathcal{M}}(t,x)\right).
\end{equation*}
\par Consequently, we obtain that
\begin{equation}\label{seventh}
  \norm{ \frac{\partial^{l}}{\partial t^{l}}\mathcal{L}_{\mathcal{M}+1,2}(t,x)}_{H^{s}_{x}}\lesssim_{s,l} v^{2\mathcal{M}+2+l}\left(\vert t\vert v+\ln{\left(\frac{1}{v}\right)}\right)^{n_{\mathcal{M}}}e^{{-}2\sqrt{2}\vert t\vert v},
\end{equation}
and
\begin{align*}
    \mathcal{L}_{\mathcal{M}+1,2}(t,x)\cong_{2\mathcal{M}+4} {-}e^{{-}\sqrt{2}d(t)}U^{(3)}\left(H_{0,1}\left(w_{\mathcal{M}}(t,x)\right)\right)\mathcal{G}\left(w_{\mathcal{M}}(t,x)\right)L_{1}\left(\Gamma(t,\cdot)\right)\left(w_{\mathcal{M}}(t,x)\right)\\
    {+}e^{{-}\sqrt{2}d(t)}L_{1}\left(\Gamma(t,\cdot)\right)\left(w_{\mathcal{M}}(t,x)\right)U^{(3)}\left(H_{0,1}\left(w_{\mathcal{M}}(t,x)\right)\right) e^{{-}\sqrt{2}w_{\mathcal{M}}(t,x)}.
\end{align*}
\textbf{Step 8.}(Estimate of $\Lambda(\phi_{\mathcal{M}+1,v}).$) From the equation \eqref{Lambdaphi0,v} and the conclusions obtained in all the steps before, we deduce
\begin{multline*}
\Lambda(\varphi_{\mathcal{M}+1,v})(t,x)- \Lambda(\phi_{\mathcal{M}+1,v,0})(t,x)\\ 
\begin{aligned}
\cong_{2\mathcal{M}+4} &{-}\Gamma\left(t,w_{\mathcal{M}}(t,x)\right)+\Gamma\left(t,w_{\mathcal{M}}(t,{-}x)\right)\\&{+}\frac{\dot d(t)^{2}}{4-\dot d(t)^{2}}\frac{\partial^{2}}{\partial y^{2}}\Big\vert_{y=w_{\mathcal{M}}(t,x)}L_{1}\left(\Gamma(t,\cdot)\right)(y)-\frac{\dot d(t)^{2}}{4-\dot d(t)^{2}}\frac{\partial^{2}}{\partial y^{2}}\Big\vert_{y=w_{\mathcal{M}}(t,{-}x)}L_{1}\left(\Gamma(t,\cdot)\right)(y)\\ &{-}\frac{\partial^{2}}{\partial t^{2}}L_{1}\left(\Gamma(t,\cdot)\right)\left(w_{\mathcal{M}}(t,x)\right)+\frac{\partial^{2}}{\partial t^{2}}L_{1}\left(\Gamma(t,\cdot)\right)\left(w_{\mathcal{M}}(t,{-}x)\right)\\ &{+}e^{{-}\sqrt{2}d(t)}L_{1}\left(\Gamma(t,\cdot)\right)\left(w_{\mathcal{M}}(t,x)\right)U^{(3)}\left(H_{0,1}\left(w_{\mathcal{M}}(t,x)\right)\right) e^{{-}\sqrt{2}w_{\mathcal{M}}(t,x)}\\ &{-}e^{{-}\sqrt{2}d(t)}L_{1}\left(\Gamma(t,\cdot)\right)\left(w_{\mathcal{M}}(t,{-}x)\right)U^{(3)}\left(H_{0,1}\left(w_{\mathcal{M}}(t,{-}x)\right)\right)e^{{-}\sqrt{2}w_{\mathcal{M}}(t,x)}\\
   & {-}e^{{-}\sqrt{2}d(t)}U^{(3)}\left(H_{0,1}\left(w_{\mathcal{M}}(t,x)\right)\right)\mathcal{G}\left(w_{\mathcal{M}}(t,x)\right)L_{1}\left(\Gamma(t,\cdot)\right)\left(w_{\mathcal{M}}(t,x)\right)\\&{+}e^{{-}\sqrt{2}d(t)}U^{(3)}\left(H_{0,1}\left(w_{\mathcal{M}}(t,{-}x)\right)\right)\mathcal{G}\left(w_{\mathcal{M}}(t,{-}x)\right)L_{1}\left(\Gamma(t,\cdot)\right)\left(w_{\mathcal{M}}(t,{-}x)\right).
   \end{aligned}
\end{multline*}
 Furthermore, from \eqref{LambdahM1} and the estimates of $\mathcal{L}_{\mathcal{M}+1,j}$ for $3\leq j\leq 6,$ we deduce
\begin{align*}\label{ffff}
\Lambda(\phi_{\mathcal{M}+1,v,0})(t,x)\\ \cong_{2\mathcal{M}+4} & \Lambda(\varphi_{\mathcal{M},v})(t,x)+2\sqrt{2}r_{\mathcal{M}+1}(t) e^{{-}\sqrt{2}d(t)}\left[24 M^{w_{\mathcal{M}}}(t,x)-30N^{w_{\mathcal{M}}}(t,x)\right]\\&{+}\frac{r_{\mathcal{M}+1}(t)e^{{-}\sqrt{2}d(t)}}{\sqrt{1-\frac{\dot d(t)^{2}}{4}}}\left[24\left( M^{'}\right)^{w_{\mathcal{M}}}(t,x)-30 \left(N^{'}\right)^{w_{\mathcal{M}}}(t,x)\right]\\
&{+}r_{\mathcal{M}+1}(t)\frac{\partial}{\partial x}\left(\left[{-}\mathcal{G}^{(2)}+U^{''}\left(H_{0,1}\right)\mathcal{G}\right]^{w_{\mathcal{M}}}(t,x)\right)e^{{-}\sqrt{2}d(t)}\\
    &{+}\frac{\ddot r_{\mathcal{M}+1}(t)}{\sqrt{1-\frac{\dot d(t)^{2}}{4}}}\left[H^{'}_{0,1}\left(w_{\mathcal{M}}\left(t,x\right)\right)-H^{'}_{0,1}\left(w_{\mathcal{M}}\left(t,{-}x\right)\right)\right]\\ &{-}\frac{\dot r_{\mathcal{M}+1}(t)\dot d(t)}{1-\frac{\dot d(t)^{2}}{4}}\left[ H^{''}_{0,1}\left(w_{\mathcal{M}}\left(t,x\right)\right)-H^{''}_{0,1}\left(w_{\mathcal{M}}\left(t,{-}x\right)\right)\right],
\end{align*}
from which with Remark \eqref{Gequationn} we deduce that
\begin{align*}
\Lambda(\phi_{\mathcal{M}+1,v,0})(t,x)\\ \cong_{2\mathcal{M}+4} & \Lambda(\varphi_{\mathcal{M},v})(t,x)+2\sqrt{2}r_{\mathcal{M}+1}(t) e^{{-}\sqrt{2}d(t)}\left[24 M^{w_{\mathcal{M}}}(t,x)-30N^{w_{\mathcal{M}}}(t,x)\right]\\&{+}\frac{8\sqrt{2}r_{\mathcal{M}+1}(t)e^{{-}\sqrt{2}d(t)}-\dot r_{\mathcal{M}+1}(t)\dot d(t)}{\sqrt{1-\frac{\dot d(t)^{2}}{4}}}\left( H^{''}_{0,1}\right)^{w_{\mathcal{M}}}(t,x)\\&{+}\frac{\ddot r_{\mathcal{M}+1}(t)}{\sqrt{1-\frac{\dot d(t)^{2}}{4}}}\left(H^{'}_{0,1}\right)^{w_{\mathcal{M}}}(t,x),
\end{align*}
We also have, from Lemmas \ref{represent1}, \ref{projectionl}, for all $l\in\mathbb{N}\cup\{0\}$ and any $s\geq 0$ that if $0<v\ll 1,$ then
\begin{multline*}
    \norm{\frac{\partial^{l}}{\partial t^{l}}\left[\Lambda\left(\varphi_{\mathcal{M},v}\right)(t,x)- \Gamma\left(t,w_{\mathcal{M}}(t,x)\right)+\Gamma\left(t,w_{\mathcal{M}}(t,{-}x)\right)\right]}_{H^{s}_{x}}\\ \lesssim_{s,l} v^{2\mathcal{M}+2+l}\left(\vert t\vert v+\ln{\left(\frac{1}{v}\right)}\right)^{n_{\mathcal{M}}}e^{{-}2\sqrt{2}\vert t\vert v}.
\end{multline*}
Therefore, from the estimates above, inequalities \eqref{seventh}, \eqref{rmd},
Lemmas \ref{dlemma}, \ref{represent1} and Remark \ref{perturbt}, we obtain that the estimate \eqref{geraldecay} of Theorem \ref{strongerr} is true for $k=\mathcal{M}+1.$ 
\par Furthermore, Lemma \ref{porrataylor} and \eqref{rmd} imply that if $h\in S^{+}_{\infty},$ then we have for all $l\in\mathbb{N}\cup\{0\}$
the following inequality
\begin{multline*}
     \left\vert\frac{d^{l}}{d t^{l}}\left\langle h\left(w_{\mathcal{M}}(t,x+r_{\mathcal{M}+1}(t))\right)-h\left(w_{\mathcal{M}}(t,x)\right),  H^{'}_{0,1}\left(w_{\mathcal{M}}(t,x+r_{\mathcal{M}+1}(t))\right) \right\rangle\right\vert \\ \lesssim_{l}
     v^{2\mathcal{M}+l}\left[\ln{\frac{1}{v}}\right]^{n_{\mathcal{M}+1}}.
\end{multline*}
Therefore, the estimates above, Remark \ref{elint}, the ordinary differential equation \eqref{odek} satisfied by $r_{\mathcal{M}+1}$ and estimate \eqref{rmd} of the derivatives of $r_{\mathcal{M}+1}$ imply \eqref{orthodecay} for $k=\mathcal{M}+1.$ In conclusion, by induction on $k,$ we deduce that Theorem \ref{strongerr} is true for all $k\in\mathbb{N}_{\geq 2}.$
\end{proof}
\begin{remark}\label{limisup}
From Theorem \ref{strongerr}, we have that if $v\ll1,$ then
$\lim_{t\to+\infty} \sum_{k=1}^{\mathcal{M}}r_{k}(v,t) $ exists.
\end{remark}
\subsection{Proof of Theorem \ref{approximated theorem}}
\begin{proof}[Proof of Theorem \ref{approximated theorem}]
The Theorem \ref{strongerr} implies the existence, for any $k\in\mathbb{N}_{\geq 2},$ of a smooth function $\varphi_{k,v}(t,x)$ and a even function $r(t)\in L^{\infty}(\mathbb{R})$ such that if $v\ll1,$ then
\begin{gather*}
   \lim_{t\to \pm\infty} \norm{\varphi_{k,v}(t,x)-H_{0,1}\left(\frac{x\mp vt +r(t)}{\sqrt{1+v^{2}}}\right)-H_{-1,0}\left(\frac{x\pm vt -r(t)}{\sqrt{1+v^{2}}}\right)}_{H^{1}_{x}}=0,\\
   \lim_{t\to \pm\infty} \norm{\partial_{t}\varphi_{k,v}(t,x)\pm \frac{v}{\sqrt{1-v^{2}}} H^{'}_{0,1}\left(\frac{x\mp vt +r(t)}{\sqrt{1+v^{2}}}\right)\mp \frac{v}{\sqrt{1-v^{2}}} H^{'}_{-1,0}\left(\frac{x\pm vt -r(t)}{\sqrt{1+v^{2}}}\right)}_{L^{2}_{x}}\\=0,
\end{gather*}
and $\lim_{t\to+\infty} \vert r(t)\vert \lesssim v^{2}\ln{\left(\frac{1}{v}\right)}.$ In conclusion, from Lemma \ref{dlemma}, Remark \ref{limisup} and Theorem \ref{strongerr}, the function
\begin{equation*}
    \phi_{k}(v,t,x)=\varphi_{k,v}\left(t+\frac{\ln{\left(v^{2}\right)}-\ln{(8)}}{2\sqrt{2}v}+\lim_{s\to+\infty}\frac{r(s)}{v},x\right)
\end{equation*}
satisfies Theorem \ref{approximated theorem}.
\end{proof}

\section*{Acknowledgement(s)}

The author would like to thank his supervisors
Thomas Duyckaerts and Jacek Jendrej for providing helpful comments and
orientation, which were essential to conclude this paper. Finally, the author
is grateful to the math department LAGA of the University Sorbonne Paris
Nord and all the referees for providing remarks and suggestions in the writing of
this manuscript.

\section*{Funding}

The author acknowledges the support of the French State Program ”Investisse-
ment d’Avenir”, managed by the ”Angence Nationale de la Recherche” under
the grant ANR-18-EURE-0024. The author is a Ph.D. candidate at University
Sorbonne Paris Nord.
\bibliographystyle{plain}
\bibliography{interacttfpsample}
\bigskip

\appendix
\section{Complementary information}\label{auxi}
In the Appendix section, we demonstrate the complementary information necessary to understand the content of this paper.
\begin{lemma}\label{kernelbad}
The function $\xi:\mathbb{R}\to\mathbb{R}$ denoted by
$
    \xi(x)=\left(\frac{x}{4\sqrt{2}}-\frac{1}{16 e^{2\sqrt{2}x}}\right)
$
satisfies
\begin{equation*}
    \left[-\frac{d^{2}}{dx^{2}}+U^{''}(H_{0,1}(x))\right]\xi(x)H^{'}_{0,1}(x)=H^{'}_{0,1}(x).
\end{equation*}
\end{lemma}
\begin{proof}[Proof of Lemma \ref{kernelbad}.]
Clearly, we have that $\xi^{'}(x)=\frac{1}{4\sqrt{2}}+\frac{1}{4\sqrt{2}e^{2\sqrt{2}x}},$ so using identity $H^{'}_{0,1}(x)=\sqrt{2}e^{\sqrt{2}x}\left(1+e^{2\sqrt{2}x}\right)^{{-}\frac{3}{2}},$ we obtain that
\begin{equation*}
    {-}\frac{d}{dx}\left[  \xi^{'}(x)H^{'}_{0,1}(x)^{2}\right]=H^{'}_{0,1}(x)^{2}.
\end{equation*}
In conclusion, since  $0<H^{'}_{0,1}$ and $H^{'}_{0,1}\in\ker \left({-}\frac{d^{2}}{dx^{2}}+ U^{''}(H_{0,1})\right),$ we have 
$
    \left[{-}\frac{d^{2}}{dx^{2}}+U^{''}(H_{0,1})\right]\xi(x)H^{'}_{0,1}(x)=H^{'}_{0,1}(x).
$
\end{proof}
\begin{remark}\label{formulaG}
From the identity
\begin{equation*}
    U^{''}(H_{0,1}(x))=2-24H_{0,1}(x)^{2}+30H_{0,1}(x)^{4},
\end{equation*}
we deduce that
\begin{equation*}
    \left[{-}\frac{d^{2}}{dx^{2}}+U^{''}(H_{0,1}(x))\right]e^{-\sqrt{2}x}=\left(30H_{0,1}(x)^{4}-24H_{0,1}(x)^{2}\right)e^{{-}\sqrt{2}x}.
\end{equation*}
In conclusion, Lemma \ref{kernelbad} implies that
\begin{multline*}
     \left[{-}\frac{d^{2}}{dx^{2}}+U^{''}(H_{0,1}(x))\right]\left(e^{{-}\sqrt{2}x}+8\sqrt{2}\xi(x)H^{'}_{0,1}(x)\right)=\left(30H_{0,1}(x)^{4}-24H_{0,1}(x)^{2}\right)e^{{-}\sqrt{2}x}\\{+}8\sqrt{2}H^{'}_{0,1}(x),
\end{multline*}
so,
\begin{equation*}
    -\frac{d^{2}}{dx^{2}}\mathcal{G}(x)+U^{''}(H_{0,1}(x))\mathcal{G}(x)=\left(30H_{0,1}(x)^{4}-24H_{0,1}(x)^{2}\right)e^{{-}\sqrt{2}x}+8\sqrt{2}H^{'}_{0,1}(x),
\end{equation*}
for all $x\in\mathbb{R}.$
\end{remark}
\begin{lemma}\label{schwartzinver}
In notation of Lemma \ref{firstinvert}, if $g(x)\in \mathscr{S}(\mathbb{R})$ and $\left\langle g(x),\,H^{'}_{0,1}(x)\right\rangle_{L^{2}_{x}(\mathbb{R})}=0,$ then we have that $L_{1}(g)(x) \in \mathscr{S}(\mathbb{R}).$
\end{lemma}
\begin{proof}[Proof of Lemma \ref{schwartzinver}]
\textbf{Step 1.}($f(x) \in \cap_{k\geq 1} H^{k}_{x}(\mathbb{R}).$)
Following Lemma \ref{firstinvert}, we have the existence of the unique function $f=L_{1}(g)\in H^{1}_{x}(\mathbb{R})$ such that $\left\langle f(x),\,H^{'}_{0,1}(x)\right\rangle=0$ and 
\begin{equation}\label{eqpp}
-f^{''}(x)+U^{''}(H_{0,1}(x))f(x)=g(x).    
\end{equation}
The identity \eqref{eqpp} above implies that $f\in H^{2}_{x}(\mathbb{R}).$ Moreover, since $H_{0,1} \in L^{\infty}_{x}(\mathbb{R})$ and 
\begin{equation*}
    H^{'}_{0,1}(x)=\sqrt{2}\frac{e^{\sqrt{2}x}}{\left(1+e^{2\sqrt{2}x}\right)^{\frac{3}{2}}} \in \mathscr{S}(\mathbb{R}),
\end{equation*}
we obtain that $\frac{d^{l}}{dx^{l}}U^{''}(H_{0,1}(x)) \in \mathscr{S}(\mathbb{R})$ for all natural $l\geq 1.$
So, we obtain that if $f(x) \in H^{k}_{x}(\mathbb{R})$ for $k\geq 1,$ then, since $H^{k}_{x}(\mathbb{R})$ is an algebra for $k\geq 1,\, g(x)-U^{''}(H_{0,1}(x))f(x) \in H^{k}(\mathbb{R}).$ 
\par Then, from equation \eqref{eqpp}, if $f \in H^{k}_{x}(\mathbb{R}),$ then $f^{''}(x) \in H^{k}_{x}(\mathbb{R}),$ which would imply that $f^{(k+2)}(x)$ is in $L^{2}_{x}(\mathbb{R}),$ and by elementary Fourier analysis theory or interpolation theory we would verify obtain $f^{(l)}(x)\in L^{2}_{x}(\mathbb{R})$ for any natural $l$ satisfying $0\leq l \leq k+2.$ 
In conclusion, by a standard argument of induction, we obtain that, for any natural $k,\,f(x) \in H^{k}_{x}(\mathbb{R}),$ and as a consequence $f(x) \in C^{\infty}(\mathbb{R}).$
\\ \textbf{Step 2.}($f(x) \in \mathscr{S}(\mathbb{R}).$)
 Since
$
    U^{''}(\phi)=2-24\phi^{2}+30\phi^{4},
$
we have $\lim_{x\to+\infty} U^{''}(H_{0,1}(x))=8$ and $\lim_{x\to-\infty}U^{''}(H_{0,1}(x))=2.$ From equation \eqref{eqpp}, we have the following identities
\begin{align}\label{+infty}
    -f^{''}(x)+2f(x)=g(x)+\left[2-U^{''}(H_{0,1}(x))\right]f(x), \\ \label{-infty}
    -f^{''}(x)+8f(x)=g(x)+\left[8-U^{''}(H_{0,1}(x))\right]f(x). 
\end{align}
Next, we consider a smooth cut function $\chi:\mathbb{R}\to\mathbb{R}$ satisfying $0\leq \chi\leq 1$ and
\begin{equation*}
    \chi(x)=\begin{cases}
       0\text{, if $x\leq 4,$}\\
       1\text{, if $x\geq 5.$}
    \end{cases}
\end{equation*}
 Identity \eqref{-infty} implies that
$h(x)=\chi(x)f(x)$ satisfies
\begin{equation}\label{doido1}
    -h^{''}(x)+8h(x)=\chi(x)g(x)+\left[8-U^{''}(H_{0,1}(x))\right]\chi(x)f(x)-2\dot\chi(x)\dot f(x)-\ddot\chi(x)f(x).
\end{equation}
 From the definition of $\chi,$
$\chi^{'}$ is a smooth function with compact support, so both functions $\chi^{'},\,\chi^{''}\in \mathscr{S}(\mathbb{R}).$ In conclusion, since $f\in C^{\infty}(\mathbb{R})$ from first step, we deduce that $\chi^{'}\dot f,\,\chi^{''}f \in \mathscr{S}(\mathbb{R}).$
 Also, using estimate \eqref{le2} for $k=1$
\begin{equation*}
    \left\vert H^{'}_{0,1}(x)\right\vert\lesssim\min\left(e^{\sqrt{2}x},e^{-2\sqrt{2}x}\right),
\end{equation*}
we conclude from the Fundamental theorem of calculus the following estimate
\begin{equation*}
    \left\vert8-U^{''}(H_{0,1}(x))\right\vert\lesssim e^{-2\sqrt{2}x} \text{ for all $x>1.$}  
\end{equation*}
So, $f$ being in $ C^{\infty}\left(\mathbb{R}\right),$ the definition of $\chi$ and estimate \eqref{le2} imply 
\begin{equation*}
    \left[8-U^{''}(H_{0,1}(x))\right]\chi(x)f(x) \in \mathscr{S}(\mathbb{R}).
\end{equation*}
\par In conclusion, since $f(x)\chi(x)\in H^{k}_{x}\left(\mathbb{R}\right)$ for any $k\geq 0,$ identity  \eqref{doido1} implies that $\chi(x)f(x)\in\mathscr{S}(\mathbb{R}).$ By analogy, using \eqref{+infty} and the function $h_{1}=(1-\chi)f,$ we conclude that $(1-\chi)f\in\mathscr{S}(\mathbb{R}),$ so $f\in\mathscr{S}(\mathbb{R}).$
\end{proof}

\end{document}